\documentclass[10pt,a4paper,english]{article}

\usepackage{amsfonts,amssymb,stmaryrd,amsmath,amsthm,amssymb}
\usepackage{graphicx,color}
\usepackage{multimedia}
\usepackage{float}
\usepackage{placeins}

\usepackage{mathrsfs}  

\usepackage{algorithm}

\usepackage{mathtools}

\usepackage[a4paper]{geometry}
\newtheorem{theorem}{Theorem}
\newtheorem{lemma}{Lemma}
\newtheorem{remark}{Remark}
\newtheorem{proposition}{Proposition}
\newtheorem{corollary}{Corollary}

\newcommand \bx {\boldsymbol{\mathrm{x}}}
\newcommand \bn {\boldsymbol{\mathrm{n}}}
\newcommand \bw {\boldsymbol{\mathrm{w}}}

\newcommand \bu {\boldsymbol{\mathrm{u}}}
\newcommand \bv {\boldsymbol{\mathrm{v}}}

\newcommand \bff {\boldsymbol{\mathrm{f}}}
\newcommand \bgg {\boldsymbol{\mathrm{g}}}
\newcommand \blambda {\boldsymbol{\mathrm{\lambda}}}

\newcommand \bmu {\boldsymbol{\mathrm{\mu}}}

\newcommand \bvarphi {\boldsymbol{\mathrm{\varphi}}}

\newcommand \p {\partial}

\newcommand \Q {\mathrm{Q}}
\newcommand \R {\mathbb{R}}

\renewcommand \L {\mathrm{L}}

\newcommand \VV {\mathbf{V}}
\newcommand \HH {\mathbf{H}}
\newcommand \LL {\mathbf{L}}

\newcommand \ZZ {\mathbf{Z}}
\newcommand \WW {\mathbf{W}}
\renewcommand \H {\mathrm{H}}

\newcommand \I {\mathrm{I}}

\renewcommand \d {\mathrm{d}}

\newcommand \trace {\mathrm{trace}}

\DeclareMathOperator{\divg}{div}

\begingroup\makeatletter\ifx\SetFigFont\undefined%
\gdef\SetFigFont#1#2#3#4#5{%
  \reset@font\fontsize{#1}{#2pt}%
  \fontfamily{#3}\fontseries{#4}\fontshape{#5}%
  \selectfont}%
\fi\endgroup%

\title{A fictitious domain approach for a mixed finite element method solving the two-phase Stokes problem with surface tension forces\thanks{This work is supported by the Austrian Science Fund (FWF) special research grant SFB-F32 "Mathematical Optimization and Applications in Biomedical Sciences", and the BioTechMed-Graz.}}

\author{S\'ebastien Court\thanks{Institute for Mathematics and Scientific Computing, Karl-Franzens-Universit\"{a}t, Heinrichstr. 36, 8010 Graz, Austria, email: {\tt sebastien.court@uni-graz.at}.}}

\begin{document}

\maketitle

\begin{abstract}
In this article we study a mixed finite element formulation for solving the Stokes problem with general \textcolor{black}{surface} forces that induce a jump of the normal trace of the stress tensor, on an interface that splits the domain into two subdomains. Equality of velocities is assumed at the interface. The \textcolor{black}{interface} conditions are taken into account with multipliers. A suitable Lagrangian functional is introduced for deriving a weak formulation of the problem. A specificity of this work is the consideration of the interface with a fictitious domain approach. The latter is inspired by \textcolor{black}{the XFEM approach} \textcolor{black}{in the sense that cut-off functions are used}, but it is simpler to implement \textcolor{black}{since no enrichment by singular functions is provided}. In that context, getting convergence for the dual variables defined on the interface is non-trivial. For that purpose, an augmented Lagrangian technique stabilizes the convergence of the multipliers, which is important because their value would determine the dynamics of the interface in an unsteady framework. Theoretical analysis is provided, where we show that a discrete inf-sup condition, independent of the mesh size, is satisfied for the stabilized formulation. This guarantees optimal convergence rates, that we observe with numerical tests. The capacity of the method is demonstrated with robustness tests, and with an unsteady model tested for deformations of the interface that \textcolor{black}{correspond} to ellipsoidal shapes \textcolor{black}{in dimension 2}.
\end{abstract}

\hfill  

\tableofcontents

\section{Introduction}

In the present article we are interested in developing a finite element method for solving the Stokes problem with surface tension on an immersed \textcolor{black}{interface}. \textcolor{black}{This problem concerns the simulation} of the motion of a bubble-soap for example. We consider that the bubble-soap has no thickness, and is represented by a hypersurface. Its presence inside the fluid is modeled by a Neumann-type force which generates a jump of the normal trace of the stress tensor. This force is proportional to the mean curvature of the surface. In particular, at the equilibrium, this force indicates the difference of pressures inside and outside the bubble-soap. Apart from the equilibrium, this force impacts the behavior of the surrounding fluid on both sides, and the response of the fluid is a velocity on the hypersurface, by assuming the equality of velocities at the interface. This velocity determines the evolution of the interface, and thus this is how the dynamics of the bubble-soap is coupled to its own geometry. Addressing the question of existence of weak solutions for such a kind of models can be a difficult task. For more details on the mathematical aspects, we refer to~\cite{Abels2016}.

We focus our interest on the linear Stokes problem, which constitutes the corner stone of more complex models like the Navier-Stokes equations. Inside a domain $\Omega \subset \R^d$ ($d=2$ or $3$), we consider an immersed interface $\Gamma$, that we assume to be a closed smooth oriented manifold of codimension 1 without boundary, let us say a smooth perturbation of the sphere, for the sake of simplicity. The hypersurface $\Gamma$ splits the domain $\Omega$ into two connected subset $\Omega^+$ and $\Omega^-$, as described in Figure~\ref{fig1}. We thus have $\Omega = \Omega^+ \cup \overline{\Omega^-}$. The velocity-pressure couples are denoted by $(\bu^+,p^+)$ and $(\bu^-,p^-)$ inside $\Omega^+$ and $\Omega^-$, respectively. 

\begin{figure}[!h]
	\begin{center}
		\scalebox{0.5}
		{ \input{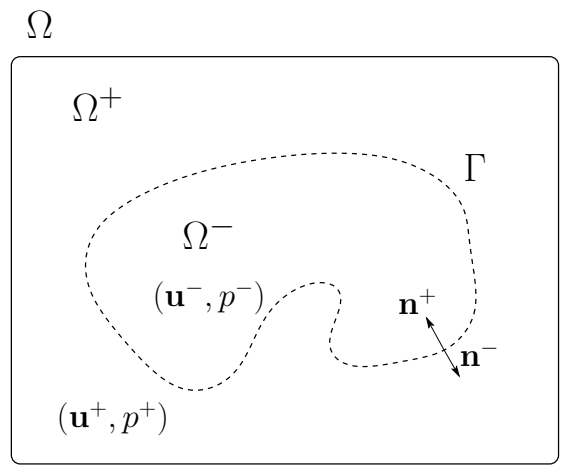} }
		\caption{A \textcolor{black}{surface} force applied on an interface, separating the fluid domain into two parts.}\label{fig1}
	\end{center}
\end{figure}
\FloatBarrier

The system we are interested in is the following:
\begin{eqnarray*}
	\left\{ \begin{array} {rcl}
		-\nu^+ \Delta \bu^+ + \nabla p^+ = \bff^+ & & \text{in } \Omega^+, \\
		-\nu^- \Delta \bu^- + \nabla p^- = \bff^- & & \text{in } \Omega^-, \\
		\divg \bu^+ = 0 & & \text{in } \Omega^+, \\
		\divg \bu^- = 0 & & \text{in } \Omega^-, \\
		\bu^+ = 0 & & \text{on } \p \Omega, \\
		\bu^+ - \bu^- = 0 & & \text{across } \Gamma, \\
		\sigma(\bu^+, p^+)\bn^+ + \sigma(\bu^-, p^-)\bn^- = \bgg & & \text{across } \Gamma.
	\end{array} \right.
\end{eqnarray*}
The notation $\sigma(\bu,p) = \nu(\nabla \bu + \nabla \bu ^T) -p\, \I$ refers to the stress tensor, where $\nu =\nu^+$ or $\nu^-$ stands for the viscosities. The data on which we focus our interest is a general \textcolor{black}{surface} force $\bgg$. We also consider volume forces $\bff^+$ and $\bff^-$ in the right-hand-side of the first two equations in the system above, but their consideration in finite element formulations does not involve any specific difficulty. We denote by $\bn^+$ and $\bn^-$ the outward unit normal on $\Gamma$ of $\Omega^+$ and $\Omega^-$, respectively. The goal of this paper is to define a robust approximation of a Poincar\'e-Steklov operator (of type {\it Neumann-to-Dirichlet}), that computes the trace of the velocity on $\Gamma$ from the surface tension force $\bgg$.

The state-of-the-art of finite element methods developed for solving two-phase flow problems with surface tension forces can be divided into different types of strategies. The first one consists of methods that adapt the mesh in function of the shape of the interface. Among them, adaptive methods were developed in~\cite{Dufour1998, Kou2014, Xie2016}. Another strategy consists in deforming the mesh in function of the deformation of the interface. A Lagrangian framework was considered in~\cite{Peric2001}. {\it Arbitrary Lagrangian Eulerian} (ALE) formulations are more famous for fluid-structure interactions models. However, we can mention~\cite{Navti1997, Kou2014, Anjos2014, Liu2017, Anjos2018} where ALE-FEM methods are developed in the context of surface tension models. Finally, the use of unfitted mesh -- which is our concern -- was considered in~\cite{Zhang2016} . \textcolor{black}{These techniques require local treatments and specific approximations of the forces on the interface, in particular when the latter is implemented with a level-set function. In this fashion, we can evoke the works~\cite{Gross2007-2, Gross2007}, where enrichment of basis functions are provided.}

\textcolor{black}{Following a fictitious domain approach, other strategies can be mentioned for capturing an interface in the context of two-phase flows. Level-set methods for multiphase flows are developed in~\cite{Engquist2002, Mesri2016, Turek2018}. Discontinuous Galerkin methods are used in~\cite{Whiteley2015, Moortgat2016}. The problem of determining th position of the interface, and tracking it constitutes the wide family of interface-capturing methods. In the context of the present work, we can cite~\cite{Ohmori1997, Devals2007}, and more recently~\cite{Owkes2013, Denner2014, Dhar2015, Park2018, Duret2018, Heinrich2018}. In our case, a level-set function helps us in practice to determine the position of the interface, by saying on which side of the interface a point of the domain is located, in the same fashion as~\cite{Mesri2016}.
}

\textcolor{black}{In the present work the focus is on the introduction of dual variables, namely multipliers for taking into account the interface conditions. This is achieved with the consideration of a judicious Lagrangian functional, from which the finite element formulation is derived. When the optimality conditions for this Lagrangian functional are satisfies, these multipliers correspond to velocities and forces on this interface, whose values are unknowns in the problem, and so their approximation is of interest for the consideration of coupled systems.}

\textcolor{black}{An other} originality of our method lies in the fact that the \textcolor{black}{interface} $\Gamma$ is taken into account with a fictitious domain approach. That means that the \textcolor{black}{interface} does not fit the mesh, and so the latter is chosen independently of the \textcolor{black}{interface} (Cartesian mesh, structured mesh...). Our approach is inspired by the eXtended Finite Element Method (XFEM) introduced by~\cite{Moes1999}. This method consists in enriching the  set of basis functions with singular functions, in order to handle variables (defined on the \textcolor{black}{interface}) whose degrees of freedom are independent of the mesh edges. See~\cite{reviewXfem} for a review of the applications of the method. Applications of XFEM to the context of two-phase flow were tested in~~\cite{Chessa2003, Chessa2003-2, Reusken2007, Reusken2008, Gross2011, Sauerland2011, Cheng2012, Liao2012, Sauerland2013, Fahsi2017}, for instance. In our case, \textcolor{black}{one of the main difference with the XFEM approach lies in the fact that} we do not provide singular functions as enrichment, but merely the trace on the \textcolor{black}{interface} of the standard basis functions (see section~\ref{sec-fict} for the details). The price-to-pay -- in comparison with XFEM -- is a lack of robustness with respect to the geometry, and a lack of quality for the convergence of the dual variables. We circumvent this drawback by performing a stabilization technique with an augmented Lagrangian \`a la Barbosa-Hughes~\cite{Barbosa1, Barbosa2}. See also~\cite{Tezduhar2003}. The present strategy was first introduced in~\cite{Renard2009}, and next adapted to fluid mechanics in~\cite{Court2014, Court2015} in the context of Fluid-Structure Interactions \textcolor{black}{based on Dirichlet boundary conditions}. This approach has also shown its capacities in~\cite{Court2016} where complex non-planar cracks in 3D were taken into account in the context of Geophysics. In this fashion, let us also mention that Nitsche type methods were developed for solving this kind of problems, in several works~\cite{Hansbo2002, Hansbo2005, Hansbo2005-2, Hansbo2005-3, Hansbo2010, Hansbo2011, Hansbo2014}. This family of methods does not require the introduction of Lagrange multipliers for the boundary conditions, whose consideration can be made with overlapping meshes, for instance.

\textcolor{black}{In the present work, Taylor-Hood elements will be used for the choice of pairs between primal and dual variables. We will then focus our interest in the use of structured meshes, since unstructured meshes can lead to instabilities in that case (see~\cite{Case2011, Gonzalez2015} for instance)}.

By developing such a method, our goal is to have a tool which enables us to perform unsteady simulations involving a moving interface with \textcolor{black}{surface} forces, in complex situations where sparing computation time is crucial. The purpose can be the study of the movement of a bubble-soap, or the simulation of the stabilization of this latter, by the use of an electric field as a control function for instance (see~\cite{Sato1998}), that acts on the interface as a surface tension type force. The interest of our method lies in the fact that, for updating the geometry between two time-steps, we only need to update a number of objects which is of the same range as the number of degrees of freedom chosen for describing the interface (see section~\ref{sec-smartupdate} for more details). 

We illustrate the capacity of the method and our underlying motivation by performing simulations for an unsteady coupled model. The initial geometric configuration, given by the \textcolor{black}{interface} $\Gamma$ at time $t=0$, generates a surface tension force $\bgg = -\mu \kappa \bn^-$ where $\mu$ is a coefficient and $\kappa$ is the mean curvature of the \textcolor{black}{interface}. While solving the problem for this force, we obtain the trace of the velocity on the \textcolor{black}{interface}. This velocity enables us to update the \textcolor{black}{interface} for the next time step. This is a so-called {\it partitioned} method. \textcolor{black}{In this fashion, let us mention~\cite{Turek2004}, which treats of an other type of coupled problems}.\\

The plan is organized as follows: In section~\ref{sec-setting} we set the problem and its variational formulation, by introducing a judicious Lagrangian functional. The discretization is described in section~\ref{sec-fict}: \textcolor{black}{the} fictitious domain method is explained in section~\ref{subsec-fict}, and the theoretical analysis is provided in section~\ref{subsec-theor} (without stabilization) and in section~\ref{subsec-theorstab} (with stabilization). Explanations about the practical implementation are given in section~\ref{sec-impl}. Numerical tests are provided in section~\ref{sec-numtests}. Convergence and accuracy are tested with and without the stabilization technique in section~\ref{subsec-cvstab0} and section~\ref{subsec-cvstab1}, respectively. In particular, robustness with respect to the geometry is demonstrated in section~\ref{subsec-cvrobust}. The unsteady simulations are given in section~\ref{sec-unsteady}. Conclusions are given in section~\ref{sec-conclusion}. \textcolor{black}{The Appendix is devoted to the proofs of technical results.}

\paragraph{Notation.} The symbol $\pm$ will be used for indicating that we consider both symbols $+$ and $-$, for the sake of concision. The jump of a variable $\varphi$ across $\Gamma$ will be denoted by $\left[ \varphi \right]$, equal to $\varphi^+ - \varphi^-$. As unit normal of reference, we denote $\bn = \bn^-$, and so $\bn^+ = -\bn$. We denote by $\sigma(\bv,q) = 2\nu \varepsilon(\bv) - q\, \I$ the stress tensor, where $\varepsilon(\bv) = \frac{1}{2}(\nabla \bv + \nabla \bv^T)$ is the symmetric Cauchy stress tensor and $\I$ is the identity matrix of $\R^{d\times d}$. When $\divg \bv = 0$, we recall that $-\divg \sigma(\bv,q) = -\nu \Delta \bv + \nabla q$. We denote by $|\cdot |$ the Euclidean norm of $\R^d$ or $\R^{d\times d}$, and by $A:B = \trace(A^T B)$ the scalar product in $\R^{d\times d} \times \R^{d\times d}$.

\section{Setting of the problem} \label{sec-setting}
Let us consider the following system:
\begin{eqnarray} \label{sysjump} \label{mainsys} \label{syscont}
\left\{ \begin{array} {rcl}
	-\divg \sigma^\pm(\bu^\pm,p^\pm)  = \bff^\pm & & \text{in } \Omega^\pm, \label{eqstokes}\\
	\divg \bu^\pm = 0 & & \text{in } \Omega^\pm, \\
	\bu^+ = 0 & & \text{on } \p \Omega, \\
	\left[ \bu \right] = 0 & & \text{across } \Gamma, \label{eqjump} \\
	\left[\sigma(\bu, p)\right]\bn = \bgg & & \text{across } \Gamma. \label{eqjumpg}
	\end{array} \right.
\end{eqnarray}
The notation $\sigma^\pm(\bu^\pm,p^\pm) := 2\nu^\pm \varepsilon(\bu^\pm) - p^\pm\I$ is introduced for considering  different (constant) viscosities. Assuming that $\Gamma$ is closed, we define the following function spaces:
\begin{eqnarray*}
	\begin{array} {lcl}
		\mathbf{V}^{+} = \displaystyle \left\{\bv \in \mathbf{H}^1(\Omega^+)\mid \ \bv_{|\p \Omega} = 0\right\}, & &
		\mathbf{V}^- = \mathbf{H}^1(\Omega^-)
		, \\[5pt]
		\Q^{\pm} = \displaystyle \left\{q \in \LL^2(\Omega^\pm) \mid \ \int_{\Omega^\pm} q \, \d \Omega^\pm = 0 \right\}, & & \\[10pt]
		\WW = \HH^{-1/2}(\Gamma), & & \ZZ = \WW' = \HH^{1/2}(\Gamma).
	\end{array}
\end{eqnarray*}
\textcolor{black}{We will denote by $\langle \, \cdot \, ; \cdot \, \rangle_{\WW;\WW'}$ the duality bracket between $\WW'$ and its dual space $\WW'' \equiv \WW$.} The equality of velocities in the third equation of~\eqref{mainsys} suggests the existence of a function $\Phi$ such that $\bu^+ = \Phi$ and $\bu^- = \Phi$. We take into account these two equalities by introducing two multipliers denoted by $\blambda^\pm$. More specifically, we look for a weak solution of system~\eqref{sysjump} as a critical point of the following Lagrangian functional:
\begin{eqnarray}
\mathscr{L}_0(\bu^+, p^+, \blambda^+, \bu^-, p^-, \blambda^-, \Phi) & = &
2\nu \int_{\Omega^+} |\varepsilon(\bu^+)|^2\d \Omega^+ + 2\nu \int_{\Omega^-} |\varepsilon(\bu^-)|^2\d \Omega^- \nonumber\\
& & - \int_{\Omega^+} \bff^+\cdot \bu^+\d \Omega^+ - \int_{\Omega^-} \bff^-\cdot \bu^-\d \Omega^- \nonumber\\
& & - \int_{\Omega^+} p^+\divg \bu^+ \d \Omega^+ - \int_{\Omega^-} p^-\divg \bu^- \d \Omega^- \nonumber\\
& & -  \langle \blambda^+ ; \bu^+ - \Phi \rangle_{\mathbf{W};\mathbf{W}'}
-  \langle \blambda^- ; \bu^- - \Phi\rangle_{\mathbf{W};\mathbf{W}'}
- \langle \bgg ; \Phi \rangle_{\mathbf{W};\mathbf{W}'}. \nonumber
\end{eqnarray}
\textcolor{black}{When the derivatives of $\mathscr{L}_0$ with respect to $\bv^\pm$ vanish, by integration by parts we obtain the first equation of~\eqref{eqstokes} and also $\blambda^{\pm}  =  \sigma^\pm(\bu^{\pm},p^{\pm})\bn^{\pm}$. Next, the sensitivity of $\mathscr{L}_0$ with respect to $\blambda^\pm$ and $\Phi$ implies to the transmission conditions of~\eqref{eqstokes}.}
For the sake of concision, we will denote
\begin{eqnarray*}
\mathfrak{u}  =  (\bu^+,p^+, \blambda^+,\bu^-, p^-, \blambda^-, \Phi)
& \text{ and } &
\mathfrak{v}  =  (\bv^+,q^+,\bmu^+,\bv^-,q^-, \bmu^-, \varphi).
\end{eqnarray*} 
The first-order optimality conditions satisfied by a saddle-point of $\mathscr{L}_0$ then yield the following variational formulation:
\begin{eqnarray}
	& & \text{Find 
$\mathfrak{u} \in \VV^+  \times \Q^+  \times \mathbf{W}\times \VV^-\times \Q^- \times \mathbf{W} \times \ZZ$ such that} \nonumber \\
	& & \left\{ \begin{array} {lcl}
		\mathcal{A}_0^\pm(\mathfrak{u};\bv) =
		\mathcal{F}^\pm(\bv) & & 
		\forall \bv \in \VV^\pm, \\[5pt]
		\mathcal{B}_0^\pm(\mathfrak{u};q) = 0 & & 
		\forall q \in \Q^\pm, \\[5pt]
		\mathcal{C}_0^\pm(\mathfrak{u};\bmu) = 0 & & 
		\forall \bmu \in \WW^\pm, \\[5pt]
		\mathcal{D}_0^\pm(\mathfrak{u};\bvarphi) = 
		\mathcal{G}(\bvarphi) & & 
		\forall \bvarphi \in \ZZ .
	\end{array} \right.
\end{eqnarray}
In this formulation we introduced the following bilinear forms
\begin{eqnarray*}
\mathcal{A}_0^\pm(\mathfrak{u};\bv)  =  
2\nu \int_{\Omega^{\pm}} \sigma^\pm(\bu^{\pm},p^{\pm}): \varepsilon(\bv) \, \d \Omega^{\pm} 
 - \langle \blambda^{\pm} \; \bv \rangle_{\mathbf{W};\mathbf{W}'}, & &
\mathcal{F}^\pm(\bv)  =  
\int_{\Omega^\pm} \bff^\pm \cdot \bv \, \d \Omega^\pm, \\
\mathcal{B}_0^\pm(\mathfrak{u};q)  =  -\int_{\Omega^\pm} q^\pm \divg \bu^\pm \d \Omega^\pm,
& & \\[5pt]
\mathcal{C}_0^\pm(\mathfrak{u};\bmu)  =  
-\langle \bmu ; \bu^\pm - \Phi \rangle_{\mathbf{W};\mathbf{W}'}, & & \\
\mathcal{D}_0^\pm(\mathfrak{u};\bvarphi) =  
\langle \blambda^{+} + \blambda^- ; \bvarphi \rangle_{\mathbf{W}; \mathbf{W}'},  & &
\mathcal{G}(\bvarphi)  =  \langle \bgg ; \bvarphi \rangle_{\mathbf{W}; \mathbf{W}'}.
\end{eqnarray*}

\section{Discrete formulation} \label{sec-fict}
In the rest of the paper, we will consider a Cartesian mesh, and we will denote the mesh parameter by $h  =  \max_{T\in \mathcal{T}_h} h_T$, where $h_T$ is the diameter of a triangle $T$, and $\mathcal{T}_h$ is the set of the triangles of the mesh.

\subsection{The fictitious domain method} \label{subsec-fict}
We first consider global functions on the same whole domain $\Omega$, that we discretize with a structured mesh. On this mesh we define discrete finite element spaces, $\tilde{\VV}_h \subset \HH^1(\Omega)$, $\tilde{Q}_h \subset \L^2_0(\Omega)$, $\tilde{\WW}_h \subset \LL^2(\Omega)$ and $\tilde{\ZZ}_h \subset \HH^1(\Omega)$. We set
\begin{eqnarray*}
\tilde{\VV}_h  =  \left\{\bv_h \in \mathscr{C}(\overline{\Omega}) \mid \ {\bv_h}_{|\p \Omega} = 0, \ {\bv_h}_{\left| T\right.} \in P(T), \ \forall T \in \mathcal{T}_h\right\},
\end{eqnarray*}
where $P(T)$ denotes a finite dimensional space of smooth functions that contains polynomial functions of degree $k \geq 1$ on a triangle $T$, taken in the set $\mathcal{T}_h$ of triangles of the mesh. We refer to~\cite{Ern} for details. 
The fictitious finite element spaces are defined as follows:
\begin{eqnarray*}
	\VV^\pm_h := {\tilde{\mathbf{V}}{}_h}_{\left| \Omega^{\pm} \right.}, \quad 
	\Q^\pm_h := {\tilde{\Q}{}_h}_{\left| \Omega^{\pm} \right.}, \quad 
	\WW_h := {\tilde{\WW}{}_h}_{\left| \Gamma \right.}, \quad
	\ZZ_h := {\tilde{\ZZ}{}_h}_{\left| \Gamma \right.}.
\end{eqnarray*}
Note that these spaces are the intuitive discretizations of spaces $\VV^\pm$, $\Q^\pm$, $\WW$ and $\ZZ$, respectively. The corresponding selection of degrees of freedom is explained in Figure~\ref{fig-fictistyle}. \textcolor{black}{For this task, we need to know whether a node in one side of the interface are the other, and this can be realized in practice with the use of a level-set function. This object represents the interface, and at this stage is used only for implementation purposes.} In particular, the points of intersection of the edges of the mesh with the \textcolor{black}{interface} (a circle in Figure~\ref{fig-fictistyle}) are determined, in order to define an approximation of the \textcolor{black}{level-set} (by piecewise polynomial functions), as well as degrees of freedom for the multipliers. This technique of discretization is inspired by XFEM, with the difference that here we do not provide enrichment of the standard basis elements with specific singular functions, we only take into account the standard basis functions multiplied by the Heaviside functions ($H(\mathrm{x}) =1$ when $\mathrm{x} \in \Omega^\pm$, $H(\mathrm{x}) = 0$ otherwise). The resulting products appear in the integrals of the variational formulation, during the assembly procedure (see section~\ref{sec-smartupdate}).

\begin{minipage}{\linewidth}
\begin{center}
	\includegraphics[trim = 10cm 4.5cm 10cm 4.5cm, clip, scale=0.22]{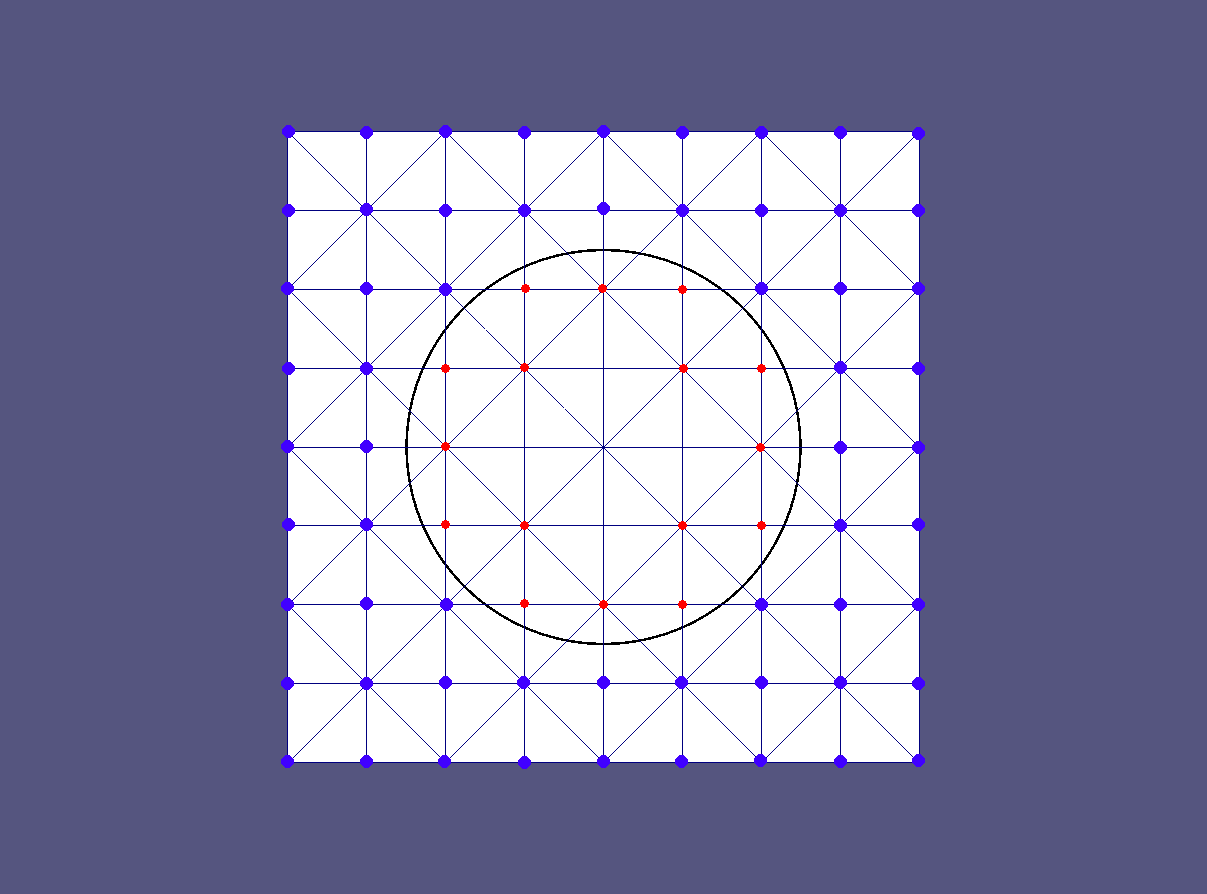}
	\hspace*{5pt}
	\includegraphics[trim = 10cm 4.5cm 10cm 4.5cm, clip, scale=0.22]{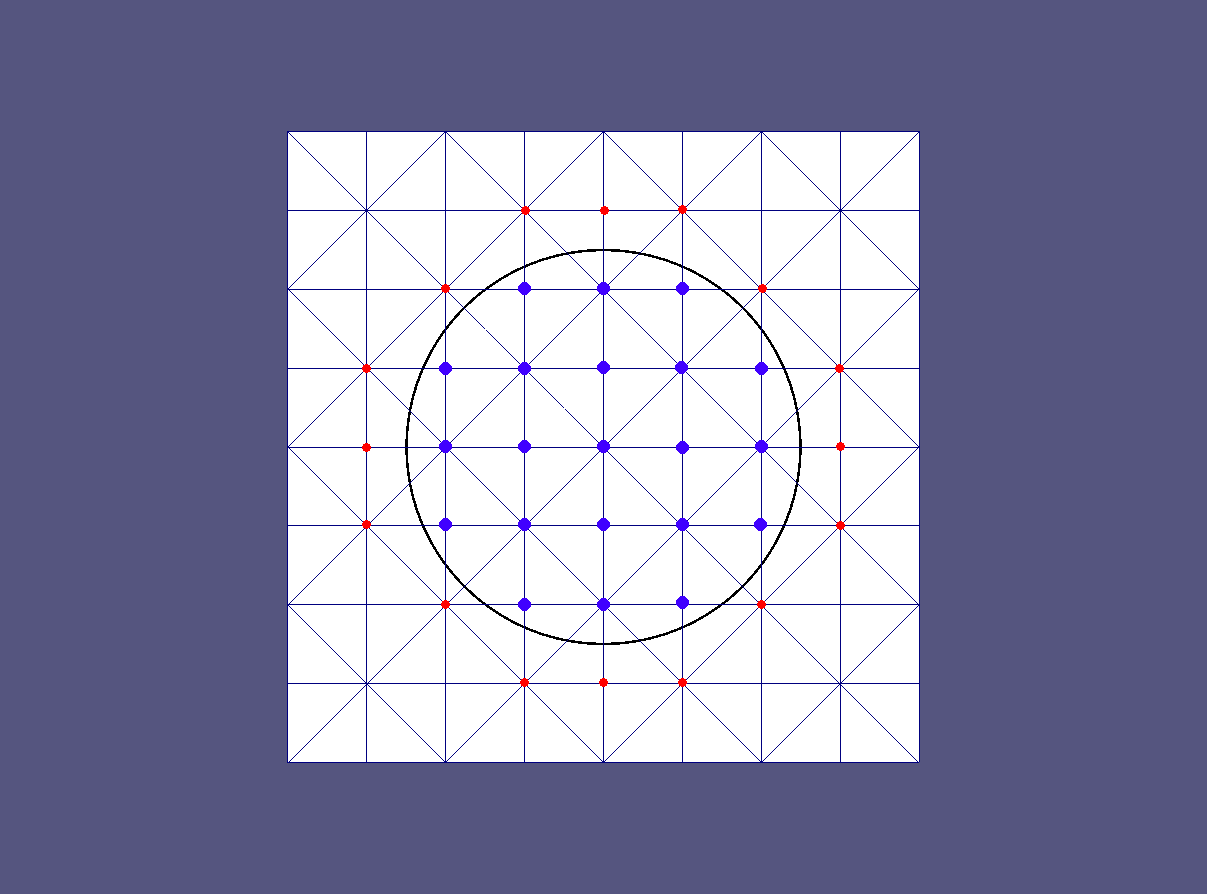}
	\hspace*{5pt}
	\includegraphics[trim = 10cm 4.5cm 10cm 4.5cm, clip, scale=0.22]{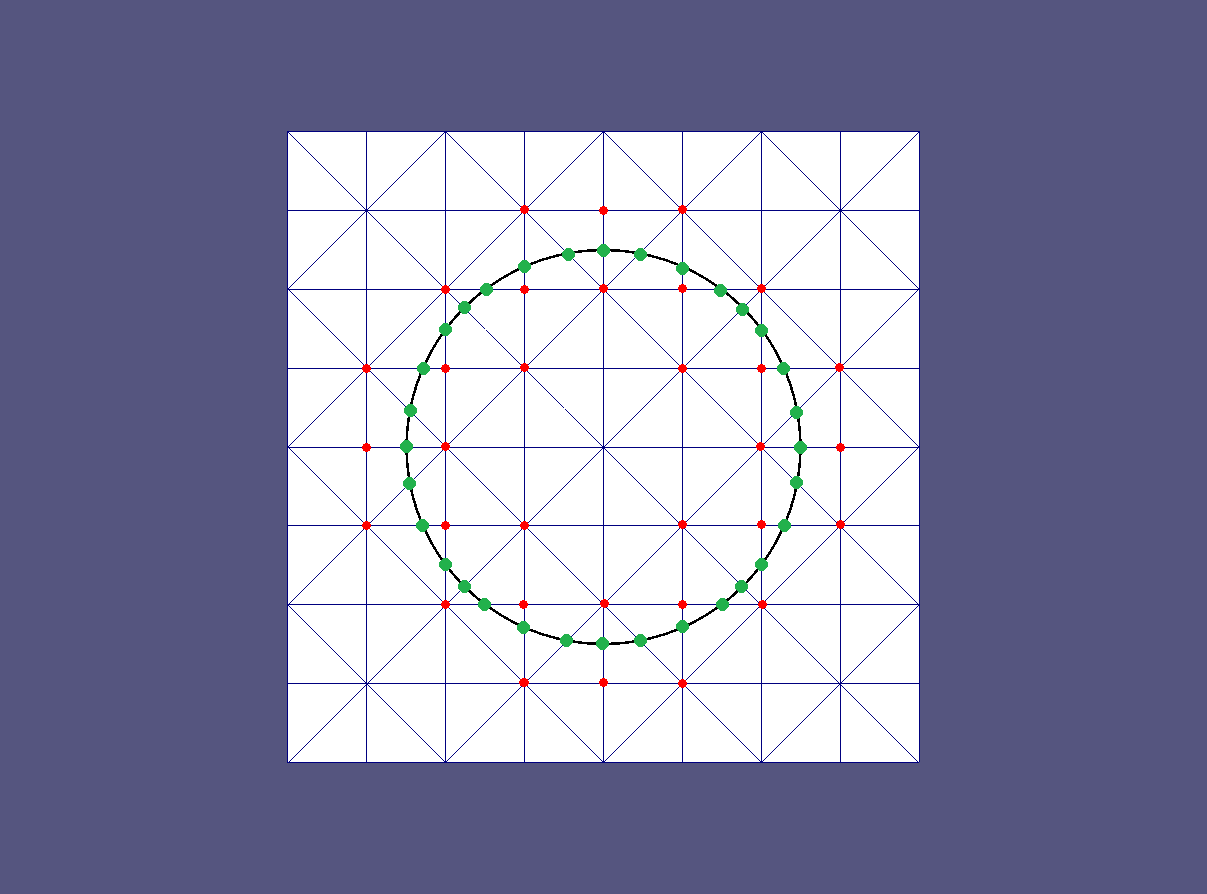}
	\vspace*{-10pt}
	\begin{figure}[H]
		\caption{Degrees of freedom used for each space: $\VV_h^+$ and $\Q_h^+$ (left), $\VV_h^-$ and $\Q_h^-$ (center), $\WW_h$ and $\ZZ_h$ (right). The blue base nodes are chosen for the degrees of freedom of the \textcolor{black}{fluid variables} ($\bu^\pm$ and $p^\pm$), while the red ones are just kept for localizing the \textcolor{black}{interface}, and cutting the standard basis functions. The green nodes are determined for defining the basis functions of the multipliers ($\blambda^\pm$ and $\Phi$) on the \textcolor{black}{interface}.}\hfill \\
		\label{fig-fictistyle}
	\end{figure}
\end{center}
\end{minipage}

Then Problem~\eqref{mainsys} is approximated as follows:
\begin{eqnarray}
& & \text{Find } (\bu_h^\pm,p_h^\pm, \blambda_h^\pm, \Phi_h) \in \VV_h^\pm \times \Q_h^\pm \times \WW_h \times \ZZ_h \text{ such that } \nonumber \\
& & \left\{ \begin{array} {rcl}
a_0^\pm(\bu_h^\pm,\bv_h^\pm) + b_0^\pm(\bv_h^\pm,p_h^\pm) + c_0(\bv_h^\pm,\blambda_h^\pm) = \mathcal{F}^\pm(\bv^\pm_h) & & 
\forall \bv_h^\pm \in \mathbf{V}_h^\pm, \\[5pt]
b_0^\pm(\bu_h^\pm,q_h^\pm) = 0 & &  \forall q_h^\pm \in \Q_h^\pm, \\[5pt]
c_0(\bu_h^\pm, \bmu_h^\pm) - c_0(\Phi_h, \bmu_h^\pm) = 0 & & \forall  \bmu_h^\pm \in \mathbf{W}_h, \\[5pt]
c_0(\bvarphi_h, \blambda_h^+ + \blambda_h^-) = \mathcal{G}(\bvarphi_h) & & 
\forall \bvarphi_h \in \mathbf{Z}_h,
\end{array}\right. \label{mainsysapprox} \label{sysdisc}
\end{eqnarray}
where we denote
\begin{eqnarray*}
a_0^\pm(\bu,\bv) = 2\nu^\pm\int_{\Omega^\pm} \varepsilon(\bu):\varepsilon(\bv) \, \d \Omega^\pm, 
\quad b_0^\pm(\bu, q) = -\int_{\Omega^\pm} q\divg \bu \, \d \Omega^\pm, 
\quad c_0(\bvarphi, \blambda)  =  - \int_{\Gamma} \bvarphi \cdot \blambda \, \d \Gamma.
\end{eqnarray*}

Note that the duality bracket $\langle \, \cdot \, ; \cdot \, \rangle_{\WW; \WW'}$, with $\WW = \HH^{-1/2}(\Gamma)$, has been replaced by the inner product of $\LL^2(\Gamma)$. \textcolor{black}{The aim is to avoid to define an approximation of} the Laplace–Beltrami operator on $\Gamma$, which is a non-trivial task because of the fictitious domain approach (see~\cite{Massing2017} for instance). Under stronger regularity assumptions for the data $\bff^\pm$ and $\bgg$ \textcolor{black}{(for instance $\bff^\pm \in \LL^2(\Omega^\pm)$ and $\bgg \in \LL^2(\Gamma)$)}, we can reasonably consider this simplification, and thus we now set
\begin{eqnarray*}
\WW \equiv \WW' \equiv \ZZ = \LL^2(\Gamma), 
\quad \WW_h \subset \LL^2(\Gamma),
\quad \ZZ_h \subset \LL^2(\Gamma).
\end{eqnarray*}
However, for the mathematical analysis, we keep the abstract formalism involving the notation $\WW$ and $\WW'$.

\subsection{Theoretical convergence} \label{subsec-theor}
For the mathematical analysis, we make the following assumptions:
\begin{itemize}
\item[$(\mathbf{H1})$:] There exists a constant\footnote{Throughout the paper, $C$ denotes a generic positive constant independent of the mesh size $h$.} $C >0$ independent of $h$ such that
\begin{eqnarray*}
\inf_{q_h^\pm \in \Q^\pm_h\setminus\{0\}} \sup_{\bv_h^\pm \in \VV^\pm_{0,h}\setminus\{0\}}
\frac{b_0^\pm(\bv^\pm_h,q^\pm_h)}{\| \bv^\pm_h \|_{\VV_h^\pm} \| q_h^\pm \|} & \geq & C,
\end{eqnarray*}
where we denote $\VV_{0,h}^\pm := \left\{\bv^\pm_h \in \VV_h^\pm \mid c_0(\bv_h^\pm, \bmu_h^\pm) = 0 \ \forall \bmu_h^\pm \in \WW_h \right\}$.
\item[$(\mathbf{H2})$:] \quad If $\overline{\bmu}_h \in \mathbf{W}_h$ satisfies $c^+_0(\bv_h^+, \overline{\bmu}_h) = 0$ for all $\bv_h^+ \in \mathbf{V}_h^{+}$, or $c^-_0(\bv_h^-, \overline{\bmu}_h) = 0$ for all $\bv_h^- \in \mathbf{V}_h^{-}$, then $\overline{\bmu}_h = 0$.
\end{itemize}

Assumption~$(\mathbf{H1})$ is a discrete inf-sup condition for the couple {\it velocity}/{\it pressure}. It implies in particular the following property: If $\overline{q}^\pm_h \in \Q^\pm_h$ satisfies $b^\pm_0(\bv_h^+, \overline{q}^\pm_h) = 0$ for all $\bv_h^+ \in \VV_{0,h}^\pm$, then $\overline{q}^\pm_h = 0$. Assumption~$(\mathbf{H2})$ is weaker than an inf-sup condition for the couple $\bu^\pm / \blambda^\pm$. It demands only than the spaces $\VV^\pm_h$ are rich enough with respect to the space $\WW_h$.\\

Now we define the space
\begin{eqnarray*}
	\mathbb{V}^0_h & = & \left\{
	(\bv_h^+,\bv_h^-) \in \mathbf{V}_h^+ \times \mathbf{V}_h^- | \quad c(\bv_h^+-\bv_h^-, \bmu_h)=0 \quad \forall \bmu_h \in \mathbf{W}_h \right\}.
\end{eqnarray*}

\begin{lemma} \label{lemma-coer}
The bilinear form
\begin{eqnarray*}
	((\bu^+,\bv^+),(\bu^-,\bv^-)) & \mapsto & 
	a_0^+(\bu^+,\bv^+) + a_0^-(\bu^-,\bv^-)
\end{eqnarray*}
is uniformly $\mathbb{V}_h^0$-elliptic, in the sense that there exists a constant $C>0$ independent of $h$ such that
\begin{eqnarray*}
a_0^{+}(\bv_h^+,\bv_h^+) + a_0^{-}(\bv_h^-,\bv_h^-) \geq 
C \left(\|\bv_h\|^2_{\mathbf{V}^{+}_h} + \|\bv_h\|^2_{\mathbf{V}^{-}_h} \right), \quad 
\forall (\bv_h^+,\bv_h^-) \in \mathbb{V}_h^0.
\end{eqnarray*}
\end{lemma}

\begin{proof}
This result is an application of the Petree-Tartar lemma. From Korn's inequality we have
\begin{eqnarray*}
\| \bv_h^+\|_{\VV^+}^2 + \| \bv_h^-\|_{\VV^-}^2 & \leq &
a_0^{+}(\bv_h^+,\bv_h^+) + a_0^{-}(\bv_h^-,\bv_h^-) + 
\| \bv_h^+\|_{\LL^2(\Omega^+)}^2 + \| \bv_h^-\|_{\LL^2(\Omega^-)}^2.
\end{eqnarray*}
Since from the Rellich-Kondrachov theorem the embeddings $\VV^\pm \hookrightarrow \LL^2(\Omega^\pm)$ are compact, we just have to verify that \textcolor{black}{$a_0^{+}(\bv_h^+,\bv_h^+) + a_0^{-}(\bv_h^-,\bv_h^-) = 0  \Rightarrow  (\bv^+,\bv^-) = 0$} in $\mathbb{V}^0_h$. The equality of the left-hand-side of this assertion is equivalent to $a_0^{\pm}(\bv_h^\pm,\bv_h^\pm) = 0$, and so in particular $\varepsilon(\bv_h^\pm) = 0$ in $\Omega^\pm$. From~\cite[page~18]{Temam}, the functions $\bv_h^\pm$ reduce to affine forms. The function $\bv_h^+$ is actually $0$, because ${\bv_h^+}_{\mid \p \Omega } =0$. Since $(\bv_h^+,\bv_h^-) \in \mathbb{V}^0_h$, we deduce that ${\bv^-_h}_{\mid \Gamma} = 0$. Indeed, it is easy to see that the intersection of  $\mathbb{V}^0_h$ with the space of affine functions is actually contained into the space $\left\{
(\bv_h^+,\bv_h^-) \in \mathbf{V}_h^+ \times \mathbf{V}_h^- | \ {\bv_h^+}_{\mid \Gamma} = {\bv_h^-}_{\mid \Gamma} \right\}$.
Thus $\bv_h^- = 0$ in $\Omega^-$, which completes the proof.
\end{proof}

\begin{proposition}
Assume that assumptions~$(\mathbf{H1})-(\mathbf{H2})$ hold. Then system~\eqref{sysdisc} admits a unique solution that we denote by $(\bu_h^+,p_h^+,\blambda_h^+, \bu_h^-,p_h^-,\blambda_h^-,\Phi_h)$.
\end{proposition}

\begin{proof}
Since system~\eqref{sysdisc} is linear and of finite dimension, existence is equivalent to uniqueness. Let us prove uniqueness by showing that $(\bu_h^\pm,p_h^\pm, \blambda_h^\pm, \Phi_h) = 0$ when $\mathcal{F}^\pm \equiv 0$ and $\mathcal{G} \equiv  0$. In that case, taking $\bv^\pm_h = \bu_h^\pm$ in the first equation of~\eqref{sysdisc}, combined with its second equation taken with $q_h^\pm = p_h^\pm$, yields $a_0^\pm(\bu^\pm_h,\bu_h^\pm) + c_0(\bu^\pm_h, \blambda_h^\pm) = 0$. Using the third equation then leads us to $a_0^\pm(\bu^\pm_h,\bu_h^\pm) + c_0(\Phi_h, \blambda_h^\pm) = 0$. By summing these two identities, and by using the fourth equation with $\bvarphi_h = \Phi_h$, we obtain
\begin{eqnarray*}
a_0^+(\bu^+_h,\bu_h^+) + a_0^-(\bu^-_h,\bu_h^-) = 0,
\end{eqnarray*}
and so $\bu_h^\pm = 0$ in $\VV^\pm_h$ by Lemma~\ref{lemma-coer}. Next, still in the first equation, we get $b_0^\pm(\bv_h^\pm, p_h^\pm) = 0$ for all $\bv^\pm_h \in \VV^\pm_{0,h}$, and then $p_h^\pm = 0$ from assumption~$(\mathbf{H1})$. Then it remains only $c_0(\bv_h^\pm, \blambda_h^\pm) =0$, valid for all $\bv_h^\pm \in \VV_h^\pm$, and this yields $\blambda_h^\pm =0$ from assumption~$(\mathbf{H2})$. Finally, we obtain $\Phi_h=0$ by using assumption~$(\mathbf{H2})$ in the third equation.
\end{proof}

\begin{proposition} \label{prop-cvnaive}
Assume that assumptions~$(\mathbf{H1})-(\mathbf{H2})$ hold. Denote by $(\bu^+,p^+,\blambda^+, \bu^-,p^-,\blambda^-, \Phi)$ and \\$(\bu_h^+,p_h^+,\blambda_h^+, \bu_h^-,p_h^-,\blambda_h^-,\Phi_h)$ the respective solutions of system~\eqref{syscont} and system~\eqref{sysdisc}. Then
\begin{eqnarray}
& & \|\bu^+-\bu_h^+\|_{\mathbf{V}^+} + \|p^+-p_h^+\|_{\L^2(\Omega^+)} +
\|\bu^--\bu_h^-\|_{\mathbf{V}^-} + \|p^--p_h^-\|_{\L^2(\Omega^-)} \nonumber\\
& & \leq C\left( \inf_{(\bv_h^+,\bv_h^-)\in \mathbb{V}^0_h} 
\left(\|\bu^+-\bv_h^+ \|_{\VV^+} + \|\bu^--\bv_h^- \|_{\VV^-}\right) +
\inf_{q_h^+\in \Q_h^+}\|p^+-q_h^+\|_{\L^2(\Omega^+)} + \inf_{q_h^-\in \Q_h^-}\|p^--q_h^-\|_{\L^2(\Omega^-)} \right. \nonumber \\
& & \left. \qquad + \inf_{(\bmu_h^+,\bmu_h^-) \in \WW_h \times \WW_h} \left(\|\blambda^+ - \bmu_h^+\|_{\WW} ; 
\|\blambda^- - \bmu_h^-\|_{\WW}\right) \right), \label{est-naive}
\end{eqnarray}
where the constant $C>0$ is independent of $h$.
\end{proposition}

\textcolor{black}{The proof of Proposition~\ref{prop-cvnaive} is given in section~\ref{App-A1}. It} provides us an estimate for the velocities and the pressures. We have no such estimate for the multipliers $\Phi$ and $\blambda^\pm$. Note that in the right-hand-side of estimate~\eqref{est-naive}, no term involving the variable $\Phi$ appears. The accuracy on the velocities and the pressures is then not conditioned by the approximation of the variable $\Phi$. We can have a good approximation of $\Phi$ without having necessarily a good approximation on the other variables.

\paragraph{On the limitation of the order of convergence.}

Besides the lack of information on the convergence for the dual variables, let us mention a theoretical result that limits the rate of convergence for the velocities and the pressures.

\begin{proposition} \label{prop-limit}
Assume that assumptions $(\mathbf{H1})-(\mathbf{H2})$ hold. With the notation of Proposition~\ref{prop-cvnaive}, assume that $\bu^\pm \in \HH^{1+d/2+\eta}(\Omega^\pm) \cap \VV^\pm$ fro some $\eta >0$, and that 
\begin{eqnarray*}
	\inf_{\bmu_h \in \WW_h}\| \blambda^\pm - \bmu_h\| & \leq & Ch^\delta
\end{eqnarray*}
for some $\delta \geq 1/2$. Then
\begin{eqnarray}
\| \bu^\pm - \bu^\pm_h \|_{\VV^\pm} 
+ \| p^\pm - p_h^\pm \|_{\Q^\pm}
& \leq & 
Ch^{1/2}. \label{est-limit}
\end{eqnarray}
\end{proposition}

For the sake of concision, we do not provide the proof of this result, since the proofs given in~\cite{Renard2009} (Proposition~3) and~\cite{Court2014} (Proposition~3) can be straightforwardly transposed to our case. In the context of fictitious domain approaches, it is classical to observe theoretically this kind of limitation on the order of convergence. Actually, for our approach, in view of the numerical tests presented in section~\ref{subsec-cvstab0}, estimate~\eqref{est-limit} seems to be not sharp enough. In practice, it is tricky to find a numerical test for which this rate is observed. Such a test was provided in~\cite{Renard2009} for the Poisson problem in a test performed with a very specific configuration, and simply not found in~\cite{Court2014} for a Stokes problem. A more advanced theoretical analysis of problem~\eqref{sysdisc} is not our main concern, and so we do not comment further this point. In the following subsection, we modify the variational formulation in order to obtain an optimal convergence result that concerns all the variables.

\subsection{Stabilization technique} \label{subsec-theorstab}
The Lagrangian $\mathscr{L}$ is augmented with quadratic terms as follows:
\begin{eqnarray*}
\mathscr{L}(\bu^+, p^+, \blambda^+, \bu^-,p^-, \blambda^-, \Phi) & = &
\mathscr{L}_0(\bu^+, p^+, \blambda^+, \bu^-,p^-, \blambda^-, \Phi) \\
& & + \frac{\alpha_0}{2} \left( \| \Phi - \bu^+ \|^2_{\ZZ} + \| \Phi - \bu^- \|^2_{\ZZ} \right) \\
& & - \frac{\gamma}{2} \left( \| \sigma(\bu^+,p^+)\bn^+ -\blambda^+ \|^2_{\LL^2(\Gamma)} + 
\| \sigma(\bu^+,p^+)\bn^+ -\blambda^- \|_{\LL^2(\Gamma)}^2 \right) .
\end{eqnarray*}
The additional terms constitute the so-called stabilization technique. \textcolor{black}{The first terms, proportional to the coefficient $\alpha_0 >0$, are} introduced in order to enforce the convergence for the variable $\Phi$. The other terms, those which are proportional to the coefficient $\gamma>0$, are added in order to enforce the convergence for the multipliers $\blambda^\pm$ towards the normal traces of the stress tensor. \textcolor{black}{For practical purpose,} it is convenient to choose the coefficient $\gamma$ proportional to the mesh size:
\begin{eqnarray*}
\gamma = \gamma_0h, & & \text{where $\gamma_0 >0$ is independent of $h$.}
\end{eqnarray*}
The first-order derivatives of $\mathscr{L}$ lead to the following stabilized formulation:
\begin{eqnarray}
& & \text{Find 
	$\mathfrak{u} \in \VV^+  \times \Q^+  \times \mathbf{W} \times \VV^-\times \Q^- \times \mathbf{W} \times \ZZ$ such that} \nonumber \\
& & \left\{ \begin{array} {rcl}
\mathcal{A}^\pm(\mathfrak{u};\bv) =
\mathcal{F}(\bv) & & 
\forall \bv \in \VV^\pm, \\[5pt]
\mathcal{B}^\pm(\mathfrak{u};q) = 0 & & 
\forall q \in \Q^\pm, \\[5pt]
\mathcal{C}^\pm(\mathfrak{u};\bmu) = 0 & & 
\forall \bmu \in \WW^\pm, \\[5pt]
\mathcal{D}^\pm(\mathfrak{u};\bvarphi) = 
\mathcal{G}(\bvarphi) & & 
\forall \bvarphi \in \ZZ ,
\end{array} \right.
\end{eqnarray}
where
\begin{eqnarray*}
	\mathcal{A}^\pm(\mathfrak{u};\bv) & = & 
	\mathcal{A}_0^\pm(\mathfrak{u};\bv)
	-\alpha_0 \int_\Gamma \Phi \cdot \bv^\pm \d \Gamma
	- \gamma \int_\Gamma 2\nu \varepsilon(\bv^\pm) \cdot \left(\sigma(\bu^\pm, p^\pm)\bn^\pm - \blambda^\pm \right) \d \Gamma,\\
	\mathcal{B}^\pm(\mathfrak{u};q) & = & \mathcal{B}_0^\pm(\mathfrak{u};q) 
	+ \gamma \int_\Gamma q^\pm \bn^\pm \cdot \left(\sigma(\bu^\pm, p^\pm)\bn^\pm - \blambda^\pm \right) \d \Gamma,\\
	\mathcal{C}^\pm(\mathfrak{u};\bmu) & = & 
	\mathcal{C}_0^\pm(\mathfrak{u};\bmu) + \gamma \int_\Gamma \bmu^\pm \cdot \left(\sigma(\bu^\pm, p^\pm)\bn^\pm - \blambda^\pm \right) \d \Gamma,\\
	\mathcal{D}^\pm(\mathfrak{u};\bvarphi) & = & 
	\mathcal{D}_0^\pm(\mathfrak{u};\bvarphi) 
	+ 2\alpha_0 \int_\Gamma \Phi \cdot \bvarphi \, \d \Gamma .
\end{eqnarray*}
With the introduction of the stabilization terms, the goal is to prove a theoretical result, namely Theorem~\ref{th-infsup}, leading to the optimal convergence of all the variables, in particular the multipliers. For that purpose, we make a list of assumptions:
\begin{itemize}
\item[$(\mathbf{A1})$:] There exists a constant $C >0$ independent of $h$ such that
\begin{eqnarray*}
	\inf_{q^{\pm}_h \in \Q^{\pm}_h} \sup_{\bv_h^{\pm} \in \mathbf{V}^{\pm}_{0,h}}
	\frac{b(\bv^{\pm}_h, q^{\pm}_h)}{\|q^{\pm}_h\|_{\Q^{\pm}_h} \|\bv^{\pm}_h\|_{\mathbf{V}^{\pm}_{h}}} & \geq & C. 
\end{eqnarray*}

\item[$(\mathbf{A2})$:] There exists $C>0$ independent of $h$ such that for all $\bv^{\pm}_h \in \mathbf{V}^{\pm}_h$ one has 
\begin{eqnarray*}
\qquad  h\|\varepsilon(\bv^{\pm}_h)\bn^\pm\|^2_{\LL^2(\Gamma)} \leq 
C\| \bv^{\pm}_h \|^2_{\mathbf{V}^{\pm}_h}.
\end{eqnarray*}

\item[$(\mathbf{A3})$:] There exists $C>0$ independent of $h$ such that for all $q^{\pm}_h \in \Q^{\pm}_h$ one has 
\begin{eqnarray*}h\|q^{\pm}_h\|^2_{\mathbb{L}^2(\Gamma)} \leq 
C\| q^{\pm}_h \|^2_{\Q^{\pm}_h}.
\end{eqnarray*}
\end{itemize}

Assumptions~$(\mathbf{A2})-(\mathbf{A3})$ are in the same fashion as those made in~\cite{Court2014}. In practice, we can consider that they are satisfied when the intersections of $\Omega^\pm$ with the simplices of the mesh are not too small. This question is discussed in~\cite[Section~6 and Appendix~B]{Renard2009}, where an alternative stabilization technique is proposed in order to avoid situations for which the geometric configuration would not allow the fulfillment of these assumptions. However, in practice, the frequency of these \textcolor{black}{geometric} situations \textcolor{black}{can be reduced (by refining the mesh locally for instance)}, and their impact on the accuracy of the method is quite negligible, so that it is reasonable to consider these assumptions. On the other hand, assumption~$(\mathbf{A1})$ (reproduced from assumption~$(\mathbf{H1})$) is considered, but there is {\it a priori} no reason that it is satisfied for a general geometric configuration. Its fulfillment could be enforced with an additional specific stabilization technique, but this point is not of our interest in this work. Sticking to assumption~$(\mathbf{A1})$ can be made in practice by choosing pair of elements of type $P_{k+1}$/$P_k$ for the couple {\it velocity}/{\it pressure}, as the so-called Taylor-Hood elements.\\

For the sake of concision, we now denote 
\begin{eqnarray*}
\mathfrak{u}_h = (\bu_h^+,p_h^+,\blambda_h^+, \bu_h^-,p_h^-,\blambda_h^-,\Phi_h)
& \text{ and } &
\mathfrak{v}_h = (\bv_h^+,q_h^+,\bmu_h^+, \bv_h^-,q_h^-,\bmu_h^-,\bvarphi_h). 
\end{eqnarray*}
The weak formulation of the approximated stabilized problem~\eqref{pbstabilized} is given in a compact form, as follows:
\begin{eqnarray}
& & \text{Find	$\mathfrak{u}_h \in \VV_h^+  \times \Q_h^+ \times \mathbf{W}_h \times \VV_h^-\times \Q_h^- \times \mathbf{W}_h \times \ZZ_h$ such that} \nonumber \\
& & \mathcal{M}(\mathfrak{u}_h, \mathfrak{v}_h)  =  \mathcal{H}(\mathfrak{v}_h) \qquad
\text{for all } \mathfrak{v}_h \in \VV_h^+  \times \Q_h^+ \times \mathbf{W}_h \times \VV_h^-\times \Q_h^- \times \mathbf{W}_h \times \ZZ_h.
\label{pbstabilized}
\end{eqnarray}
where
\begin{eqnarray*}
	\mathcal{M}(\mathfrak{u};\mathfrak{v}) &  =  & 
	2\nu^+ \int_{\Omega^+}\varepsilon(\bu^+):\varepsilon(\bv^+)\, \d \Omega^+ + 2\nu^- \int_{\Omega^+}\varepsilon(\bu^-):\varepsilon(\bv^-)\, \d \Omega^- \\
	& & - \int_{\Omega^+} \left( p^+\divg \bv^+ + q^+\divg \bu^+ \right)\d \Omega^+ - \int_{\Omega^-} \left( p^-\divg \bv^- + q^-\divg \bu^- \right)\d \Omega^- \\
	& & - \int_{\Gamma} \left( \blambda^+ \cdot (\bv^+-\bvarphi) + \bmu^+\cdot (\bu^+ - \Phi)\right)\d \Gamma - \int_{\Gamma} \left( \blambda^- \cdot (\bv^--\bvarphi) + \bmu^-\cdot (\bu^- - \Phi)\right)\d \Gamma \\
	& & + \alpha_0 \int_{\Gamma} (\bu^+ -\Phi) \cdot (\bv^+ - \bvarphi) \, \d \Gamma 
	+ \alpha_0 \int_{\Gamma} (\bu^- -\Phi) \cdot (\bv^- - \bvarphi) \, \d \Gamma \\
	& & -\gamma_0 h\int_{\Gamma} (2\nu \varepsilon(\bu^+)\bn^+ - p^+\bn^+ - \blambda^+) \cdot 
	(2\nu \varepsilon(\bv^+)\bn^+ - q^+\bn^+ - \bmu^+)\, \d \Gamma \\
	& & -\gamma_0 h\int_{\Gamma} (2\nu \varepsilon(\bu^-)\bn^- - p^-\bn^- - \blambda^-) \cdot 
	(2\nu \varepsilon(\bv^-)\bn^- - q^-\bn^- - \bmu^-)\, \d \Gamma, \\
	\mathcal{H}(\mathfrak{v}) & = & \int_{\Omega^+} \bff^+ \cdot \bv^+  \d \Omega^+
	+ \int_{\Omega^-} \bff^- \cdot \bv^-  \d \Omega^-
	+ \int_{\Gamma} \bgg \cdot \varphi  \, \d \Gamma.
\end{eqnarray*}\textbf{}

Here again, in the approximated problem we have replaced the duality brackets $\langle \, \cdot \, ; \cdot \, \rangle_{\WW', \WW}$ by $\LL^2(\Gamma)$ scalar products. We are now in position to \textcolor{black}{establish} a discrete inf-sup condition for the stabilized problem.

\begin{theorem} \label{th-infsup}
	Assume that $(\mathbf{A1})$--$(\mathbf{A3})$ hold. Then, for $\alpha_0$ and $\gamma_0$ small enough, there exists  a constant $C>0$ independent of the mesh size $h$ such that
	\begin{eqnarray*}
	\inf_{\mathfrak{u}_h \in \mathfrak{V}_h} \sup_{\mathfrak{v}_h \in \mathfrak{V}_h}
	\frac{\mathcal{M}(\mathfrak{u}_h;\mathfrak{v}_h)}
	{|||\, \mathfrak{u}_h\, |||\ |||\, \mathfrak{v}_h\, |||} & \geq & C,
	\end{eqnarray*}
	where $\mathfrak{V}_h =\VV_h^+  \times \Q_h^+ \times \mathbf{W}_h \times \VV_h^-\times \Q_h^- \times \mathbf{W}_h \times \ZZ_h$, and where the norm $|||\, \cdot \, |||$ is defined as follows:
	\begin{eqnarray*}
	|||\, \mathfrak{u}\, |||^2 & = & \|\bu^+\|^2_{\mathbf{V}^+} + \|p^+\|_{\Q^{+}} 
	+ \|\bu^-\|^2_{\mathbf{V}^-} + \|p^-\|_{\Q^{-}} 
	+h \|\blambda^+\|^2_{\mathbf{L}^2(\Gamma)} +h \|\blambda^-\|^2_{\mathbf{L}^2(\Gamma)} \\
	& & + h\left(\|\varepsilon(\bu^+)\bn\|^2_{\mathbf{L}^2(\Gamma)} +
	\|\varepsilon(\bu^-)\bn\|^2_{\mathbf{L}^2(\Gamma)} + 
	\|p^-\|^2_{\mathbf{L}^2(\Gamma)} + \|p^-\|^2_{\mathbf{L}^2(\Gamma)}\right) \\
	& & +\frac{1}{h}\|\bu^+-\Phi\|^2_{\mathbf{L}^2(\Gamma)}
	+\frac{1}{h}\|\bu^--\Phi\|^2_{\mathbf{L}^2(\Gamma)} + \| \Phi \|^2_{\mathbf{W}}.
	\end{eqnarray*}
\end{theorem}

\textcolor{black}{For the sake of clarity, we give the proof of Theorem~\ref{th-infsup} in section~\ref{App-A2}.} The consequence of this result is the optimal order of convergence for the multipliers, stated as follows:
\begin{corollary} \label{supercoro}
Denote by $k_{\bu}$, $k_p$, $k_\lambda$ and $k_{\Phi}$ the respective degrees of standard finite elements used for the velocities $\bu^\pm$, the pressures $p^\pm$ and the multipliers $\blambda^\pm$ and $\Phi$. Then
\begin{eqnarray*}
& & \max\left(\| \bu^\pm - \bu^\pm_h\|_{\VV^\pm}, \|p^\pm-p_h^\pm\|_{\Q^\pm}, h\|\blambda^\pm-\blambda^\pm_h\|_{\WW}, \|\Phi - \Phi_h\|_{\ZZ} \right) \\
& & \leq C\left( h^{k_{\bu}}\|\bu^+\|_{\HH^{k_{\bu}+1}(\Omega^+)} 
	+h^{k_{p}+1}\|p^+\|_{\H^{k_p+1}(\Omega^+)}
	+h^{k_{\blambda}+1} \|\blambda^+\|_{\HH^{k_{\blambda}+1/2}(\Gamma)} \right. \\
	& & \left. + h^{k_{\Phi}+1} \| \Phi\|_{\HH^{k_{\blambda}+1/2}(\Gamma)}
	+h^{k_{\bu}}\|\bu^-\|_{\HH^{k_{\bu}+1}(\Omega^-)}
	+ h^{k_{p}+1}\|p^-\|_{\H^{k_p+1}(\Omega^-)}  
	+h^{k_{\blambda}+1} \|\blambda^-\|_{\HH^{k_{\blambda}+1/2}(\Gamma)}\right).
\end{eqnarray*}
\end{corollary}

\begin{proof}
On the space $\mathfrak{V}_h =  \VV_h^+  \times \Q_h^+  \times \WW_h  \times \VV_h^-  \times \Q_h^- \times \WW_h \times \ZZ_h$ endowed with the norm $||| \, \cdot \, |||$, the bilinear form $\mathcal{M}$ is uniformly continuous with respect to the mesh size $h$. Then Theorem~\ref{th-infsup} combined with a C\'ea-type lemma (see~\cite{Ern} for instance) yields the following error estimate:
\begin{eqnarray*}
||| \, \mathfrak{u}- \mathfrak{u}_h \, |||& \leq & C
\inf_{\mathfrak{v}_h \in \mathfrak{V}_h} ||| \, \mathfrak{u}- \mathfrak{v}_h \, |||.
\end{eqnarray*}
We now invoke the extension theorem for the Sobolev spaces, the standard
estimates for the nodal finite element interpolation operators, and the trace inequality
\begin{eqnarray*}
	\| \bvarphi \|_{\LL^2(\Gamma)} & \leq & 
	C \left(h\|\bvarphi\|_{\LL^2(T)} + \frac{1}{h}\|\bvarphi\|_{\LL^2(T)} \right)
\end{eqnarray*}
on any convex $T \in \mathcal{T}_h$, to complete the proof. We refer to ~\cite[Appendix~A,~page~1496]{Renard2009} for more details.
\end{proof}

\section{On the implementation of the method} \label{sec-impl}
This section is dedicated to remarks about the practical implementation of the method. We mention the libraries we use for the writing of the code, and we underline their advantages for the efficiency of the method.

\subsection{Libraries used for the implementation} \label{sec-libimpl}
Our finite element code is written under the Getfem++ Library (see \cite{Getfem}). \textcolor{black}{The method follows} the approach initially introduced for the Poisson problem in~\cite{Renard2009}, where functionalities of the library were defining the fictitious domain approach. It was next extended to the Stokes problem in~\cite{Court2014, Court2015} for standard Dirichlet conditions. In dimension 2 and 3, solving the global system can be made by using the solver MUMPS (see~\cite{MUMPS1, MUMPS2}), while using the linear algebra Gmm++ Library (installed inside the Getfem++ library).

For the consideration of boundary conditions where the boundary is independent of the mesh, the library Getfem++ enables us to solve several difficulties like the following:
\begin{itemize}
\item Defining basis functions of $\mathbf{W}_h$ from traces on $\Gamma$ of the standard basis functions of $\tilde{\WW}_h$. Note that the linear independence of these basis functions is not automatically satisfied {\it a priori}, and so redundant degrees of freedom may need to be eliminated {\it a posteriori} (with a range basis algorithm), if we do not want to deal with singular systems.
\item Detecting the interface $\Gamma$. For that, we can consider a level-set function, and we use functionalities of the library that compute objects related to this level-set function (values, gradient, unit normal vector...). If we choose analytically a general expression of a level-set function for localizing the interface, then a piecewise polynomial approximation is carried out in the implementation. We can specify the (polynomial) degree of this approximation.
\item During the assembly procedure, computing with accuracy the integrals over elements that are intersected by this interface. This is done with the call of the Qhull Library (see~\cite{Qhull}).
\end{itemize}

\subsection{Efficient update of matrices when the geometry evolves} \label{sec-smartupdate}
The most important interest of fictitious domain methods is to avoid to re-mesh when the \textcolor{black}{interface} has to be modified, and thus to spare computation time and resources. Indeed, re-meshing implies re-assembly of the whole system, and the computation of numerous integration terms of the matrices, during the assembly procedure of a complex simulation, is the most costly part in terms of computation time. For avoiding this, let us explain how to update a restricted number of terms between two different geometric configurations.

First we compute a stiffness matrix on the whole domain, independent of the \textcolor{black}{interface} $\Gamma$: Denoting by $\{ \tilde{\bvarphi}_i \}$ the basis functions of the discrete space $\tilde{\VV}_h$, we assemble the following matrix:
\begin{eqnarray*}
	\tilde{\mathbf{A}}_{ij} & = & \int_{\Omega} \varepsilon(\tilde{\bvarphi_i}): \varepsilon(\tilde{\bvarphi_j}) \, \d \Omega.
\end{eqnarray*}
This matrix is assembled once for all, and stored. Given an interface $\Gamma$ immersed into $\Omega$, the goal is now to construct efficiently the stiffness matrices used effectively for solving the system for the corresponding geometric configuration, namely the matrices
\begin{eqnarray*}
	\mathbf{A}^\pm_{ij} & = & \int_{\Omega^\pm} \varepsilon(\bvarphi^\pm_i): \varepsilon(\bvarphi^\pm_j) \, \d \Omega^\pm,
\end{eqnarray*}
where $\{ \bvarphi_i^\pm \}$ denote the basis functions of spaces $\VV^\pm_h$. The indexes of these basis functions, can be deduced from those of the standard ones $\{ \tilde{\bvarphi_i} \}$ by the use of reduction matrices $\mathbf{R}^\pm$. These matrices are sparse and binary, and so can be used inexpensively. Analogously we define the extension matrices $\mathbf{E}^\pm$, with which we associate the following properties:
\begin{eqnarray*}
\mathbf{E}^\pm = {\mathbf{R}^\pm}^T, & & \mathbf{R}^\pm\mathbf{E}^\pm = \mathbf{I}.
\end{eqnarray*}
From there, we can define the following partial stiffness matrices:
\begin{eqnarray*}
	\tilde{\mathbf{A}}^+  =  {\mathbf{R}}^+ \tilde{\mathbf{A}} {\mathbf{E}}^+, & & 
	\tilde{\mathbf{A}}^-  =  {\mathbf{R}}^- \tilde{\mathbf{A}} {\mathbf{E}}^-.
\end{eqnarray*}
However, these reduction matrices enable us only to select the indexes of the family $\{ \tilde{\bvarphi_i} \}$ that concern and localize the domains $\Omega^\pm$ (see Figure~\ref{fig-fictistyle}). The definition of functions $\{ \bvarphi_i \}$ from functions $\{ \tilde{\bvarphi}_i \}$ requires the identification of the triangles of the mesh that are cut by the $\Gamma$. The latter is localized with a level-set function, and the approximation of this level-set is made with the use of piecewise polynomial functions. The way the approximated level-set cuts the mesh defines subsimplices with corresponding Heaviside functions (see section~\ref{subsec-fict}). Thus we obtain the functions $\{\bvarphi_i\}$ by multiplying $\{ \tilde{\bvarphi_i} \}$ with the Heaviside functions on \textcolor{black}{these} subsimplices. The matrices $\mathbf{A}^\pm$ are then obtained with a local reassembly of the integration terms of the matrices $\tilde{\mathbf{A}}^\pm$ by including the Heaviside functions for the terms concerned by the subsimplices aforementioned. By these steps, we claim that we update a number of objects that is of the same range as the number of the mesh elements intersected by the \textcolor{black}{interface}. Of course, the same procedure can be transposed for other matrix blocks of the system.

\begin{algorithm}[htpb]
\begin{description}
\item[First Assembly:] Compute matrix $\tilde{\mathbf{A}}$ independent of the interface, once for all, and store it.

\item[Initialization: $k=0$.] For a given parameterized geometry, initialize
\begin{description}
\item[1:] the level-set expression \texttt{ls-value}, and next the level-set object \texttt{ls},
\item[2:] the partial integration methods \texttt{mim} and the partial finite element methods \texttt{mf}.
\item[3:] From \texttt{mim}, identify the indices of \texttt{mf} concerned by the interface.
\item[4:] From $\tilde{\mathbf{A}}$, define $\tilde{\mathbf{A}}^\pm$, and next $\mathbf{A}$ by reassembling the terms concerned by the interface.
\item[5:] First solve, use of MUMPS.
\end{description}

\item[Iteration $k\geq 0$.] From the first solve, define a new geometric configuration, and update
\begin{description}
\item[1:] the level-set expression, and next the level-set object: \texttt{ls.adapt();}
\item[2:] the partial integration methods and the partial finite element methods: \texttt{mim.adapt();} \texttt{mf.adapt();}
\item[3:] From the former \texttt{mim} and the new \texttt{mim}, identify the indices of the new \texttt{mf} concerned by the change of interface.
\item[4:] From $\tilde{\mathbf{A}}$, define $\tilde{\mathbf{A}}^\pm$, and next $\mathbf{A}$ by reassembling the terms concerned by the change of interface.
\item[5:] Solve, use of MUMPS.
\end{description}
\end{description}
\caption{Efficient update algorithm}\label{UCG}
\end{algorithm}
\FloatBarrier

\textcolor{black}{Note that in particular cases where storing an LU decomposition -- of matrix $\tilde{\mathbf{A}}$ is possible, then the same reduced update procedure could be performed for updating this decomposition without redoing it entirely. Note that the stabilization matrices have to be reassembled anyway, but the corresponding time computation is negligible compared with the time needed for the assembly of $\tilde{\mathbf{A}}$.}

\section{Numerical tests} \label{sec-numtests}
In order to illustrate the theoretical analysis, and in particular to underline the result of Corollary~\ref{supercoro} which guarantees the optimal convergence rates, we propose numerical tests. The convergence tests will be performed for the square domain $\Omega = [0,1]^2$ ($d=2$), with interfaces $\Gamma$ represented by level-set of type 
\begin{eqnarray*}
	\begin{array} {lcl}
	\mathrm{\ell s}(x,y) = ( x-x_c)^2 + ( y-y_c)^2 - R^2 & & \text{if } d=2, 
	\end{array}
\end{eqnarray*}
where $(x_c,y_c)$ denotes the coordinates of the center of a \textcolor{black}{circle}, and where $R>0$ denotes its radius. We will use the following exact solutions:
\begin{eqnarray*}
	& & \begin{array} {lcl}
	\bu_{ex}(x,y) = \left( \begin{array} {c} 
	\cos(\pi x)\sin(\pi y) \\ 
	-\sin(\pi x)\cos(\pi y)
	\end{array} \right) 
	& & \\
	 \text{\textcolor{black}{$p_{ex}(x,y)$}} = c_p\left((y-y_c)\cos(2\pi x)+(x-x_c)\sin(2\pi y)\right) & & 
	 \text{with $c_p= 1$ or $3$,} \\
	& & \text{whether \textcolor{black}{$\ell s(x,y) < 0$} or $>0$ respectively.}
	\end{array}
\end{eqnarray*}
Note that these solutions satisfy $\divg \bu_{ex} = 0$, and when the interface is a sphere of center $(x_c,y_c)$, the pressures satisfy automatically $\displaystyle \int_{\Omega^{\pm}} p_{ex}^{\pm}\, \d \Omega^{\pm}  =  0$. The following data of system~\eqref{mainsys} are thus considered in the numerical tests:
\begin{eqnarray*}
\bff^\pm  =  -\nu^{\pm}\Delta_{ex}\bu_{ex} + \nabla p_{ex}, & & 
\bgg  =  2(\nu^+-\nu^-)\varepsilon(\bu_{ex})\bn - (p^+_{ex}-p^-_{ex})\bn.
\end{eqnarray*}
We choose $(x_c,y_c) = (0.5,0.5)$, $R = 0.23$, $\nu^+ = 2.0$ and $\nu^- = 1.0$.

\subsection{Convergence rates without stabilization} \label{subsec-cvstab0}
For the different variables, we compute the relative errors with respect to the exact solutions given previously. Results are given in Figure~\ref{fig-cv1}, Figure~\ref{fig-cv2} and Figure~\ref{fig-cv3}, where the rates of convergence (obtained for different mesh sizes) are computed by linear regression. The notation P2/P1/P0, for instance, indicates that P2 elements are chosen are chosen for the velocities $\bu^\pm$, P1 elements are chosen for the pressures $p^\pm$, and P0 elements are chosen for the multipliers $\Phi$ and $\blambda^\pm$. The quantities represented in Figure~\ref{fig-cv1} and Figure~\ref{fig-cv2} are relative errors computed on the global variables $\bu_h$ and $p_h$ defined as follows:
\begin{eqnarray*}
	\bu_h(\bx) = \left\{ \begin{array} {ll}
		\bu_h^+(\bx) & \text{if } \bx \in \Omega^+, \\
		\bu_h^-(\bx) & \text{if } \bx \in \Omega^-,
	\end{array} \right.
	& & 
	p_h(\bx) = \left\{ \begin{array} {ll}
		p_h^+(\bx) & \text{if } \bx \in \Omega^+, \\
		p_h^-(\bx) & \text{if } \bx \in \Omega^-.
	\end{array} \right.
\end{eqnarray*}
The quantities represented in Figure~\ref{fig-cv3} correspond to a mean value between the errors committed on $\blambda^+$ and $\blambda^-$, namely the square root of the quantity given below:
\begin{eqnarray}
\frac{
\| \blambda_h^+ - \sigma^+(\bu_{ex}^+,p_{ex}^+)\bn^+ \|^2
+ \| \blambda_h^- - \sigma^-(\bu_{ex}^-,p_{ex}^-)\bn^- \|^2
}
{
\| \sigma^\pm(\bu_{ex}^+,p_{ex}^+)\bn^+ \|^2
+ \|  \sigma^\pm(\bu_{ex}^-,p_{ex}^-)\bn^- \|^2
}. \label{formula-error-lambda}
\end{eqnarray}

\begin{minipage}{\linewidth}
\begin{center}
\hspace*{-20pt}\begin{tabular} {r|l}
\includegraphics[trim = 0cm 0cm 1.0cm 0cm, clip, scale=0.35]{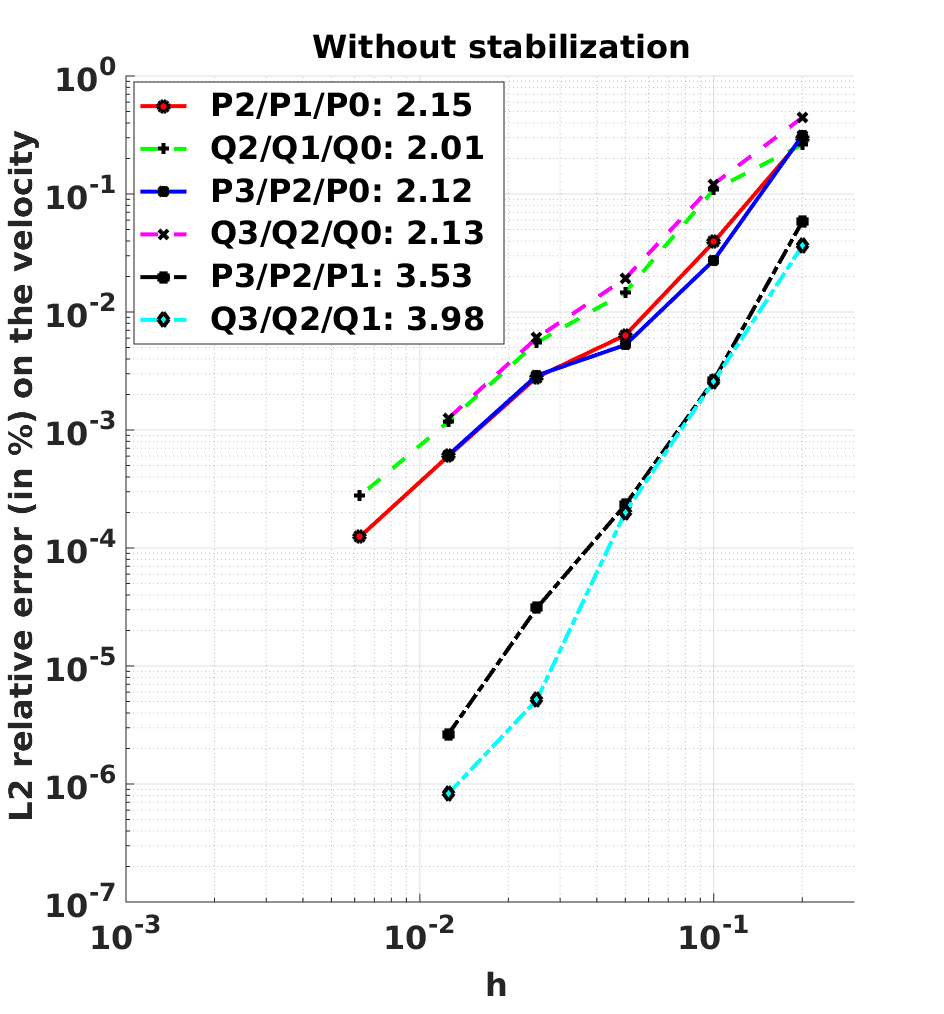}
&
\includegraphics[trim = 0cm 0cm 1.0cm 0cm, clip, scale=0.35]{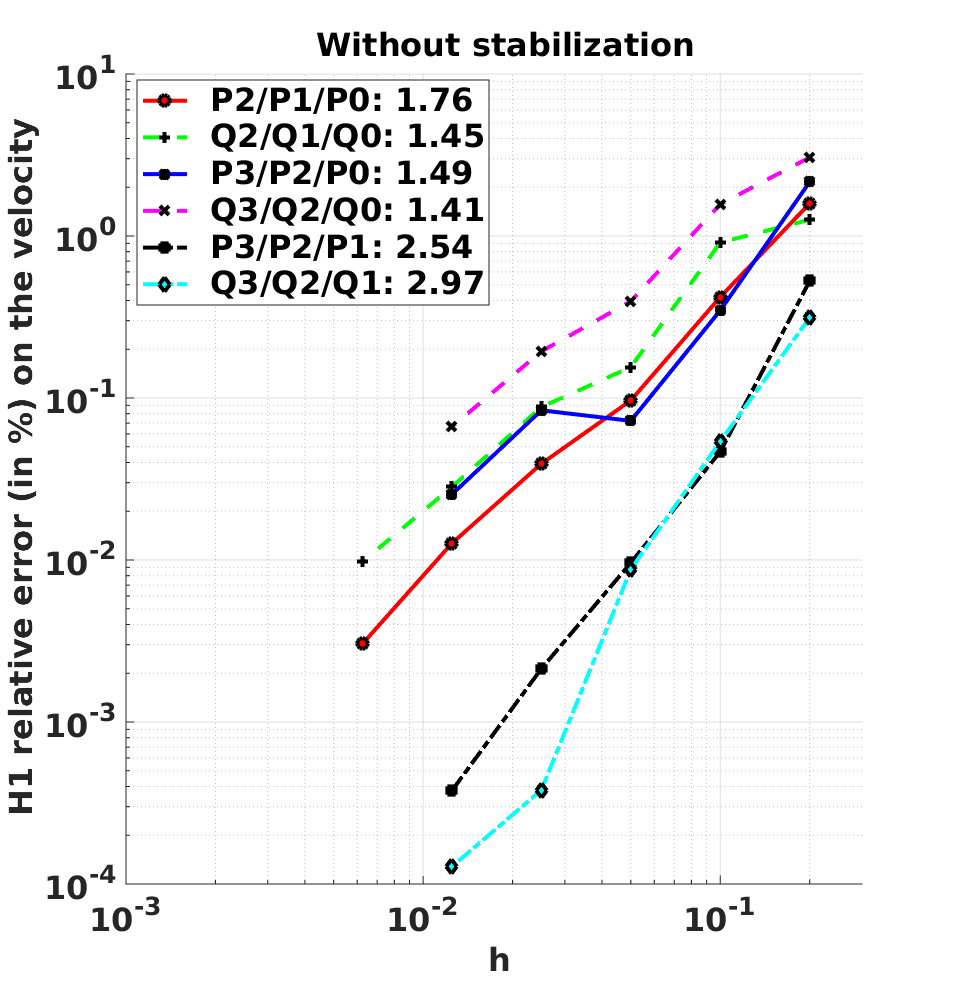}
\end{tabular}
\begin{figure}[H]
	\vspace*{-15pt}
	\caption{$\LL^2$ and $\HH^1$-relative errors (in \%) on the velocity $\bu$ in function of the mesh size, and estimation of convergence rates with the slopes of the curves obtained by linear regression.}
	\label{fig-cv1}
\end{figure}
\end{center}
\end{minipage}

\begin{minipage}{\linewidth}
\begin{center}
\hspace*{-20pt}\begin{tabular} {c}
\includegraphics[trim = 0cm 0cm 1.0cm 0cm, clip, scale=0.35]{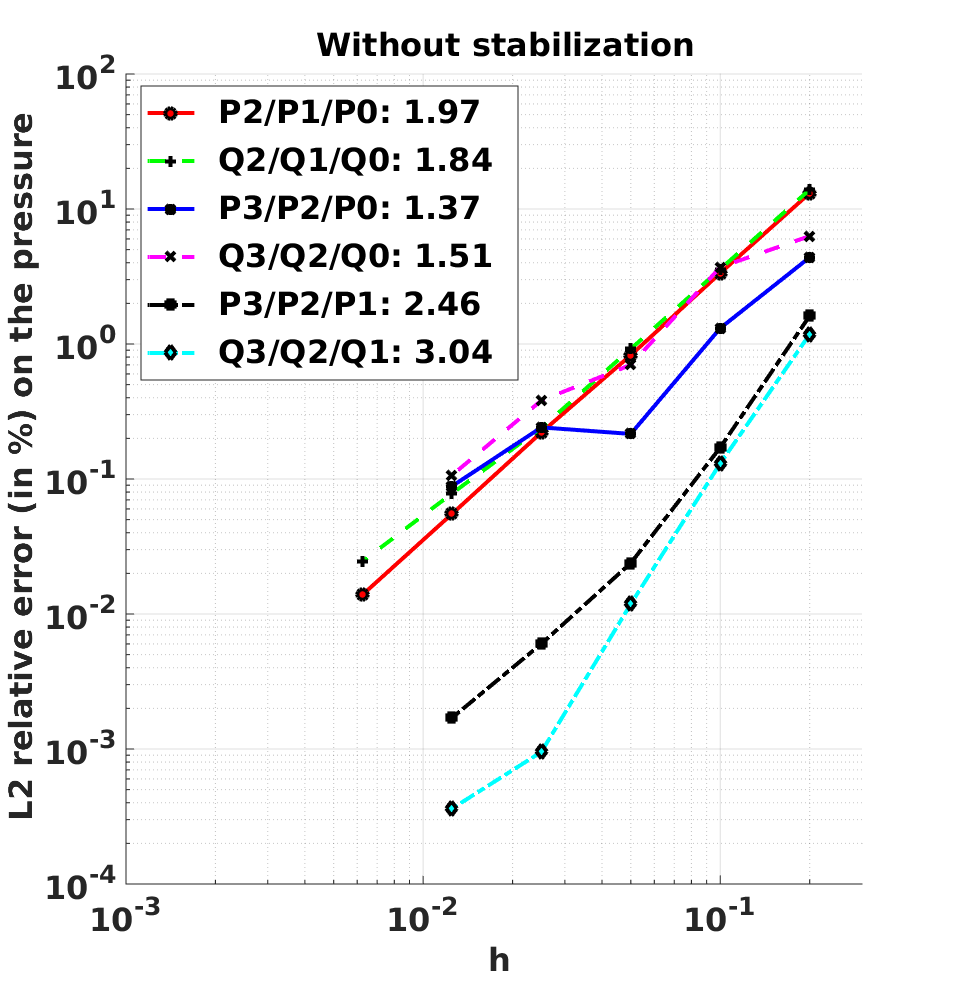}
\end{tabular}
\begin{figure}[H]
	\vspace*{-15pt}
	\caption{$\LL^2$-relative errors (in \%) on the pressure $p$ in function of the mesh size, and estimation of convergence rates with the slopes of the curves obtained by linear regression.}
	\label{fig-cv2}
\end{figure}
\end{center}
\end{minipage}
\FloatBarrier
\begin{minipage}{\linewidth}
\begin{center}
\hspace*{-20pt}\begin{tabular} {r|l}
\includegraphics[trim = 0cm 0cm 1.0cm 0cm, clip, scale=0.35]{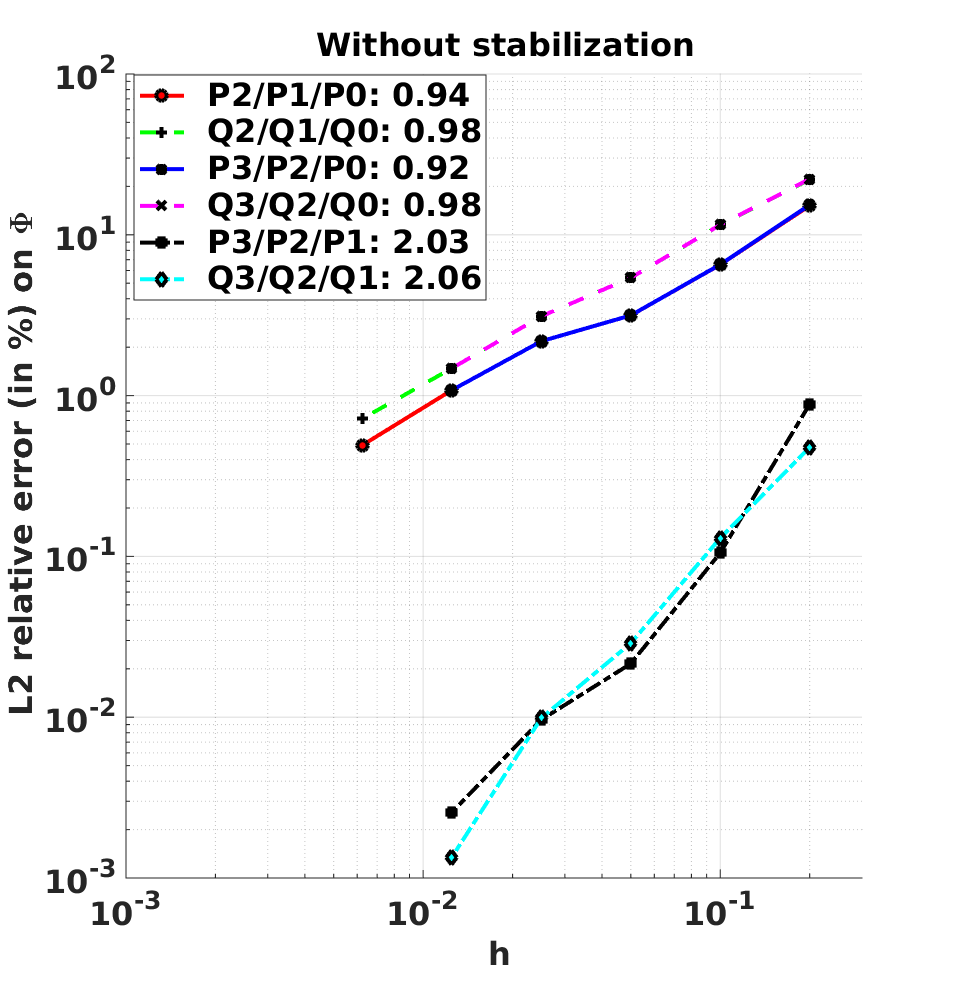}
&
\includegraphics[trim = 0cm 0cm 1.0cm 0cm, clip, scale=0.35]{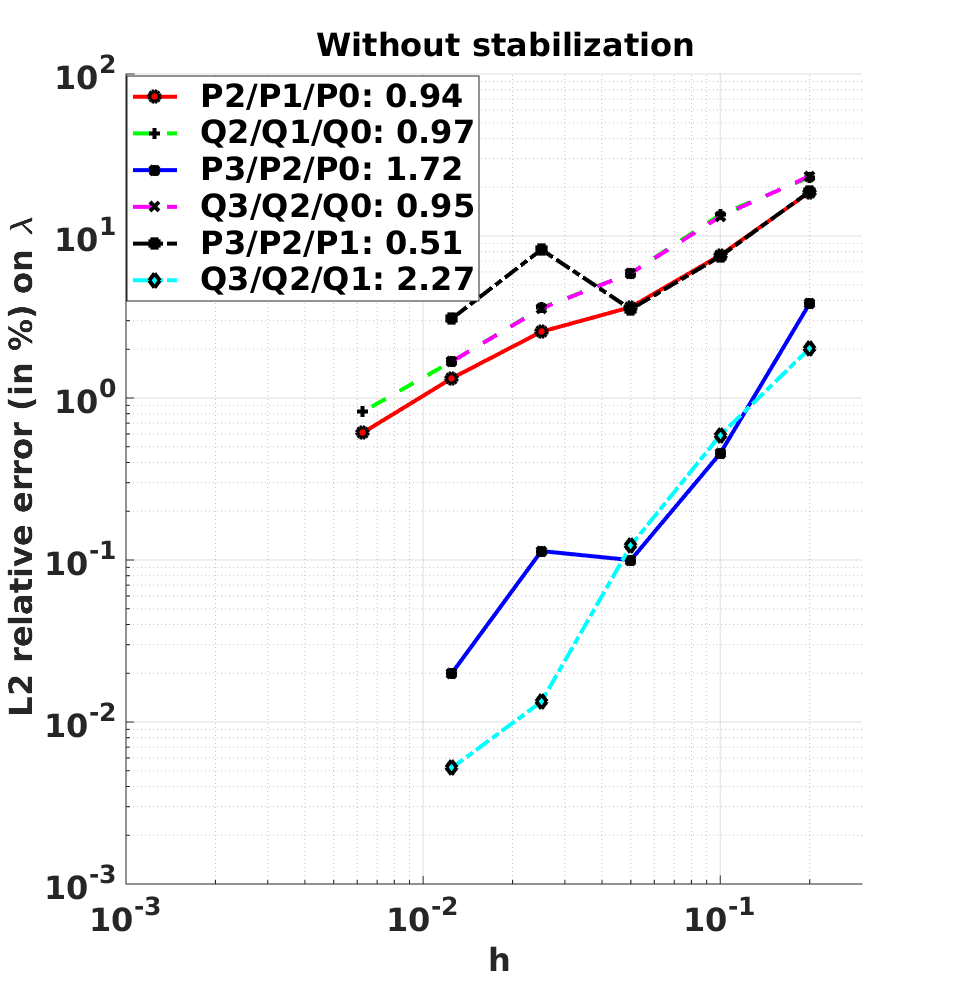}
\end{tabular}
\begin{figure}[H]
	\vspace*{-15pt}
	\caption{$\LL^2(\Gamma)$-relative errors (in \%) on the multipliers $\Phi$ and $\blambda^\pm$ in function of the mesh size, and estimation of convergence rates with the slopes of the curves obtained by linear regression.}
	\label{fig-cv3}
\end{figure}
\end{center}
\end{minipage}
\FloatBarrier
\hfill \\ \hfill \\
In Figure~\ref{fig-cv1}, we observe that the optimal convergence rates for the velocity seem to be reached for the Q3/Q2/Q1 triplet of elements. The rates for the triplet P3/P2/P1 are slightly degraded, while a limitation seem to occur when we choose P0 or Q0 elements for the multipliers. Indeed, in those cases the computed rate for the $\LL^2$-norm of the velocity is around $2$, and the one for the $\HH^1$-norm is about $1.5$. However, the limitation on the quality of convergence for the $\HH^1$-norm is not as bad as the one announced by Proposition~\ref{prop-limit}. In Figure~\ref{fig-cv2}, we can make the same observations on the quality of convergence for the pressure. The optimal rates is here again achieved when we choose the Q3/Q2/Q1 triplet, but also when we choose the P2/P1/P0 triplet. The effective rate of convergence for the pressure is significantly degraded for the P3/P2/P0 and Q3/Q2/Q0 triplets. In Figure~\ref{fig-cv3}, we see that the optimal rate of convergence for the variable $\Phi$ seems to be achieved in all cases, while for the variables $\blambda^\pm$ the accuracy appears to be particularly bad when we choose the P3/P2/P1 triplet.

\subsection{With stabilization} \label{subsec-cvstab1}
In this subsection we study the influence of the stabilization technique on the accuracy of the method. First, in all the tests we performed, we do not see any significant influence of the stabilization technique for the variable $\Phi$: For all the values we considered for the parameter $\alpha$, for all kind of mesh size and all kind of geometric configuration, the difference observed with and without stabilization is negligible. With $\alpha_0 = 0$, we already observe in Figure~\ref{fig-cv3} the optimal rates of convergence for the variable $\Phi$. This may be due to a gap between the theoretical analysis and the numerical realization. The lack of theoretical convergence announced in section~\ref{subsec-theor} seems to be too pessimistic, and maybe it is true that without the stabilization terms that concern $\Phi$, the optimal theoretical convergence for $\Phi$ is automatically guaranteed.

Thus the rest of numerical tests will be performed with $\alpha_0 = 0$, and we will focus our interest on the parameter $\gamma$ for the stabilization of the variables $\blambda^\pm$. Without stabilization for the variables $\blambda^\pm$, the accuracy showed in Figure~\ref{fig-cv3} is already satisfying, {\it a priori}, but we are going to see that in some situations the stabilization technique is crucial. Getting a good approximation of these quantities can be of interest for fluid-structure models, with the consideration of other \textcolor{black}{interface} conditions, for control problems where their expressions can appear in adjoint systems, or simply for physical reasons.

\paragraph{Choice of the stabilization parameter.}
Remind that the stabilization parameter is chosen to be proportional to the mesh size, as $\gamma = \gamma_0h$. Let us choose the parameter $\gamma_0$ judiciously. Indeed, a too large stabilization parameter would degrade the coerciveness of the system. For this task, for different values of $\gamma_0 >0$, we compute the $\LL^2(\Gamma)$ relative errors on the multipliers $\blambda^\pm$ (as explained in section~\ref{subsec-cvstab0}, formula~\eqref{formula-error-lambda}) for the P2/P1/P0 triplet of elements, with the mesh size $h=0.1$, with the geometric configuration and the exact solutions described at the beginning of section~\ref{sec-numtests}. The results are presented in Figure~\ref{fig-choice}.

\begin{center}
	\hspace*{-10pt}\includegraphics[trim = 0cm 0cm 0cm 0cm, clip, scale=0.35]{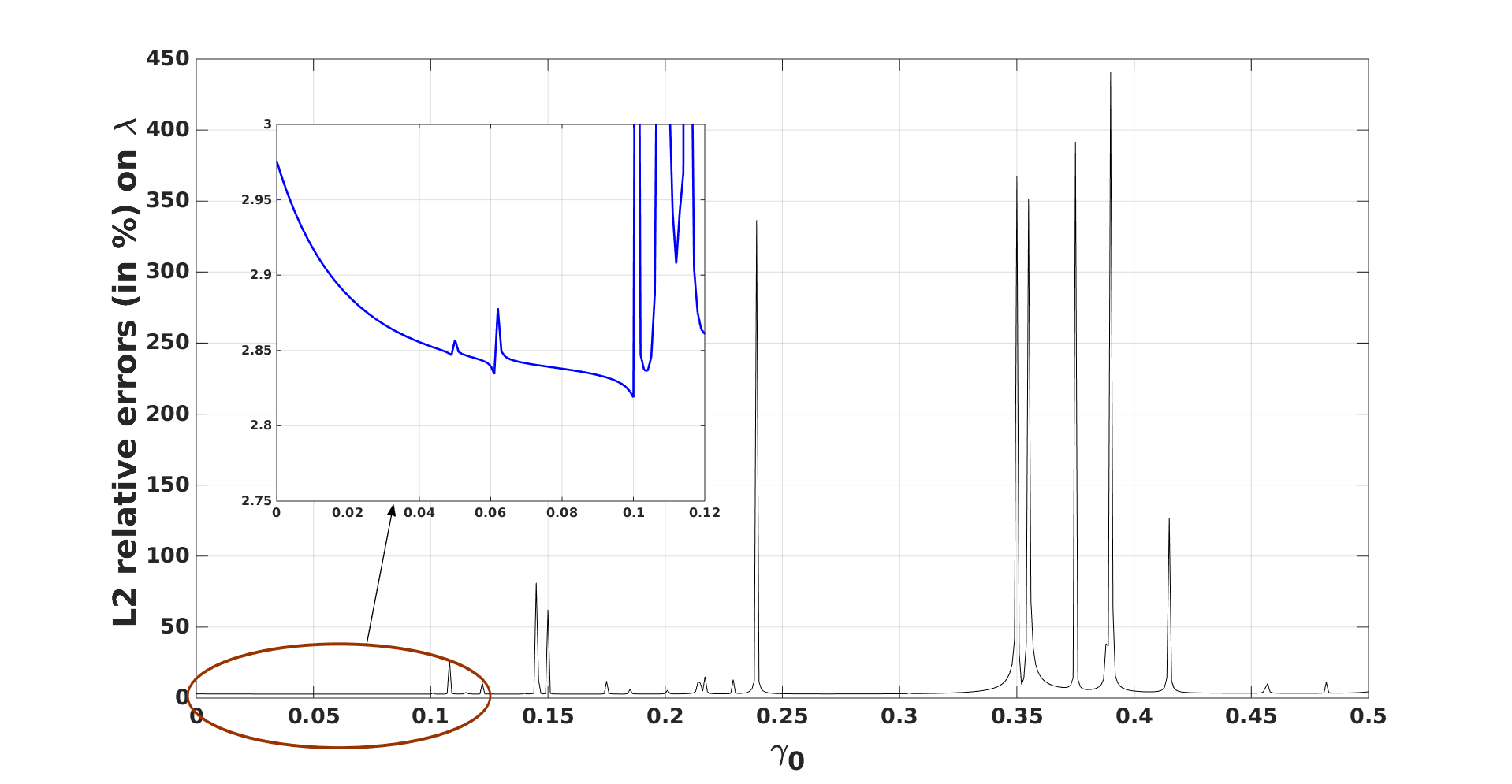}
	\vspace*{-30pt}
	\begin{figure}[H]
		\caption{$\LL^2(\Gamma)$ relative errors on the multipliers $\blambda^\pm$, for different values of the stabilization parameter $\gamma_0$.}
		\label{fig-choice}
	\end{figure}
\end{center}
\FloatBarrier

We observe in Figure~\ref{fig-choice} that serious instabilities appear for values of $\gamma_0$ larger than $0.1$. Smaller instabilities actually occur for smaller values, leading us to choose $\gamma_0$ around $0.02$, even if for this range of values the influence of the stabilization technique on the accuracy is {\it a priori} negligible (see the comments of Figure~\ref{fig-robust} for further comments). For this kind of size of computational domains, and for this range of values for the viscosities, we will choose in the rest of the paper values \textcolor{black}{of $\gamma_0$ smaller than $0.02$}. \textcolor{black}{Note that the range of acceptable values for $\gamma_0$ does not depend on the mesh size $h$, as predicted by the theory. Performing the same tests for finer meshes leads to the same range of values.} \\

In practice, for this value of $\gamma_0$, the differences between the errors computed with and without the stabilization technique are not significant, for all the variables, and so we do not show the convergence curves obtained with the stabilization technique, since qualitatively the observations are the same. Actually, we show in the next paragraph that the main interest of the stabilization technique lies in its capacities of considering various geometric configurations.

\paragraph{Robustness with respect to the geometry.} \label{subsec-cvrobust}
Let us study the behavior of the stabilization technique for different geometric configurations. The goal is to anticipate an unsteady framework for which the level-set would have to cut the mesh randomly, and to demonstrate that this stabilization technique provides a qualitatively constant behavior in terms of accuracy. For that purpose, we propose to \textcolor{black}{compute and compare the relative errors on $\blambda^\pm$ (like in the previous paragraph), with and without stabilization}, and for different values of the abscissa $x_c$ of the center of the circle. The results are presented in Figure~\ref{fig-robust}.

\begin{center}
	\hspace*{-10pt}\includegraphics[trim = 0cm 0cm 0cm 0cm, clip, scale=0.35]{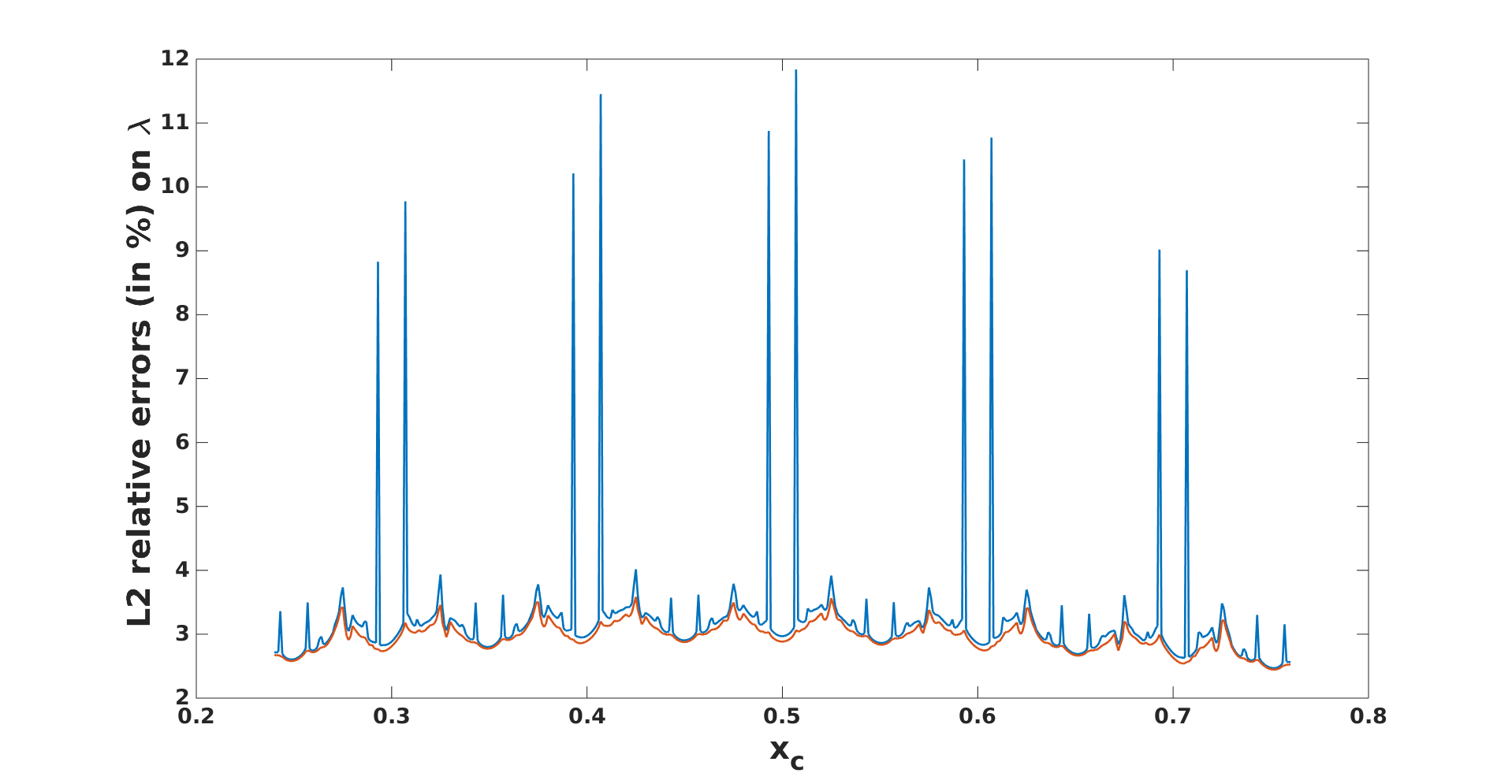}
	\vspace*{-30pt}
	\begin{figure}[H]
		\caption{$\LL^2(\Gamma)$ relative errors on the multipliers $\blambda^\pm$, for different geometric configurations, without stabilization (in blue) and with stabilization (in red) with $\gamma_0 = 0.02$.}
		\label{fig-robust}
	\end{figure}
\end{center}
\FloatBarrier

In Figure~\ref{fig-robust}, we observe that the stabilization technique enables us to prevent the instabilities that occur for some values of $x_c$ without stabilization. \textcolor{black}{Actually these instabilities are also fixed for very small values of $\gamma_0$ (like $\gamma_0 = 5.e^{-4})$.} Let us mention that for larger values of $\gamma_0$ (like $\gamma_0 = 0.03$), other instabilities appear for some values of $x_c$. Besides, the errors obtained with stabilization are -- almost -- always better than those computed without stabilization. This improvement is in general not really significant, the main interest of the stabilization technique lies of course in its robustness with respect of the geometry.

\section{An unsteady case: Deformation of an ellipsoid coupled with surface tension} \label{sec-unsteady}

This section is devoted to testing the capacities of the method in an unsteady \textcolor{black}{simplified} framework. \textcolor{black}{The aim is to illustrate that the method enables us to solve a problem for which the time evolution of an interface is coupled with its own geometry.} The \textcolor{black}{interface} $\Gamma$ will depend on time, and from now we denote it by $\Gamma(t)$. It splits the domain $\Omega$ into two parts that we denote by $\Omega^+(t)$ and $\Omega^-(t)$, playing the role of $\Omega^+$ and $\Omega^-$ respectively, as previously. The unknowns are now $(\bu^{\pm},p^{\pm})$ and the deformation of $\Gamma(t)$, that we denote by $X$. While the velocity and the pressure are described in Eulerian coordinates, the description of the deformation $X$ is Lagrangian:
\begin{eqnarray*}
	\Gamma(t) = X(\Gamma(0),t).
\end{eqnarray*}
We consider a system which couples the different unknowns mentioned above, and for which a good approximation of the variable $\Phi = \bu^{\pm}_{| \Gamma(t)}$ is essential.

\subsection{\textcolor{black}{A test at low Reynolds number}} \label{subsec-simustokes}
We consider the following system, for all $t \in [0,T]$:
\begin{eqnarray}\label{syscoupled}
\left\{ \begin{array} {rcll}
-\divg \sigma(\bu,p) & = & 0 & \text{in } \Omega^{\pm}(t), \\
\divg \bu & = & 0 &  \text{in } \Omega^{\pm}(t), \\
\bu(\cdot,t) & = & \displaystyle\frac{\p X}{\p t}(X^{-1}(\cdot,t),t) & \text{on } \Gamma(t), 
\label{sysc3}\\[5pt]
\left[ \sigma(\bu,p) \right]\bn & = & -\mu \kappa \bn, & \text{across } \Gamma(t). \label{sysc4}
\end{array} \right.
\end{eqnarray}
The parameter $\mu >0$ is the surface tension parameter, assumed to be constant. The scalar function $\kappa$ denotes the mean curvature of the surface $\Gamma(t)$, related to $X$ through the formula 
\begin{eqnarray*}
\left(\Delta_{\Gamma_0} X\right)\circ X^{-1} & = & \kappa \bn.
\end{eqnarray*}
The notation $\Delta_{\Gamma_0}$ refers to the Laplace-Beltrami operator on the manifold $\Gamma_0 := \Gamma(0)$. The evolution of $\Gamma(t)$ is ruled by the following coupling: \textcolor{black}{the} geometry of $\Gamma(t)$ imposes the jump condition in the fourth equation of~\eqref{sysc4}. The response of the surrounding fluid is the trace of the velocity $\bu^{\pm}$ on $\Gamma(t)$. It determines in the third equation of~\eqref{sysc3} the time-derivative of the deformation $X$, and thus the evolution of $\Gamma(t)$.

We restrain the set of deformations $X$ to the case of an ellipsoid centered at the point of coordinates $x_{c,i} = 0.5$, for $i = 1\dots d$. \textcolor{black}{We denote $\mathrm{y} = (y_i)_{i=1\dots d}$ the space variable in the reference configuration, and $\mathrm{x} = (x_i)_{i=1\dots d}$ the one in the deformed configuration}. The deformation is parameterized by its semi-axes $(a_i)_{i=1\dots d}$, and its expression is explicit, given by
\begin{eqnarray*}
X(\mathrm{y},t) & = & \left(x_{c,i} + (y_i-x_{c,i})\frac{a_i(t)}{a_i(0)} \right)_{i=1\dots d}.
\end{eqnarray*}
We calculate easily
\begin{eqnarray}
\frac{\p X}{\p t}(X^{-1}(\mathrm{x},t),t) & = & 
\left((x_i - x_{c,i})\frac{a_i'(t)}{a_i(t)} \right)_{i=1\dots d}. \label{eqvelodisc}
\end{eqnarray}

The \textcolor{black}{interface} is implemented with a level-set function, whose expression is given by
\begin{eqnarray*}
\ell s(\mathrm{x}) & = & -1 + \sum_{i=1}^d \left( \frac{x_i-x_{c,i}}{a_i} \right)^2.
\end{eqnarray*}
\textcolor{black}{In practice, this function is approximated with piecewise polynomial functions of degree 2 (as mentioned in section~\ref{sec-libimpl}), and thus this approximation is exact in that case.} Remind that the outward unit normal vector is given by $\bn = \nabla \ell s/ |\nabla \ell s|_{\R^d}$, and the curvature of the ellipse obeys the formula $\kappa = -\divg \bn$, leading to the \textcolor{black}{formula $\kappa(\mathrm{x})  =  \frac{-1}{a_1^2a_2^2}\left(\frac{x_1^2}{a_1^4} + \frac{x_2^2}{a_2^4} \right)^{-3/2}$.}

\paragraph{Numerical scheme for the time evolution.}
Denoting by $\Phi(t)$ the multiplier taking into account the jump condition on $\Gamma(t)$ (last equation of~\eqref{sysc4}), we know that its value is the trace on $\Gamma(t)$ of the velocity fields $\bu^{\pm}(\cdot,t)$. Then the third equation of~\eqref{sysc3} becomes $\Phi(t) = \displaystyle \frac{\p X}{\p t}(X^{-1}(\cdot,t),t)$. Combining this equality with~\eqref{eqvelodisc}, we deduce componentwise, for $i=1\dots d$, the equality
\begin{eqnarray*}
\Phi_i(t) & = & (x_i - x_{c,i})\frac{a_i'(t)}{a_i(t)},
\end{eqnarray*}
where $\Phi_i$ denotes the $i$-th component of the vector field $\Phi$. This equality has to be considered in the variational sense. Taking the scalar product of it with the scalar functions $\Phi_i(t)$, for $i=1\dots d$, we deduce
\begin{eqnarray*}
	a_i'(t) \langle x_i-x_{c,i}; \Phi_i\rangle_{\L^2(\Gamma(t))}  & = & 
	a_i(t) \| \Phi_i\|_{\L^2(\Gamma(t))}.
\end{eqnarray*}
From a time-stepping $(t_n)_{n=0\dots N}$ with a constantA test at low Reynolds number time-step $\Delta t$, we discretize this differential equation semi-implicitly, as follows:
\begin{eqnarray*}
	(a_i^{(n+1)} - a_i^n) \langle x_i-x_{c,i}; \Phi_i(t_n)\rangle_{\L^2(\Gamma(t_n))}  & = & 
	(\Delta t )a_i^{(n+1)} \| \Phi_i(t_n)\|_{\L^2(\Gamma(t_n))}.
\end{eqnarray*}
This yields the following scheme:
\begin{eqnarray*}
	a_i^{(n+1)} & = & 
	\left(1 - \Delta t \frac{\|\Phi_i(t_n)|_{\L^2(\Gamma(t_n))}^2}{\langle x_i - x_{c,i};\Phi_i\rangle_{\L^2(\Gamma(t_n))}}\right)^{-1} a_i^{(n)}.
\end{eqnarray*}
We can write $\Phi(t_n) = \mathcal{K}_n(-\mu \kappa(t_n)\bn(t_n))$, where $\mathcal{K}_n$ denotes the Poincar\'e-Steklov operator defined by the solution of system~\eqref{sysjump}, with $\Omega = \Omega(t_n)$ and $\bgg=-\mu \kappa(t_n) \bn(t_n)$. The term $\kappa (t_n)\bn(t_n)$ is entirely given by the geometry of $\Gamma(t_n)$, namely the parameters $(a_i^{(n)})_{i=1\dots d}$. Thus we see the explicit scheme given by
\begin{eqnarray} \label{scheme-Stokes}
	a_i^{(n+1)} & = & 
	\left(1 - \Delta t \frac{\|\mathcal{K}_n(-\mu \kappa(t_n)\bn(t_n))i|_{\L^2(\Gamma(t_n))}^2}{\langle x_i - x_{c,i};\Phi_i\rangle_{\L^2(\Gamma(t_n))}}\right)^{-1} a_i^{(n)}.
\end{eqnarray}
A simulation is performed with \textcolor{black}{Q3/Q2/Q1} elements, a triangular Cartesian mesh with $40$ subdivisions in each space direction. The other parameters are listed in Table~\ref{table1}. 

\begin{remark}
We use the whole stabilization technique (with parameters $\alpha_0$ and $\gamma_0$), even if we were not able to show that it could influence the accuracy for the variable $\Phi$. The goal is to prevent bad potential situations that we were not able to detect in section~\ref{subsec-cvrobust} (like those which occurred for the variables $\blambda^\pm$ in Figure~\ref{fig-robust}), namely situations where the accuracy on the other variables (like $\Phi$) would be also degraded for some geometric configurations. Indeed, theoretically there is {\it a priori} no reason that the convergence is guaranteed for the variable $\Phi$ without stabilization. 
\end{remark}

\begin{table}
\begin{center}
\begin{eqnarray*}
& \begin{array} {|c|c|c|c|c|c|c|c|c|}
\hline
 \nu^+ & \nu^- & \mu & a_1^{(0)} & a_2^{(0)} & T & \Delta t & \alpha_0 & \gamma_0\\
\hline 
 0.1 & 0.05 & 50 & 0.3537 & 0.2037 & 0.1 & 0.00025 & 0.01 & 0.01\\
\hline 
\end{array} &
\end{eqnarray*}
\vspace*{-10pt}
	\caption{Simulation parameters, for the Stokes model.}
	\label{table1}
\end{center}
\end{table}
\FloatBarrier

\begin{center}
	\hspace*{-10pt}\includegraphics[trim = 0cm 0.5cm 0cm 1.5cm, clip, scale=0.35]{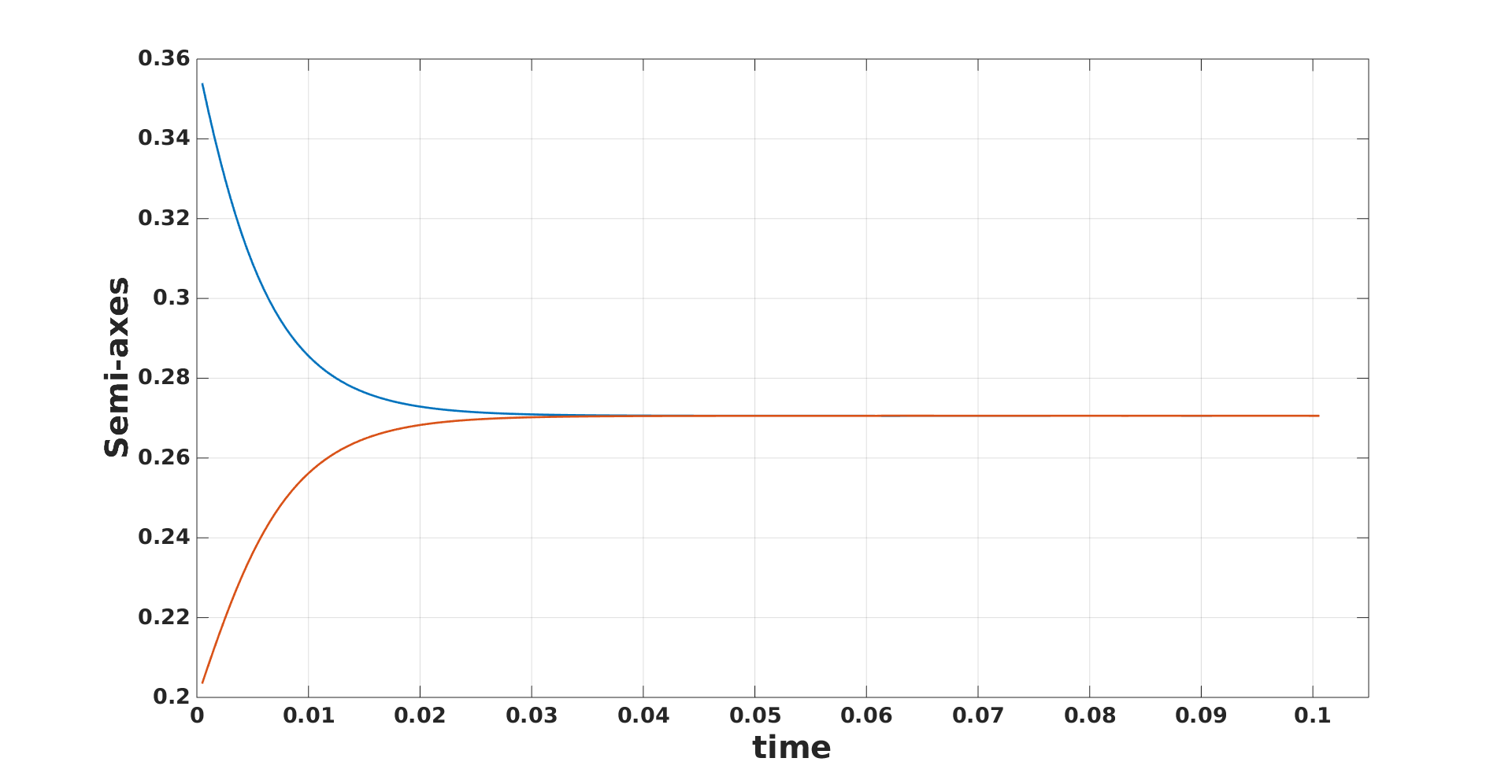}
	\vspace*{-30pt}
	\begin{figure}[H]
		\caption{Time evolution of the semi-axes of the ellipse in dimension 2, for the Stokes model.}
		\label{fig-ellipse-Stokes}
	\end{figure}
\end{center}
\FloatBarrier
Figure~\ref{fig-ellipse-Stokes} shows that the ellipse converges to the circle. Indeed, its semi-axes tend monotonously to a constant value. The potential energy of the interface, namely $\mu |\Gamma(t)|$, is dissipated by the viscosity forces as follows:
\begin{eqnarray*}
	\frac{\d}{\d t}\left( \mu\left|\Gamma(t)\right| \right) + 2\nu^+\|\varepsilon(\bu^+(\cdot,t))\|^2_{\left[\L^2(\Omega^+(t))\right]^{d\times d}} 
	+ 2\nu^-\|\varepsilon(\bu^-(\cdot,t)) \|^2_{\left[\L^2(\Omega^-(t))\right]^{d\times}} & = &  0, \qquad \forall t >0.
\end{eqnarray*}
Thus the area of $\Gamma(t)$ tends to a minimal value, that we know to be corresponding to the sphere. 

\subsection{\textcolor{black}{Application to the Navier-Stokes model}} \label{sec-NSE2D}
We test the capacities of the method when the inertia forces are not neglected in comparison with the viscosity forces. We replace the first equation of the Stokes system by the Navier-Stokes Equation:
\begin{eqnarray*}
	\rho\left(\frac{\p \bu}{\p t} + (\bu \cdot \nabla) \bu\right) -\divg \sigma(\bu,p)  =  0.
\end{eqnarray*}
The density is denoted by $\rho^\pm$, and is chosen to be constant in $\Omega^+(t)$ and $\Omega^-(t)$. The other equations of system~\eqref{syscoupled} remain the same. The chosen parameters are listed in Table~\ref{table2}. 
\begin{table}
\begin{center}
\begin{eqnarray*}
& \begin{array} {|c|c|c|c|c|c|c|c|c|c|c|}
	\hline
	\rho^+ & \rho^- & \nu^+ & \nu^- & \mu & a_1^{(0)} & a_2^{(0)} & T & \Delta t & \alpha_0 & \gamma_0\\
	\hline 
	0.2 & 0.1 & 0.1 & 0.05 & 50 & 0.3537 & 0.2037 & 0.1 & 0.00025 & 0.01 & 0.01\\
	\hline 
\end{array} &
\end{eqnarray*}
\vspace*{-10pt}
		\caption{Simulation parameters, for the Navier-Stokes model in 2D.}
		\label{table2}
\end{center}
\end{table}
\FloatBarrier
\textcolor{black}{The time-derivative of $\bu$ is discretized with the implicit Euler scheme, but on the explicit domains $\Omega^\pm(t_n)$. The nonlinear term $(\bu \cdot \nabla) \bu$ is treated with a Newton method:} 
\begin{eqnarray*}
	\left\{ \begin{array} {rcll}
\rho\left(\displaystyle \frac{\bu^{(n+1)}-\bu^{(n)}}{\Delta t} + \left(\bu^{(n+1)} \cdot \nabla\right) \bu^{(n+1)}\right) -\divg \sigma\left(\bu^{(n+1)},p^{(n+1)}\right) & = & 0 &  \text{in } \Omega^\pm(t_n), \\
\divg \bu^{(n+1)} & = & 0 & \text{in } \Omega^{\pm}(t_{n}), \\
\left[ \sigma(\bu^{(n+1)},p^{(n+1)}) \right]\bn & = & -\mu \kappa(t_n) \bn(t_n), & \text{across } \Gamma(t_{n}).
\end{array} \right.
\end{eqnarray*}
Note that some degrees of freedom of $\bu^{(n+1)}$, corresponding to nodes of $\Omega^{+}(t_n)$ for instance, can be next considered in $\Omega^-(t_{n+1})$ for the next time step (and conversely). These nodes are those which are concerned by the update of the \textcolor{black}{interface} $\Gamma(t_n)$ into $\Gamma(t_{n+1})$, and the updates of the different matrices of the system require only the re-assembly of the terms whose indexes correspond to these nodes (see section~\ref{sec-smartupdate}). For updating $\Gamma(t_n)$ (and thus $\Omega^\pm(t_n)$), we still use the scheme adopted in section~\ref{subsec-simustokes}.

\begin{minipage}{\linewidth}
	\hspace{-0.05\linewidth}%
	\centering
	\begin{minipage}{0.32\linewidth}
		\begin{figure}[H]
			\includegraphics[trim = 16cm 3cm 12cm 3cm, clip, scale=0.18]{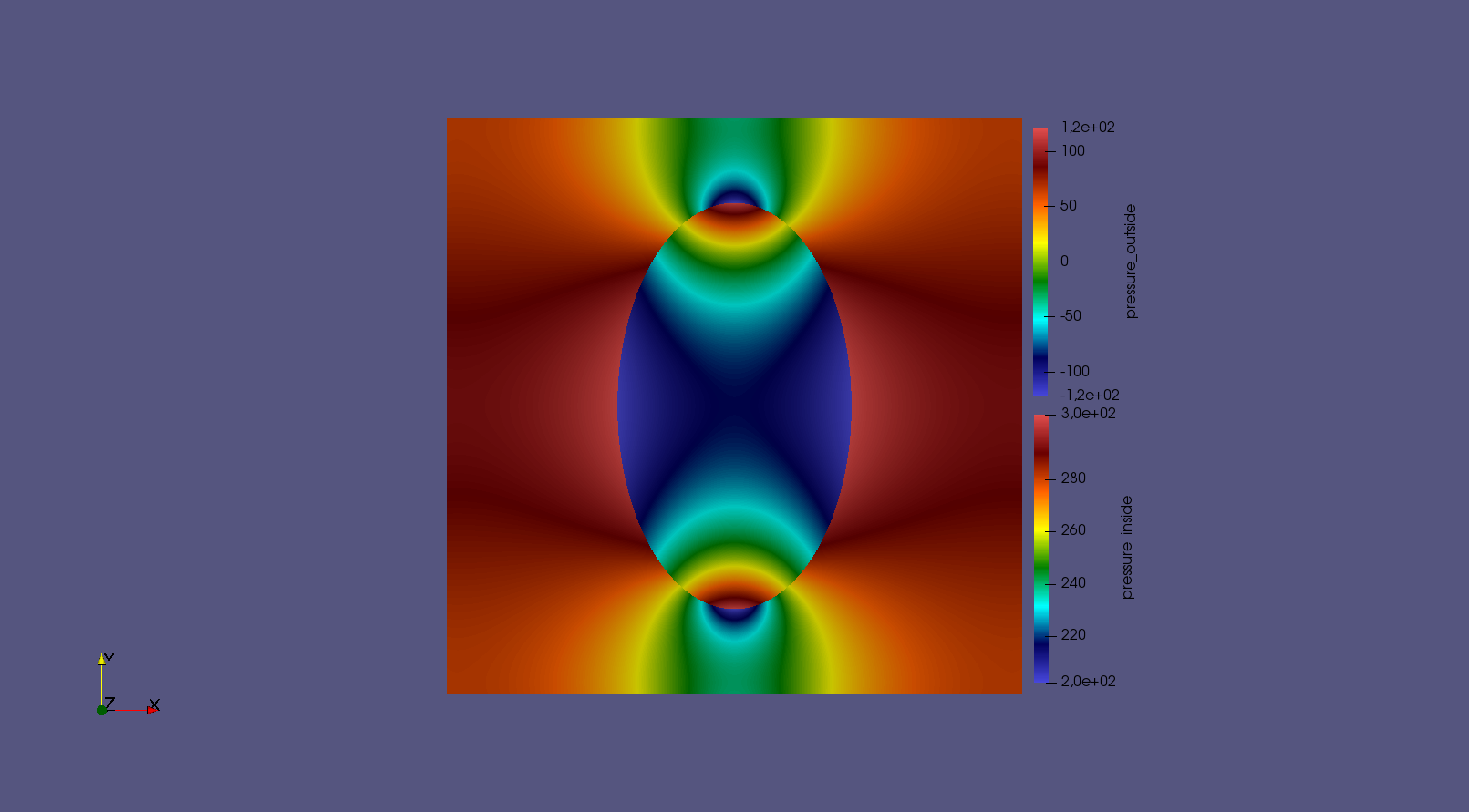}
			\begin{center}\begin{small} $ t = 0.0 $ \end{small}\end{center}
		\end{figure}
	\end{minipage}
	\hspace{0.0\linewidth}
	\begin{minipage}{0.32\linewidth}
		\begin{figure}[H]
			\includegraphics[trim = 16cm 3cm 12cm 3cm, clip, scale=0.18]{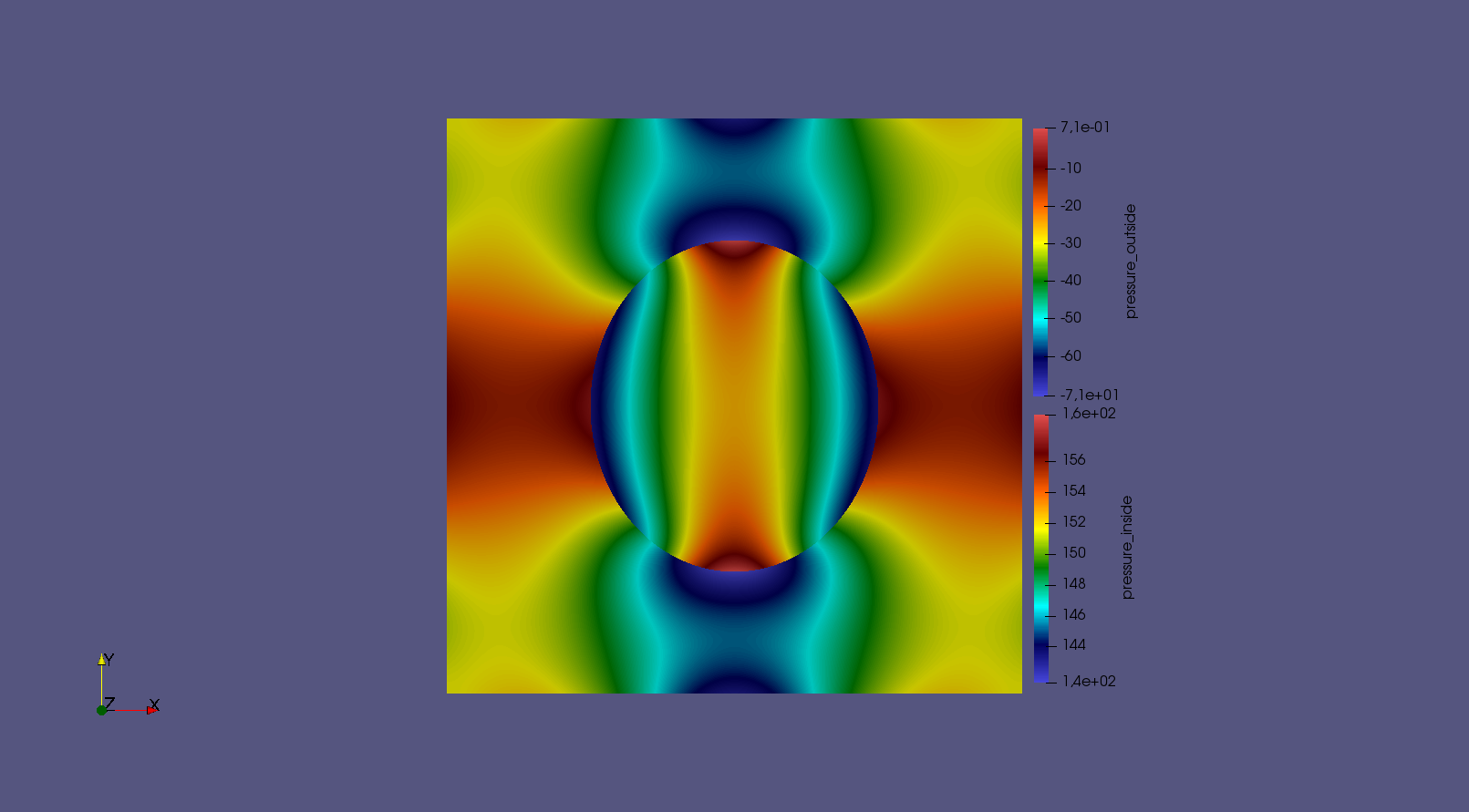}
			\begin{center}\begin{small} $ t = 0.00675 $ \end{small}\end{center}
		\end{figure}
	\end{minipage}
	\hspace{0.0\linewidth}
	\begin{minipage}{0.32\linewidth}
		\begin{figure}[H]
			\includegraphics[trim = 16cm 3cm 12cm 3cm, clip, scale=0.18]{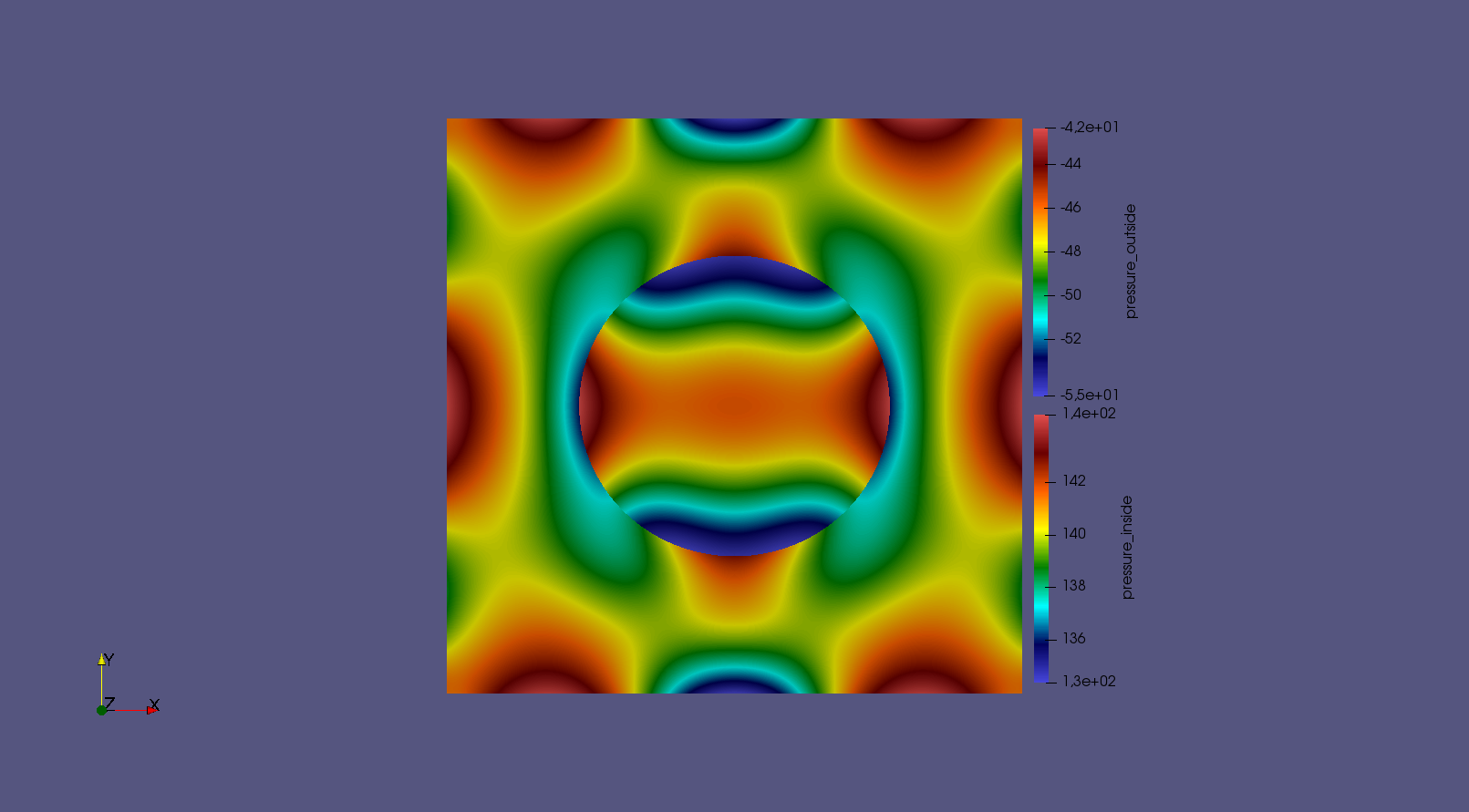}
			\begin{center}\begin{small} $ t = 0.00975 $ \end{small}\end{center}
		\end{figure}
	\end{minipage}
\\ 
\hspace{-0.05\linewidth}%
	\centering
	\begin{minipage}{0.32\linewidth}
		\begin{figure}[H]
			\includegraphics[trim = 16cm 3cm 12cm 3cm, clip, scale=0.18]{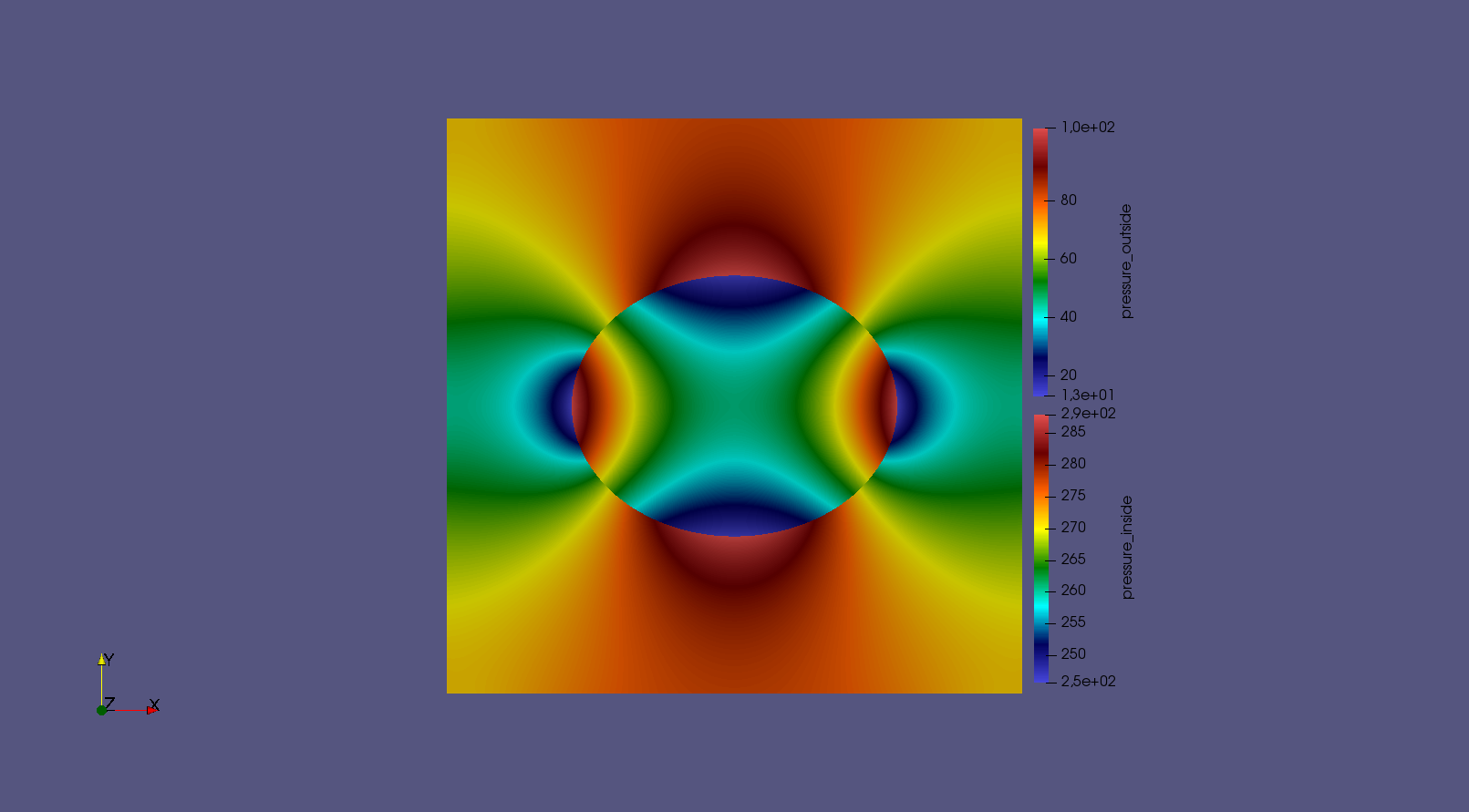}
			\begin{center}\begin{small} $ t = 0.0135 $ \end{small}\end{center}
		\end{figure}
	\end{minipage}
	\hspace{0.0\linewidth}
	\begin{minipage}{0.32\linewidth}
		\begin{figure}[H]
			\includegraphics[trim = 16cm 3cm 12cm 3cm, clip, scale=0.18]{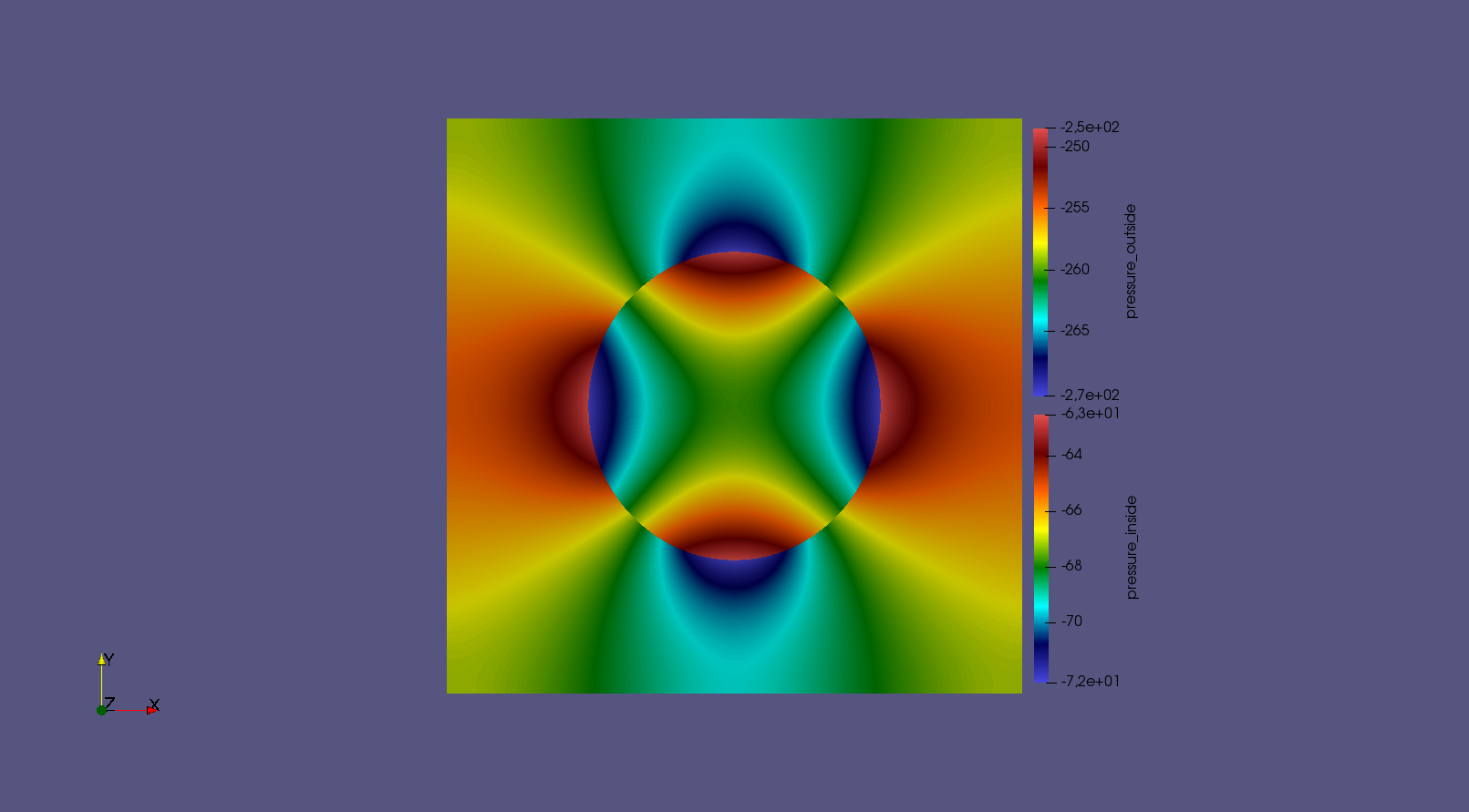}
			\begin{center}\begin{small} $ t = 0.02725 $ \end{small}\end{center}
		\end{figure}
	\end{minipage}
	\hspace{0.0\linewidth}
	\begin{minipage}{0.32\linewidth}
		\begin{figure}[H]
			\includegraphics[trim = 16cm 3cm 12cm 3cm, clip, scale=0.18]{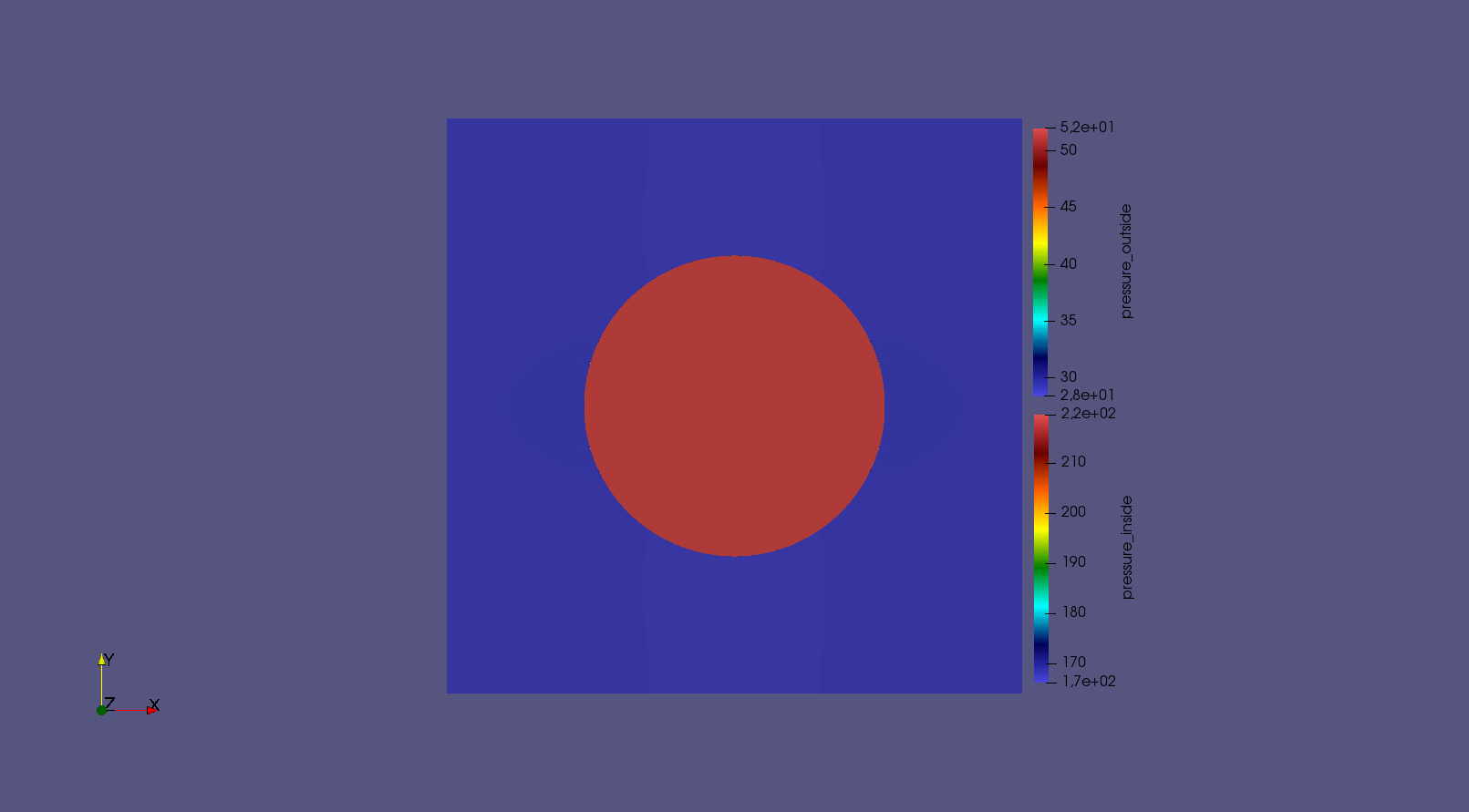}
			\begin{center}\begin{small} $ t = 0.08 $ \end{small}\end{center}
		\end{figure}
	\end{minipage}
\end{minipage}
\begin{minipage}{\linewidth}
	\begin{figure}[H]
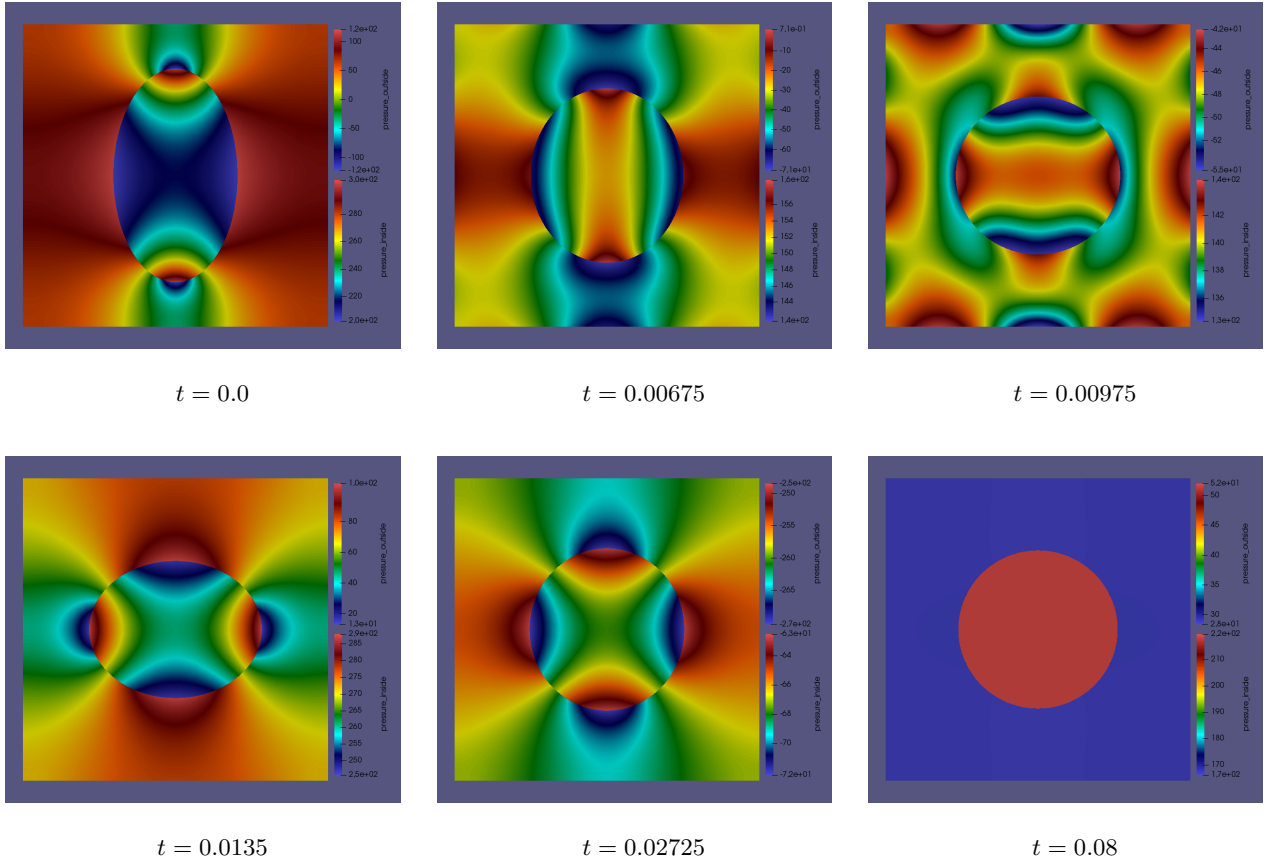

		\caption{Time evolution of the pressure outside and inside the interface for the Navier-Stokes model with surface tension forces. \textcolor{white}{}\label{fig-pressureNSE}}
	\end{figure}
\end{minipage}\\
\FloatBarrier

\begin{minipage}{\linewidth}
	\hspace{-0.05\linewidth}%
	\centering
	\begin{minipage}{0.32\linewidth}
		\begin{figure}[H]
			\includegraphics[trim = 16cm 3cm 12cm 3cm, clip, scale=0.18]{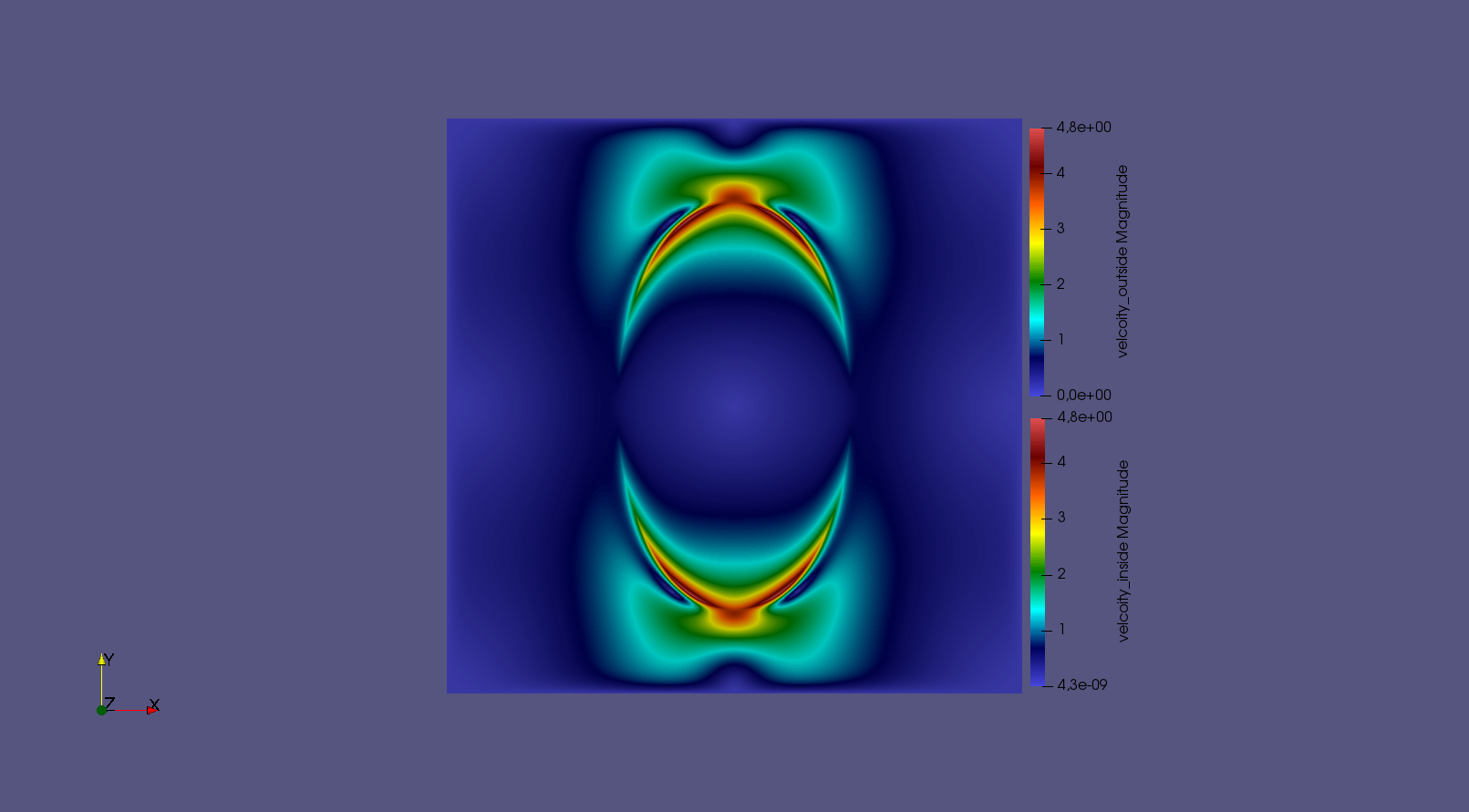}
			\begin{center}\begin{small} $ t = 0.0 $ \end{small}\end{center}
		\end{figure}
	\end{minipage}
	\hspace{0.0\linewidth}
	\begin{minipage}{0.32\linewidth}
		\begin{figure}[H]
			\includegraphics[trim = 16cm 3cm 12cm 3cm, clip, scale=0.18]{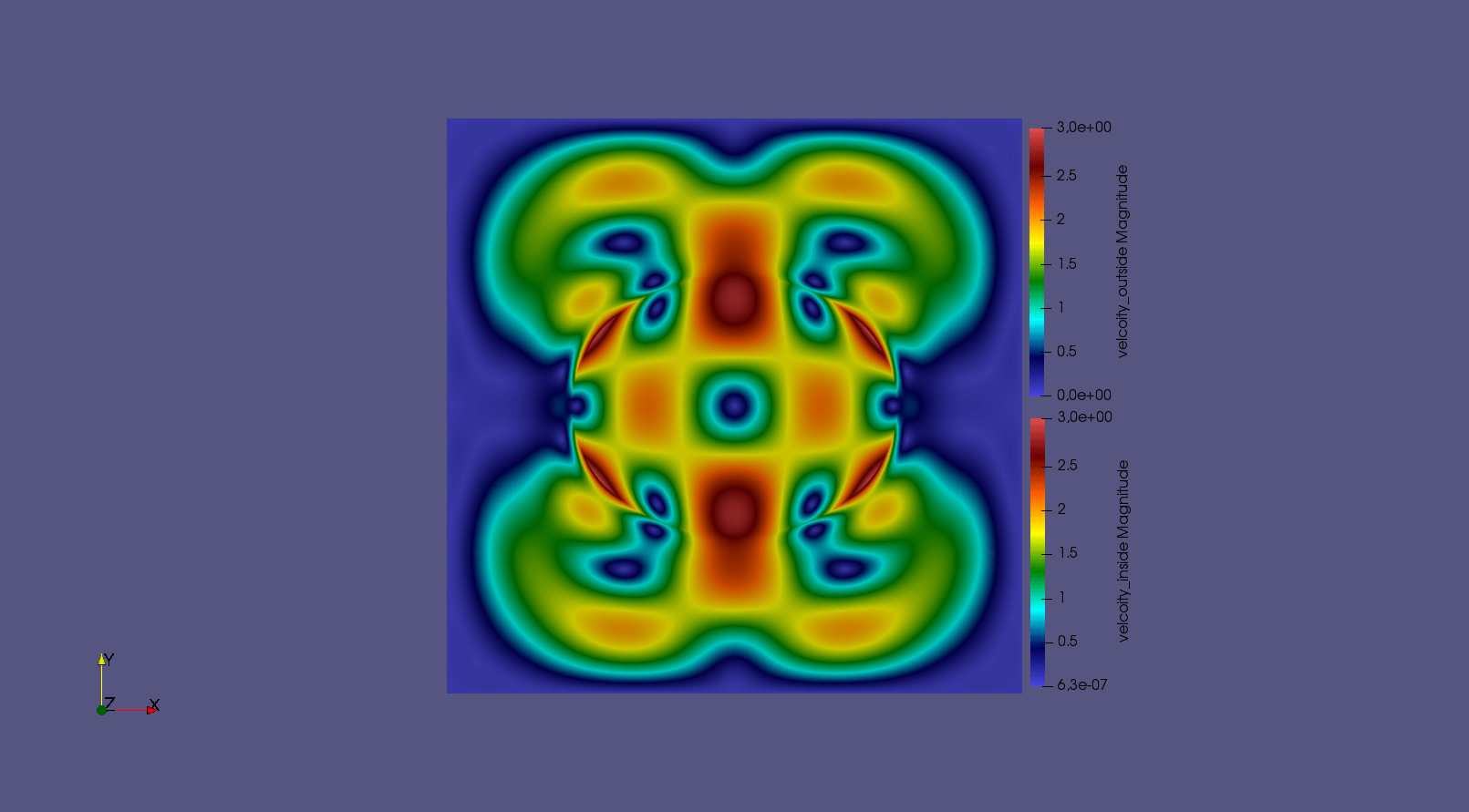}
			\begin{center}\begin{small} $ t = 0.00675 $ \end{small}\end{center}
		\end{figure}
	\end{minipage}
	\hspace{0.0\linewidth}
	\begin{minipage}{0.32\linewidth}
		\begin{figure}[H]
			\includegraphics[trim = 16cm 3cm 12cm 3cm, clip, scale=0.18]{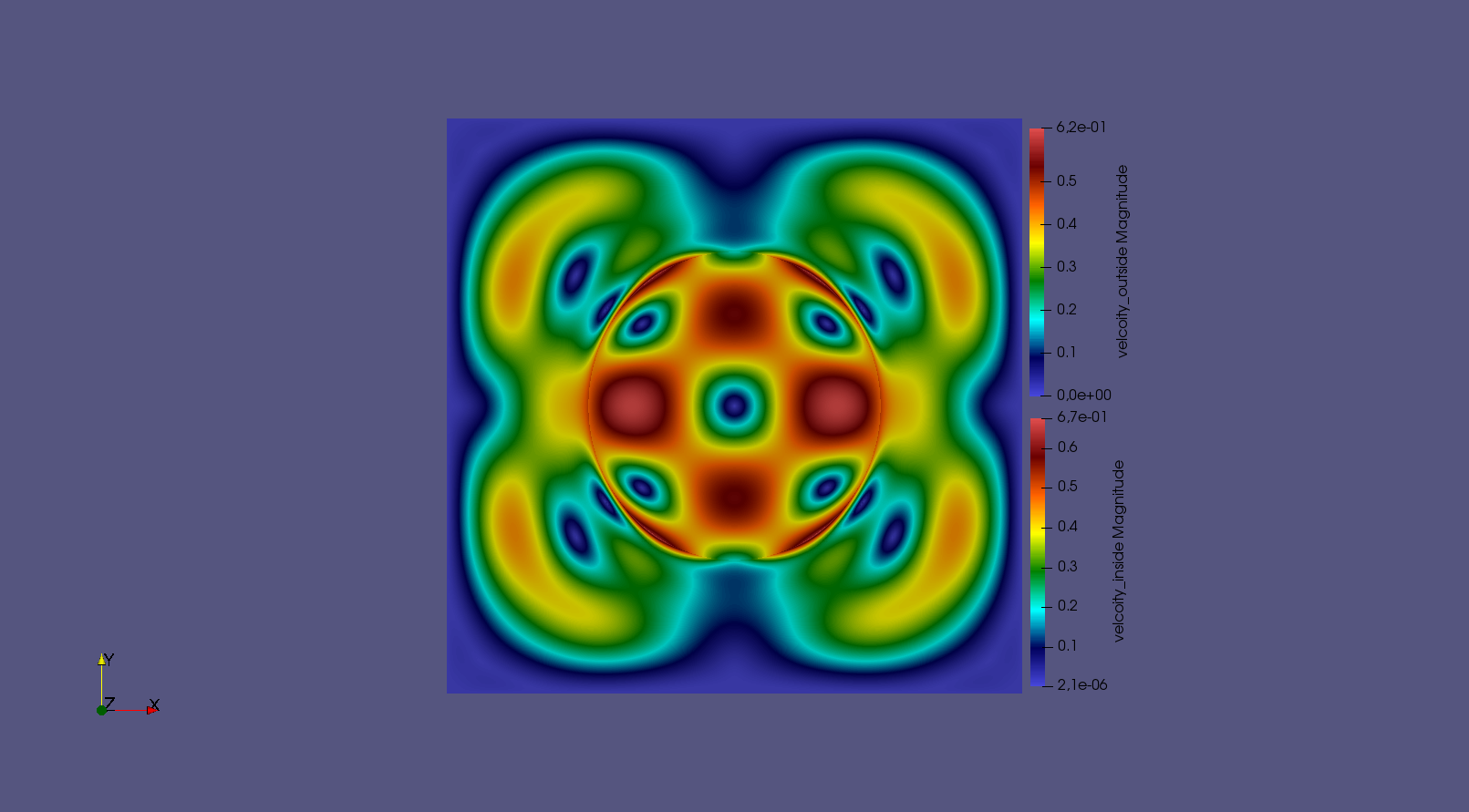}
			\begin{center}\begin{small} $ t = 0.00975 $ \end{small}\end{center}
		\end{figure}
	\end{minipage}
\\ 
\hspace{-0.05\linewidth}%
	\centering
	\begin{minipage}{0.32\linewidth}
		\begin{figure}[H]
			\includegraphics[trim = 16cm 3cm 12cm 3cm, clip, scale=0.18]{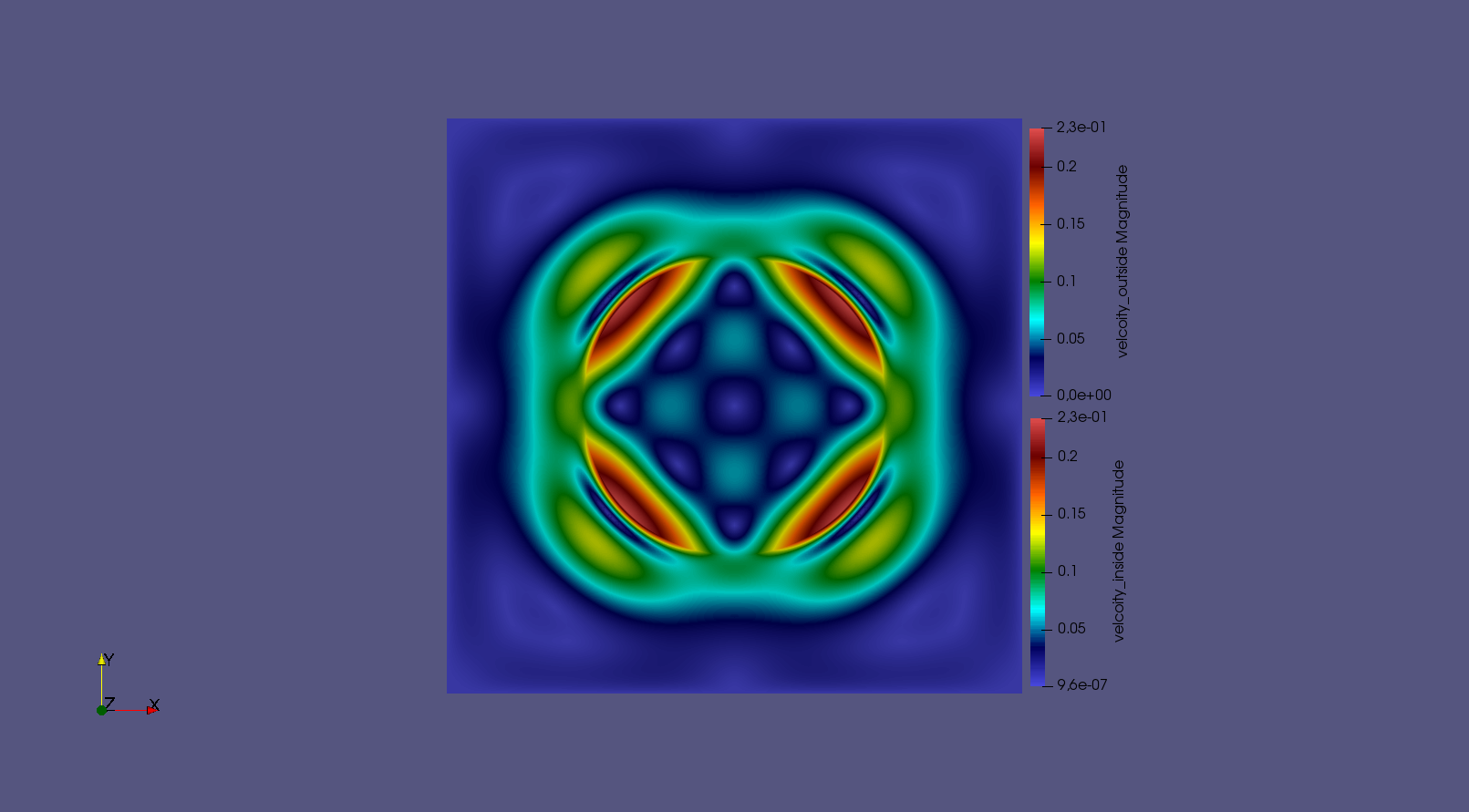}
			\begin{center}\begin{small} $ t = 0.0135 $ \end{small}\end{center}
		\end{figure}
	\end{minipage}
	\hspace{0.0\linewidth}
	\begin{minipage}{0.32\linewidth}
		\begin{figure}[H]
			\includegraphics[trim = 16cm 3cm 12cm 3cm, clip, scale=0.18]{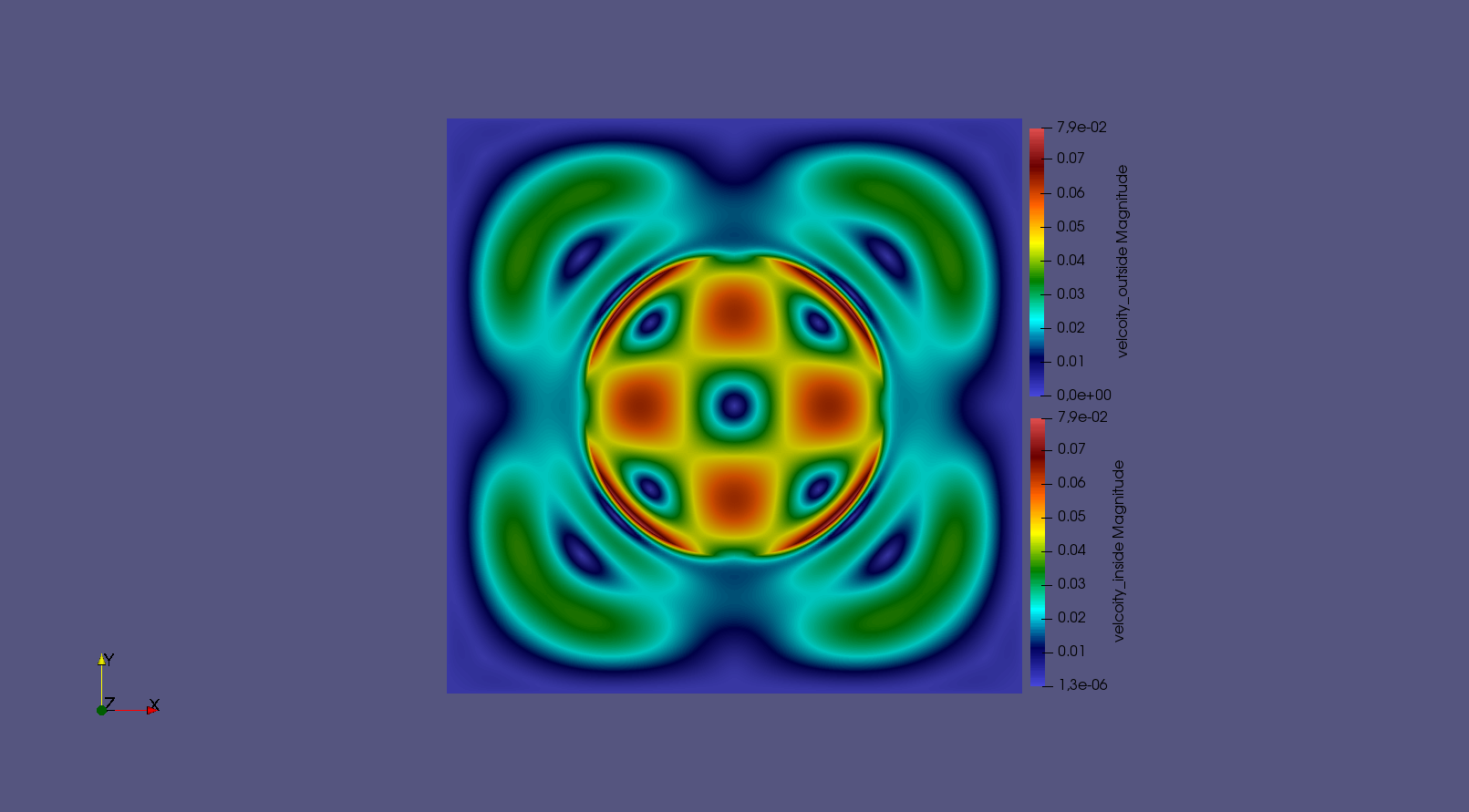}
			\begin{center}\begin{small} $ t = 0.02725 $ \end{small}\end{center}
		\end{figure}
	\end{minipage}
	\hspace{0.0\linewidth}
	\begin{minipage}{0.32\linewidth}
		\begin{figure}[H]
			\includegraphics[trim = 16cm 3cm 12cm 3cm, clip, scale=0.18]{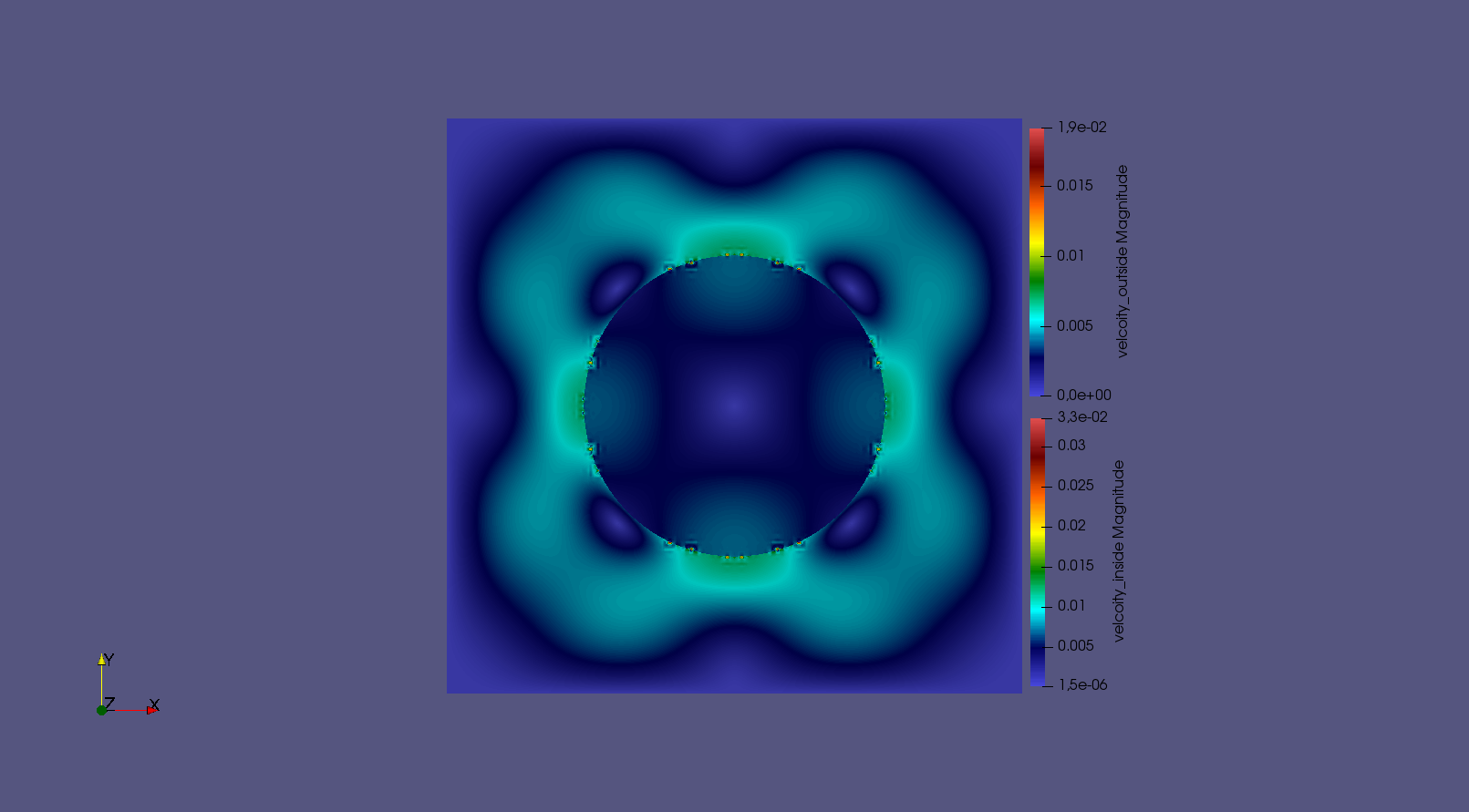}
			\begin{center}\begin{small} $ t = 0.08 $ \end{small}\end{center}
		\end{figure}
	\end{minipage}
\end{minipage}
\begin{minipage}{\linewidth}
	\begin{figure}[H]
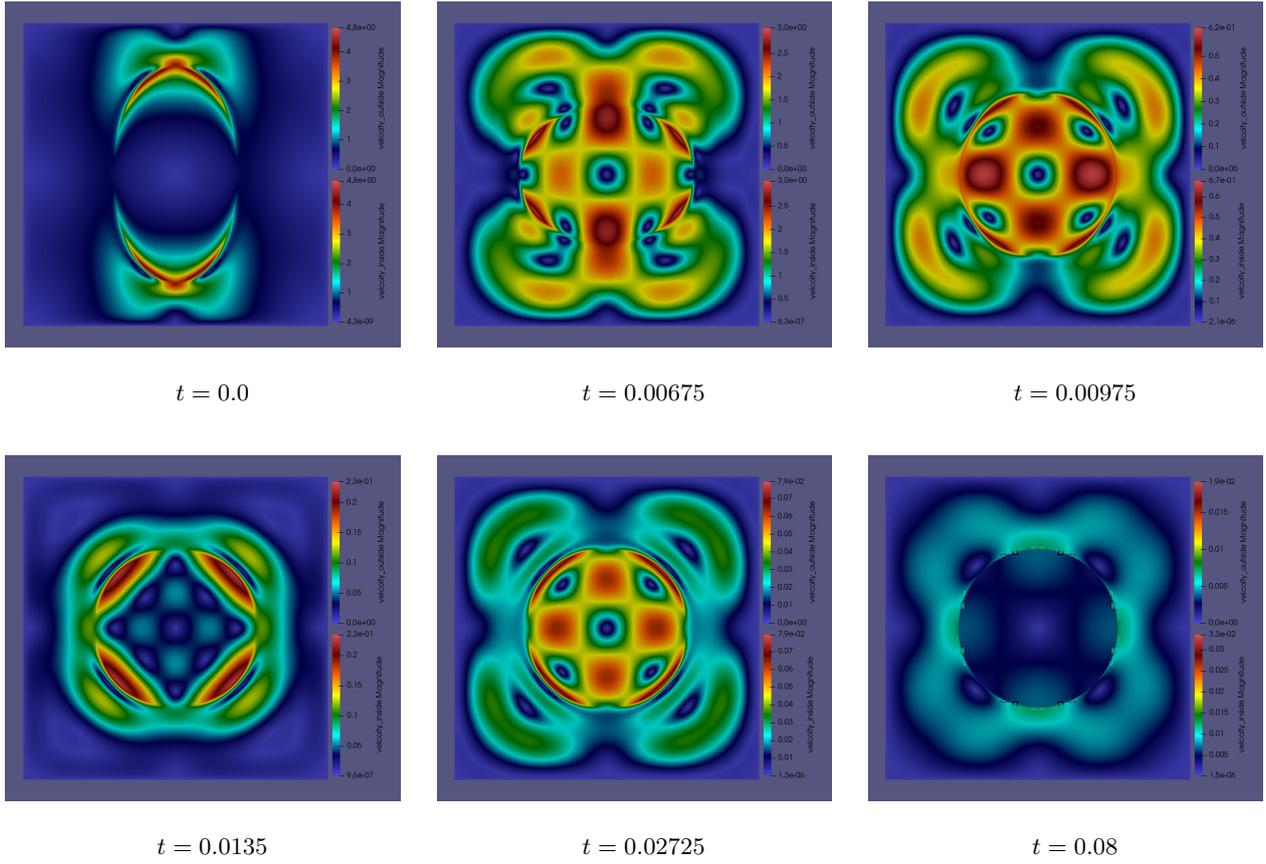

		\caption{Time evolution of the velocity outside and inside the interface for the Navier-Stokes model with surface tension forces. \textcolor{white}{}\label{fig-velocityNSE}}
	\end{figure}
\end{minipage}\\
\FloatBarrier

\begin{center}
\hspace*{-10pt}\includegraphics[trim = 0cm 0.5cm 0cm 1.5cm, clip, scale=0.35]{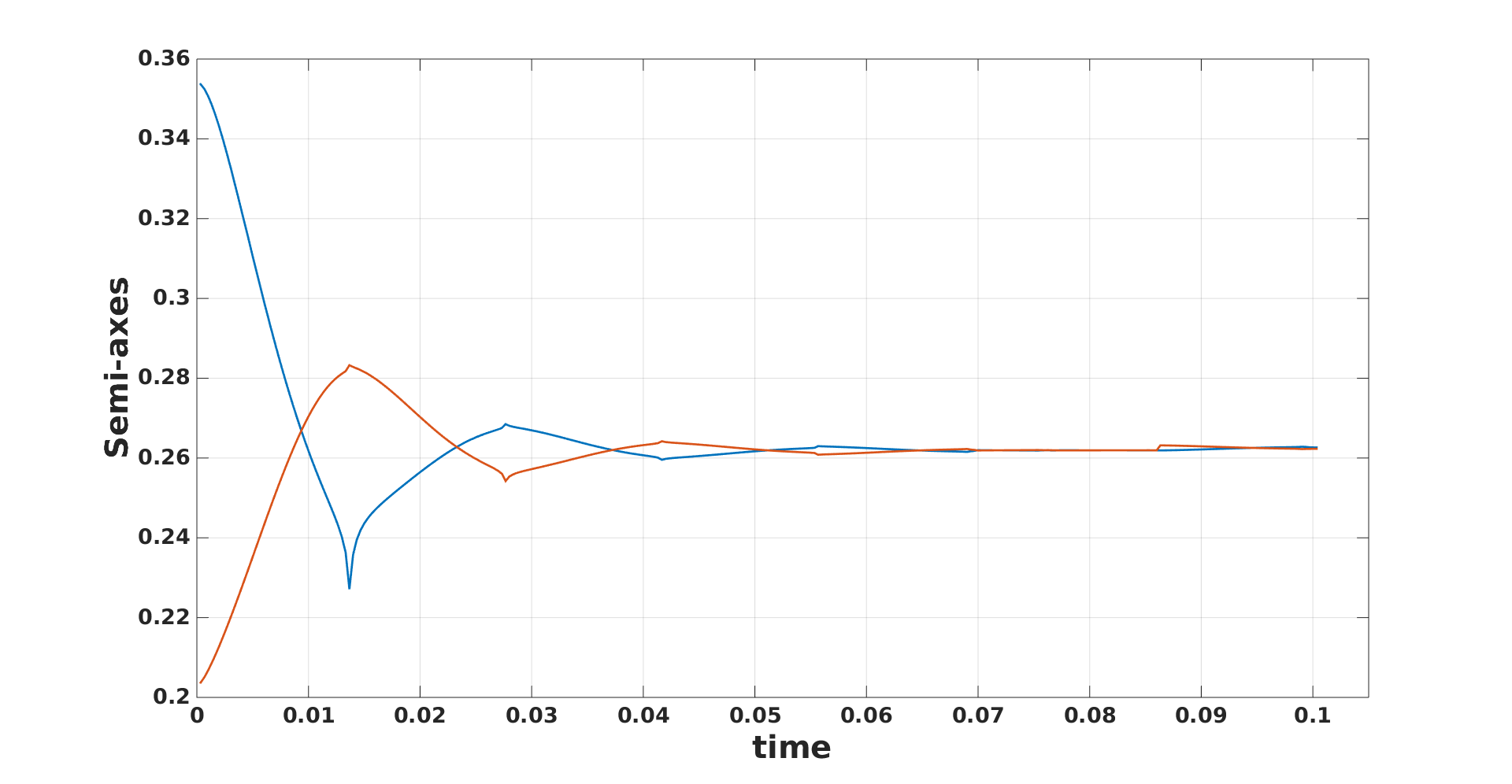}
\begin{figure}[H]
	\vspace*{-30pt}\caption{Time evolution of the semi-axes of the ellipse in dimension 2, for the Navier-Stokes model.}
	\label{fig-ellipse-NSE}
\end{figure}
\end{center}
\FloatBarrier
\textcolor{black}{The respective behaviors of the pressure and the velocity are represented in Figure~\ref{fig-pressureNSE} and Figure~\ref{fig-velocityNSE}. In Figure~\ref{fig-pressureNSE} we can observe the discontinuity of the pressure and its convergence to constant values. In Figure~\ref{fig-velocityNSE} we observe the velocity converges to zero, up to slight artifacts which appear after a long time. The time evolution of the semi-axes of the ellipsoid is represented in Figure~\ref{fig-ellipse-NSE}.} In comparison with Figure~\ref{fig-ellipse-Stokes}, Figure~\ref{fig-ellipse-NSE} shows inertia effects in the convergence of the semi-axes to a constant value. The decrease of the energy function is given by the \textcolor{black}{inequality}
\begin{eqnarray*}
	 \frac{\d}{\d t}\left(\frac{\rho^+}{2} \| \bu^+(\cdot,t)\|^2_{\mathbf{L}^2(\Omega^+(t))} + \frac{\rho^-}{2} \| \bu^-(\cdot,t)\|^2_{\mathbf{L}^2(\Omega^-(t))} + \mu\left|\Gamma(t)\right| \right) 
& \leq & 0,
\qquad \forall t >0.
\end{eqnarray*}
The conservation of mass is studied in Figure~\ref{Fig-volNSE}, where we plot the relative error of mass conservation (during the time iterations), namely the quantity defined by
\begin{eqnarray*}
100.0*\frac{|M(t_n)-M(t_0=0)|}{M(t_0)}, & & \text{with } M(t_n) = \rho^- \pi a_n^{(1)}a_n^{(2)}.
\end{eqnarray*}
This quantity represents the relative error made on the mass conservation.
\begin{center}
	\hspace*{-10pt}\includegraphics[trim = 0cm 0.5cm 0cm 1.5cm, clip, scale=0.35]{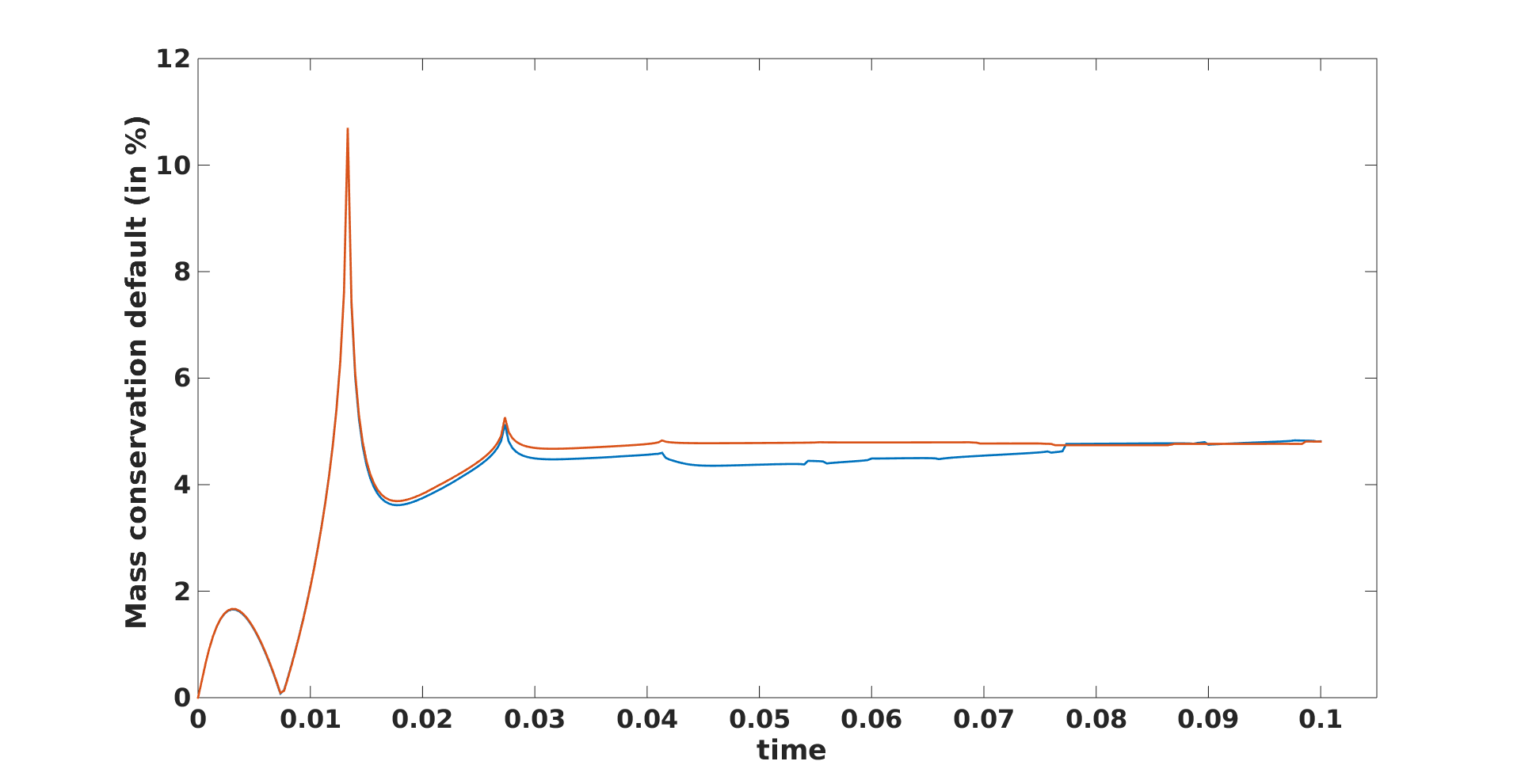}
	\begin{figure}[H]
		\vspace*{-30pt}\caption{Relative mass conservation default, for the Navier-Stokes model, \textcolor{black}{in red for a given mesh (40 subdivisions), in blue for a twice as fine mesh (80 subdivisions)}.}
		\label{Fig-volNSE}
	\end{figure}
\end{center}
\FloatBarrier
We see in Figure~\ref{Fig-volNSE} that the error can reach approximately 10\%. This error could be \textcolor{black}{slightly} improved by choosing more precise elements, \textcolor{black}{but rather by choosing} a smaller time step. However, a more appropriate time-stepping should be performed, in order to obtain numerically the mass conservation, automatically. This property is crucial for more sophisticated simulations, where the goal would be to stick to benchmarks. \textcolor{black}{Note that the mass conservation default is not showed for the Stokes model in section~\ref{subsec-simustokes}, since it is close to zero, and so this mass conservation default is clearly due to the material derivative terms.}

\paragraph{\textcolor{black}{Tests at the equilibrium.}}
\textcolor{black}{Let us test now consider the configuration at the equilibrium. The latter is characterized by a null velocity and a constant pressure on both sides of the interface reduced to a circle. The Laplace-Young law predicts that the difference of pressures must satisfy
\begin{eqnarray*}
\Delta p:= p^- - p^+ & = & \mu / r,
\end{eqnarray*}
where $r$ denotes the radius of the interface. This is observed in Table~\ref{table3} and Figure~\ref{fig-Young-Laplace}, where tests are performed in the same fashion as in~\cite[section~5.1]{Turek2017}.
}\\

\begin{minipage}{\linewidth}
	\hspace{-0.08\linewidth}%
	\centering
	\begin{minipage}{0.48\linewidth}
		\begin{figure}[H]
			\includegraphics[trim = 19cm 10cm 15cm 9cm, clip, scale=0.36]{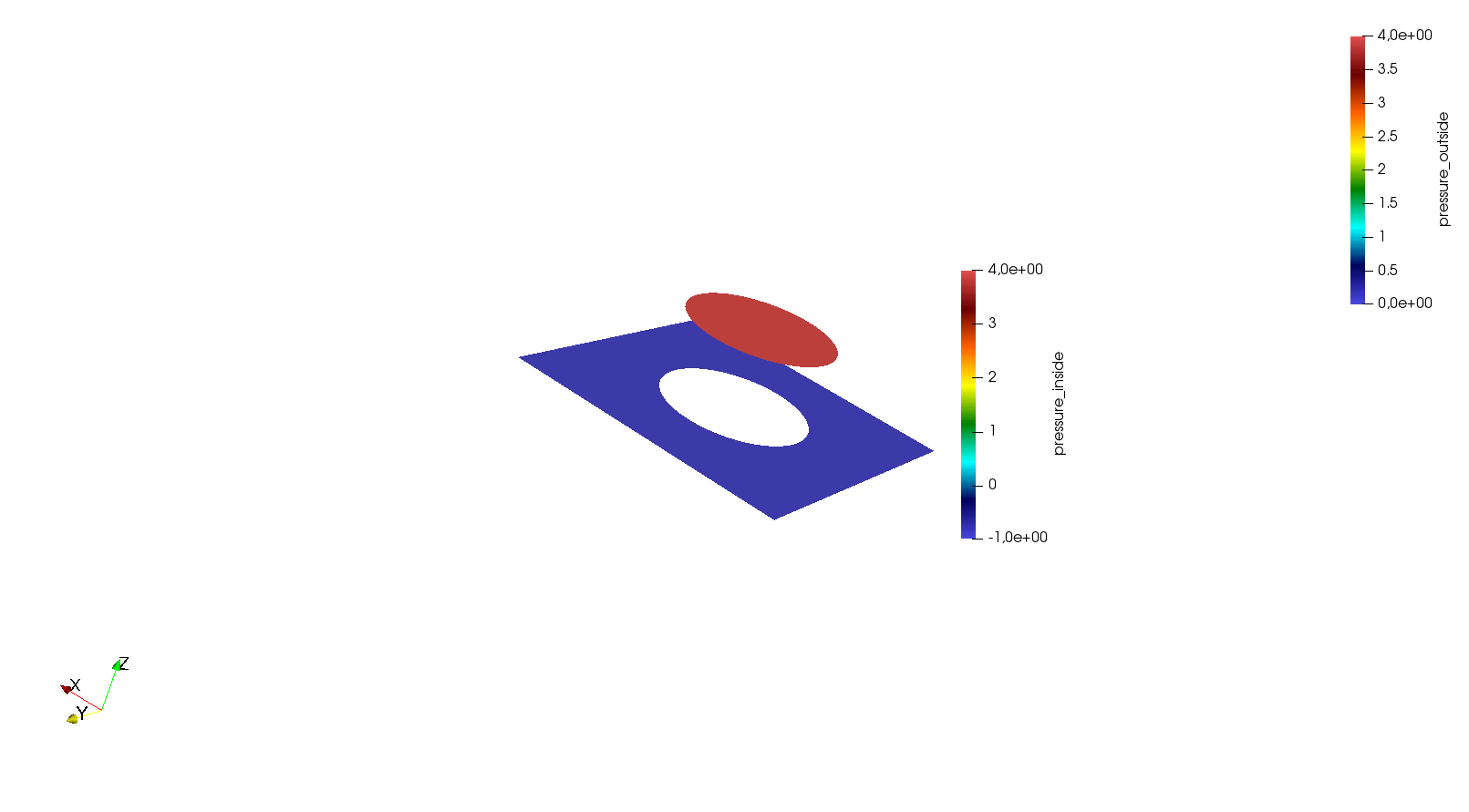}
		\end{figure}
	\end{minipage}
	\hspace{-0.02\linewidth}
	\begin{minipage}{0.48\linewidth}
		\begin{figure}[H]
			\includegraphics[trim = 0cm 0cm 0cm 0cm, clip, scale=0.25]{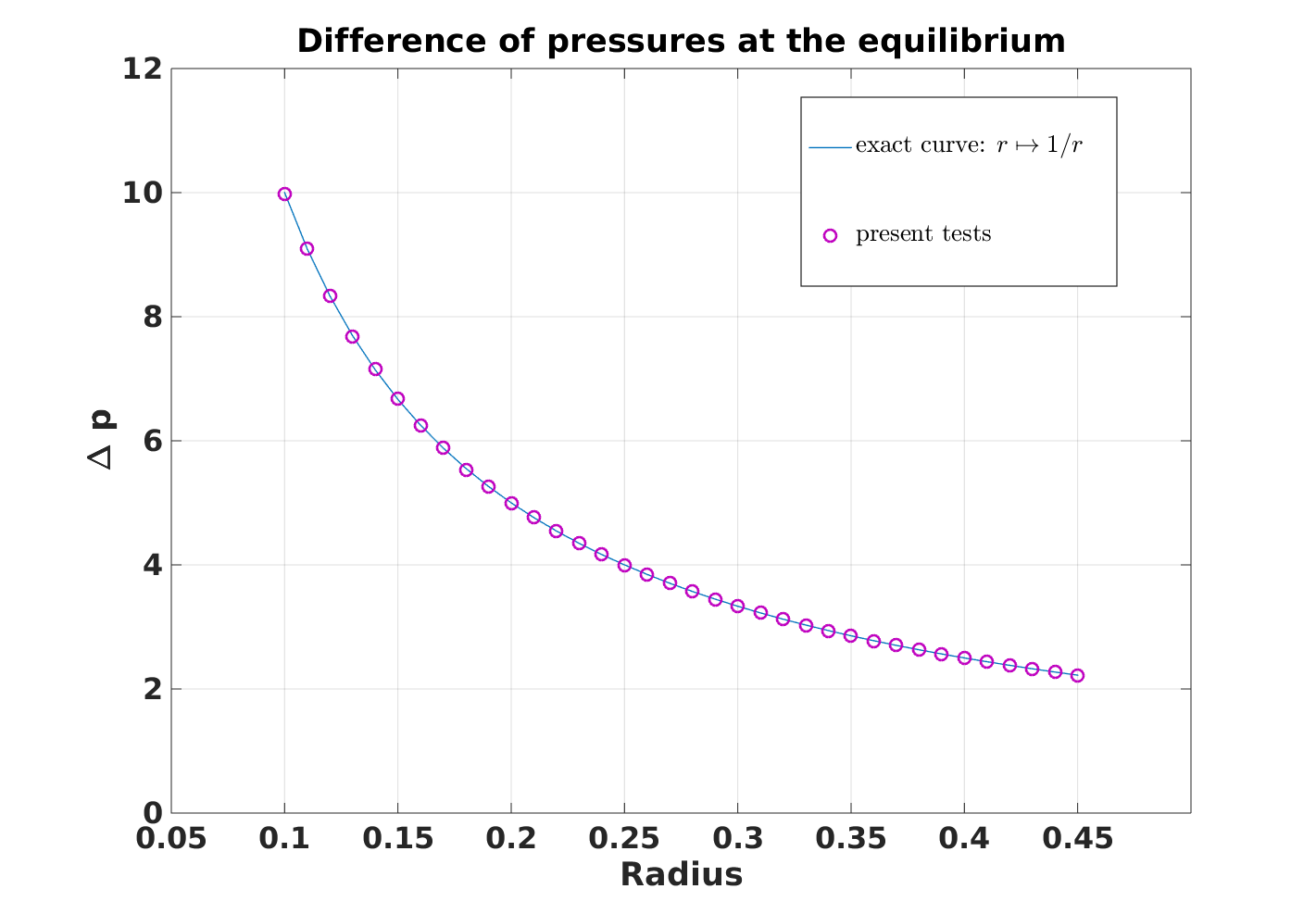}
		\end{figure}
	\end{minipage}
\begin{figure}[H]
	\vspace*{-20pt}\centering\caption{Pressures values computed at the equilibrium for $\mu = 1$ and $h = 0.05$. Left, the radius is 0.25. Right, the difference of pressures is computed for different values of the radius.}
	\label{fig-Young-Laplace}
\end{figure}
\end{minipage}

\begin{table}
\begin{center}
\begin{eqnarray*}
& \begin{array} {|c|c|c|c|c|}
	 h & p^+ & p^- & |\Delta p - \mu/r|/(\mu/r) & |\Delta p |/(\mu/r) \\
	\hline 
	 0.1 & -0.93593 & 3.07232 & 0.00206 & 1.00206 \\
	\hline 
	 0.05 & -0.08705 & 3.90449 & 0.00211 & 0.997885 \\
	\hline 
	 0.025 & -9.9013e-05 & 4.09994 & 0.02501 & 1.02501 \\
	\hline 
	 0.0125 & -0.00268 & 3.99513 & 0.00055 & 0.99945 \\
	\hline 
\end{array} &
\end{eqnarray*}
\vspace*{-10pt}
		\caption{Pressure computations for static bubble, with surface tension coefficient $\mu = 1$ and radius $r = 0.25$.}
		\label{table3}
\end{center}
\end{table}
\FloatBarrier
\hfill \\

\textcolor{black}{In Table~\ref{table3} the static pressures on both sides are presented in function of the mesh size. It is showed that globally the values of the pressures converge to the correct ones, but the difference of pressures is well-computed mainly for the finest mesh. In Figure~\ref{fig-Young-Laplace} we observe that the difference of pressures is well-computed for different values of the radius of the interface.}



\section{Conclusion} \label{sec-conclusion}
In this work we developed a mixed finite element method for the Stokes problem with jump \textcolor{black}{interface} conditions involving surface-tension-type forces. Besides, the \textcolor{black}{interface} is taken into account with a fictitious domain approach, in order to avoid remeshing when this \textcolor{black}{interface} changes. The jump condition is taken into account with a multiplier, whose role is central in coupled problems like simulating the motion of a bubble-soap in a viscous incompressible fluid. The interest of our approach lies in obtaining of a good approximation for this multiplier, while sparing computation time due to remeshing and assembly procedure. Moreover, this fictitious domain method is quite simple to be implemented, and it optimizes the complexity of local treatments for approximating dual variables (like the multiplier aforementioned) that are defined on the \textcolor{black}{interface}. But it requires an augmented Lagrangian approach, in order to prove theoretical convergence for these dual variables, and also to stabilize the approximation of the latter. \textcolor{black}{By this means, this article proposes an alternative to Xfem approaches, by avoiding the use of additional singular basis functions, and thus by simplifying the implementation. Besides}, there is a gap between theoretical analysis and numerical observations, \textcolor{black}{since a part of the stabilization technique, necessary for the theoretical analysis, seems to be unnecessary in practice} for stabilizing one of the multipliers. Without stabilization, our theoretical analysis seems to be too pessimistic, and perhaps convergence properties are hidden in the structure of the Lagrangian functional we rely on. This theoretical point remains to be investigated deeper. For anticipating an unsteady case, the crucial criterion of robustness with respect to the geometry was tested and verified (with the other part of the stabilization technique). Illustrations of time evolution of an ellipsoidal bubble-soap were provided, showing the good behavior of the method, even in the case of inertial effects. A more sophisticated time-stepping, improving the behavior of the method for the model of the Navier-Stokes Equations (conservation of mass for instance), as well as consideration of more general type of deformations are expected in a forthcoming work.

\appendix

\section{Proofs of technical results}

\subsection{Proof of Proposition~\ref{prop-cvnaive}} \label{App-A1}

We remind the following classical result about the theory of saddle-point problems (see~\cite[section~2.4.1]{Ern} for more details).

\begin{lemma} \label{lemma-SPP}
Let $X$ and $M$ be Hilbert spaces, and $A(\cdot,\cdot):X\times X \rightarrow \R$ and $B(\cdot,\cdot):X\times M \rightarrow \R$ be continuous bilinear forms. Assume that $A$ is coercive
\begin{eqnarray*}
	A(w,w) \geq \alpha \|w\|_X^2 & &  \forall w\in X, 
\end{eqnarray*}
and that $B$ satisfies the following inf-sup condition
\begin{eqnarray*}
	\inf_{\pi\in M\setminus\{0\}} \sup_{\pi\in X\setminus\{0\}} \frac{B(w,\pi)}{\|w\|_X \|\pi\|_M}
	& \geq & \beta,
\end{eqnarray*}
for some constant $\alpha, \beta >0$. Then, for all $\gamma \in X'$ and $\delta \in M'$ the problem
\begin{eqnarray*}
	& & \text{Find } w \in X \text{ and } \pi \in M \text{ such that } \\
	& & \left\{ \begin{array} {rcl}
	A(w,v) + B(v,\pi) = \langle \gamma, v \rangle_{X';X} & & \forall v\in X, \\
	B(w,q) = \langle \delta; q \rangle_{M';M} & & \forall q \in M,
	\end{array}
	\right.
\end{eqnarray*}
admits a unique solution, that satisfies
\begin{eqnarray}
	\|w\|_X + \|\pi\|_M & \leq & C\left(\| \gamma \|_{X'} + \|\delta \|_{M'} \right). 
	\label{est-SPP}
\end{eqnarray}
The constant $C$ above depends only on $\alpha$, $\beta$, and the norms of $A$ and $B$.
\end{lemma}

We are now in position to prove a convergence result for the velocity and the pressure.

\begin{proof}[Proof of Proposition~\ref{prop-cvnaive}]
Let $(\bv_h^\pm,q_h^\pm,\bmu_h^\pm) \in \VV^\pm \times \Q_h^\pm \times \WW$ be such that $(\bv^+_h, \bv^-_h) \in \mathbb{V}^0_h$. The difference of the first lines of systems~\eqref{syscont} and~\eqref{sysdisc} gives the equality
\begin{eqnarray*}
a_0^\pm(\bu^\pm_h-\bv^\pm_h,\bw^\pm_h) + b_0^\pm(\bw^\pm_h,p^\pm_h-q^\pm_h) 
+ c_0(\bw^\pm_h, \blambda^\pm_h)  =  
a_0^\pm(\bu^\pm-\bv^\pm_h,\bw^\pm_h)+b_0^\pm(\bw^\pm_h,p^\pm-q^\pm_h)& & \\
 + c_0(\bw^\pm_h, \blambda^\pm), & &
 \forall \bw^\pm_h \in \VV_h^\pm,
\end{eqnarray*}
and so in particular for all $(\bw^+,\bw^-)\in \mathbb{V}^0_h$. Summing the two equalities above, with respect to the symbols $+$ and $-$, gives
\begin{eqnarray*}
& & a_0^+(\bu^+_h-\bv^+_h,\bw^+_h) + a_0^-(\bu^-_h-\bv^-_h,\bw^-_h)
+ b_0^+(\bw^+_h,p^+_h-q^+_h) + b_0^-(\bw^-_h,p^-_h-q^-_h) \\
& &  = a_0^+(\bu^+-\bv^+_h,\bw^+_h)+a_0^-(\bu^--\bv^-_h,\bw^-_h)
+b_0^+(\bw^+_h,p^+-q^+_h)+b_0^-(\bw^-_h,p^--q^-_h) \\
& & \quad + c_0(\bw^+_h, \blambda^+) + c_0(\bw^-_h, \blambda^-)
 - c_0(\bw^+_h, \blambda^+_h) - c_0(\bw^-_h, \blambda^-_h).
\end{eqnarray*}
On the other hand, the difference of the respective fourth equations of systems~\eqref{syscont} and~\eqref{sysdisc} implies
\begin{eqnarray*}
c_0(\bvarphi_h,\blambda_h^+ + \blambda_h^-) = c_0(\bvarphi_h,\blambda^+ + \blambda^-), & &
\forall \bvarphi_h \in \ZZ_h,
\end{eqnarray*}
We can then simplify
\begin{eqnarray*}
c_0(\bw^+_h, \blambda^+) + c_0(\bw^-_h, \blambda^-)
- c_0(\bw^+_h, \blambda^+_h) - c_0(\bw^-_h, \blambda^-_h) & = & 
c_0(\bw_h^- - \bw^+_h,\blambda^-) \\
& = & c_0(\bw_h^- - \bw^+_h,\blambda^-- \bmu_h^-),
\end{eqnarray*}
where we used $c_0(\bw_h^- - \bw^+_h,\bmu_h^-) = 0$, because $(\bw_h^+,\bw_h^-) \in \mathbb{V}^0_h$.
Note that $(\bu^+_h,\bu^-_h)$ lies in $\mathbb{V}^0_h$, and so\newline $(\bu^+_h-\bv_h^+,\bu^-_h-\bv^-_h)$ too. Now consider the following problem:
\begin{eqnarray*}
	& & \text{Find $((\bx^+_h,\bx^-_h),(m^+_h,m_h^-)) \in \mathbb{V}^0_h \times \Q^+_h \times \Q^-_h$} \text{ such that } \\
	& & \left\{ \begin{array} {rcl}
	a_0^+(\bx^+_h,\bw^+_h) + a_0^-(\bx^-_h,\bw^-_h) 
	+ b_0^+(\bw^+_h, m^{+}_h )  + b_0^-(\bw^-_h, m^{-}_h )
	=  \gamma^\pm & & \forall (\bw_h^+,\bw^-_h) \in \mathbb{V}_h^0 ,\\
	b_0^+(\bx_h^+ , r_h^+) + b_0^-(\bx_h^- , r_h^-)
	=  \delta^\pm & & \forall (r_h^+,r_h^-) \in \Q_h^+ \times \Q^-_h.
	\end{array} \right.
\end{eqnarray*}
From estimate~\eqref{est-SPP} of Lemma~\ref{lemma-SPP}, used with $X = \mathbb{V}^0_h$, $M = \Q_h^+ \times \Q_h^-$,
\begin{eqnarray*}
	& & A((\bu_h^+,\bu_h^-),(\bv_h^+,\bv_h^-))  =  a^+_0(\bu_h^+,\bv_h^+) + a_0^-(\bu_h^-,\bv_h^-), \\ 
	& & B((\bu_h^+,\bu_h^-),(q_h^+,q_h^-))  =  b_0^+(\bu_h^+,q_h^+) +b_0^-(\bu_h^-,q_h^-),
\end{eqnarray*}
and $\gamma = \gamma^+ +\gamma^-$, $\delta = \delta^+ + \delta^-$, with
\begin{eqnarray*}
	 \gamma^\pm  =  a_0^\pm(\bu^\pm-\bv^\pm_h,\bw^\pm_h)
	+b_0^\pm(\bw^\pm_h,p^\pm-q^\pm_h)
	\mp c_0(\bw^\pm_h,\blambda^--\bmu^-_h),& & 
	 \delta^\pm = b_0^\pm(\bu^\pm - \bv_h^\pm, r_h^\pm),
\end{eqnarray*}
it yields that
\begin{eqnarray*}
& & \| \bu_h^+-\bv^+_h \|_{\VV^+} + \| \bu^-_h-\bv^-_h \|_{\VV^-}
+ \| p_h^+-q^+_h \|_{\VV^+} + \| p^-_h-q^-_h \|_{\VV^-} \\
& & \leq \| \bu^+-\bv^+_h \|_{\VV^+} + \| \bu^--\bv^-_h \|_{\VV^-}
+ \| p^+-q^+_h \|_{\VV^+} + \| p^--q^-_h \|_{\VV^-}
+  \| \blambda^- - \bmu^-_h\|_{\WW}.
\end{eqnarray*}
Combine this estimate with the triangle inequality, we obtain
\begin{eqnarray*}
& & \| \bu^+-\bu^+_h \|_{\VV^+} + \| \bu^--\bu^-_h \|_{\VV^-}
+ \| p^+-p^+_h \|_{\VV^+} + \| p^- - p^-_h \|_{\VV^-} \\
& & \leq \| \bu^+-\bv^+_h \|_{\VV^+} + \| \bu^--\bv^-_h \|_{\VV^-}
+ \| p^+-q^+_h \|_{\VV^+} + \| p^--q^-_h \|_{\VV^-} 
+ \| \blambda^- - \bmu^-_h\|_{\WW},
\end{eqnarray*}
which is valid for any $(\bv_h^\pm,q_h^\pm,\bmu_h^\pm) \in \VV^\pm \times \Q_h^\pm \times \WW$ such that $(\bv^+_h,\bv^-_h) \in \mathbb{V}_h^0$. In view of the symmetry of this estimate, the term $\| \blambda^- - \bmu^-_h\|_{\WW}$ can be replaced by $\min\left(\| \blambda^+ - \bmu^+_h\|_{\WW};\| \blambda^- - \bmu^-_h\|_{\WW}\right)$, and thus we obtain the announced estimate.
\end{proof}

\subsection{Proof of Theorem~\ref{th-infsup}} \label{App-A2}

We will need an estimate for the $L^2$-orthogonal projection from $\mathbf{H}^{1/2}(\Gamma)$ to $\mathbf{W}_h$. The following result is proved in~\cite[Lemma~3,~page~85]{Court2014}:
\begin{lemma} \label{lemma-projector}
	Denote by $P_h$ the $L^2$-orthogonal projector from $\mathbf{H}^{1/2}(\Gamma)$ to $\mathbf{W}_h$. Then the following estimate
	\begin{eqnarray}
	\| P_h \bv - \bv \|_{\mathbf{L}^2(\Gamma)} \leq C h^{1/2}\|\bv\|_{\mathbf{H}^{1/2}(\Gamma)}
	\label{estproj}
	\end{eqnarray}
	holds for all $\bv \in \mathbf{H}^{1/2}(\Gamma)$.
\end{lemma}

Before starting the proof of Theorem~\ref{th-infsup}, let us modify slightly the functional framework. First, since any function $\bv^+ \in \VV^+$ satisfies $\bv^+_{\mid \p \Omega} \equiv 0$, the operator $\varepsilon$ is injective on $\VV^+$ (see~\cite[page~18]{Temam} for a characterization of its kernel), and then the Petree-Tartar lemma combined with the Korn's equality yields $\| \bv^+\|_{\HH^1(\Omega^+)} \leq C\| \varepsilon(\bv^+) \|_{[\L^2(\Omega^+)]^{d\times d}}$ for all $\bv^+ \in \VV^+$, and thus we can redefine
\begin{eqnarray*}
\| \bv^+ \|_{\VV^+}  :=  \| \varepsilon(\bv^+) \|_{[\L^2(\Omega^+)]^{d\times d}}, 
& &  \forall \bv^+ \in \VV^+.
\end{eqnarray*}
Secondly, we modify the space $\VV^-$ as $\VV^- : =  \left\{ \bv \in \HH^1(\Omega^-) \mid \ \int_{\Gamma_D} \bv \, \d \Gamma_D =0\right\}$, where $\Gamma_D \subset \Gamma$ is of positive $(d-1)$-dimensional Lebesgue measure (like in~\cite{Renard2009}). Thus, in the same fashion as for $\VV^+$ previously, we can redefine
\begin{eqnarray*}
\| \bv^- \|_{\VV^-}  :=  \| \varepsilon(\bv^-) \|_{[\L^2(\Omega^-)]^{d\times d}}, 
& &   \forall \bv^- \in \VV^-.
\end{eqnarray*}
These modifications could be avoided for realizing the following proof, but they simplify the latter. 

\begin{proof}[Proof of Theorem~\ref{th-infsup}]
The technique of the proof is inspired from~~\cite[Lemma~6,~page~144]{Stenberg}. Let $\mathfrak{u}_h = (\bu_h^+,p_h^+,\blambda_h^+,\bu_h^-,p_h^-,\blambda_h^-,\Phi_h) \in \mathfrak{V}_h$.\\
\noindent{\bf Step 1.} 
Choose $\mathfrak{v}_h^{(1)} := (\bu_h^+,-p_h^+,-\blambda_h^+,\bu_h^-,-p_h^-,-\blambda_h^-,\Phi_h)$. Then
\begin{eqnarray}
\mathcal{M}(\mathfrak{u}_h,\mathfrak{v}_h^{(1)}) & = & 2\nu^+ \| \varepsilon(\bu_h^+) \|_{[\L^2(\Omega^+)]^{d\times d}}^2 +  2\nu^- \| \varepsilon(\bu_h^-) \|_{[\L^2(\Omega^-)]^{d\times d}}^2\nonumber\\
& & +\alpha_0  \| \bu_h^+ - \Phi_h\|^2_{\mathbf{L}^2(\Gamma)} 
+\alpha_0  \| \bu_h^- - \Phi_h\|^2_{\mathbf{L}^2(\Gamma)} \nonumber\\
& & + \gamma_0h\| \blambda_h^+ + p_h^+\bn^+\|_{\mathbf{L}^2(\Gamma)}^2
+ \gamma_0h\| \blambda_h^- + p_h^-\bn^-\|_{\mathbf{L}^2(\Gamma)}^2\nonumber\\
& & -4{\nu^+}^2\gamma_0h\|\varepsilon(\bu^+_h)\bn^+\|^2_{\mathbf{L}^2(\Gamma)}
-4{\nu^-}^2\gamma_0h\|\varepsilon(\bu^-_h)\bn^-\|^2_{\mathbf{L}^2(\Gamma)} \nonumber\\
& \geq & 2\nu^+ \| \bu_h^+ \|_{\VV^+}^2 +  2\nu^- \| \bu_h^- \|_{\VV^-}^2\nonumber\\
& & +2\alpha_0  \| \Phi_h\|^2_{\mathbf{L}^2(\Gamma)} 
- C\alpha_0\left(\| \bu_h^+\|_{\VV^+} + \| \bu_h^-\|_{\VV^-}\right) \nonumber\\
& & + \gamma_0h\| \blambda_h^+ + p_h^+\bn^+\|_{\mathbf{L}^2(\Gamma)}^2
+ \gamma_0h\| \blambda_h^- + p_h^-\bn^-\|_{\mathbf{L}^2(\Gamma)}^2\nonumber\\
& & -C\gamma_0\|\bu^+_h\|^2_{\mathbf{V}^+_h}
-C\gamma_0\|\bu^-_h\|^2_{\mathbf{V}^-_h}
, \label{eststep1}
\end{eqnarray}
where we used assumption~$(\mathbf{A2})$.
\hfill \\
{\bf Step 2.} From assumption~$(\mathbf{A1})$, there exists ${\bv_{p}}^{\pm}_h \in \mathbf{V}_0^{\pm}$ such that 
\begin{eqnarray}
-\int p^{\pm}_h\divg {\bv_p}^{\pm}_h \, \d \Omega^{\pm} = \| p_h^{\pm} \|^2_{\L^2(\Omega^{\pm})} 
& \text{ and } &
\| {\bv_p}^{\pm}_h\|_{\mathbf{V}^{\pm}} \leq C\| p_h\|_{\L^2(\Omega^{\pm})}.\label{estvp}
\end{eqnarray}
Choose $\mathfrak{v}_h^{(2)} = ({\bv_p}_h^{+},0,0,{\bv_p}_h^{-},0,0,0)$. Then, for some $\alpha >0$, we estimate
\begin{eqnarray*}
	\mathcal{M}(\mathfrak{u}_h,\mathfrak{v}_h^{(2)}) & = & 2\nu^+ \int_{\Omega^+}\varepsilon(\bu_h^+):\varepsilon({\bv_p}^{+}_h)\, \d \Omega^+ + 
	2\nu^- \int_{\Omega^-}\varepsilon(\bu_h^-):\varepsilon({\bv_p}^{-}_h)\, \d \Omega^- + 
	\| p_h^+\|^2_{\L^2(\Omega^+)} +  \| p_h^-\|^2_{\L^2(\Omega^-)} \\
	& & -2\nu^+\gamma_0h\int_{\Gamma}\varepsilon(\bv_{p_h}^+) \cdot (2\nu^+ \varepsilon(\bu^+)\bn^+ - p^+\bn^+ - \blambda^+) \, \d \Gamma \\
	& & -2\nu^-\gamma_0h\int_{\Gamma}\varepsilon(\bv_{p_h}^-) \cdot (2\nu^- \varepsilon(\bu^-)\bn^- - p^-\bn^- - \blambda^-) \, \d \Gamma\\
	& \geq & \| p_h^+\|^2_{\L^2(\Omega^+)} +  \| p_h^-\|^2_{\L^2(\Omega^-)} - \alpha \nu^+ \| \varepsilon(\bu_h^+) \|_{\mathbf{L}^2(\Omega^+)}^2 -  \alpha \nu^- \| \varepsilon(\bu_h^-) \|_{\mathbf{L}^2(\Omega^-)}^2 \\
	& &  - \frac{\nu^+}{\alpha}\| \varepsilon({\bv_p}^{+}_h) \|_{\mathbf{L}^2(\Omega^+)}^2 - \frac{\nu^-}{\alpha}\| \varepsilon({\bv_p}^{-}_h) \|_{\mathbf{L}^2(\Omega^-)}^2\\
	& & -C\gamma_0 h \left(\frac{{\nu^+}^2}{\alpha}\| \varepsilon({\bv_p}^{+}_h) \|_{\mathbf{L}^2(\Gamma)}^2 + \frac{{\nu^-}^2}{\alpha}\| \varepsilon({\bv_p}^{-}_h) \|_{\mathbf{L}^2(\Gamma)}^2 + 
	\alpha {\nu^+}^2\| \varepsilon({\bu}^{+}_h) \|_{\mathbf{L}^2(\Gamma)}^2
	+ \alpha {\nu^-}^2\| \varepsilon({\bu}^{-}_h) \|_{\mathbf{L}^2(\Gamma)}^2\right) \\
	& & -C\gamma_0 h\left( \| \blambda_h^+ + p_h^+\bn^+\|_{\mathbf{L}^2(\Gamma)}^2 +
	\| \blambda_h^- + p_h^-\bn^-\|_{\mathbf{L}^2(\Gamma)}^2 
	\right) \\
	& \geq & \| p_h^+\|^2_{\L^2(\Omega^+)} +  \| p_h^-\|^2_{\L^2(\Omega^-)} - \alpha \nu^+ \| \bu_h^+ \|_{\VV^+}^2 -  \alpha \nu^- \| \bu_h^- \|_{\VV^-}^2 \\
	& &  - \frac{C\gamma_0}{\alpha}\| \bv^{+}_{p_h} \|_{\mathbf{V}^+}^2 - \frac{C}{\alpha}\| \bv^{-}_{p_h} \|_{\mathbf{V}^-}^2 
	- C\alpha \|\bu^+_h\|^2_{\mathbf{V}^+} - C\alpha \|\bu^-_h\|^2_{\mathbf{V}^-}\\
	& & -C\gamma_0 h\left( \| \blambda_h^+ + p_h^+\bn^+\|_{\mathbf{L}^2(\Gamma)}^2 +
	\| \blambda_h^- + p_h^-\bn^-\|_{\mathbf{L}^2(\Gamma)}^2 
	\right).
\end{eqnarray*}
where we used the Cauchy-Schwarz inequality combined with the Young's inequality, and also assumption~$(\mathbf{A2})$ applied for $\bu_h^\pm$ and $\bv_{p_h}^\pm$. The generic constant $C$ does not depend on $h$, and is non-decreasing with respect to $\gamma_0$. Furthermore, using~\eqref{estvp} we obtain
\begin{eqnarray}
\mathcal{M}(\mathfrak{u},\mathfrak{v}^{(2)}_h) & \geq & \left(1 - \frac{C\gamma_0}{\alpha} \right)\left( \| p_h^+\|^2_{\L^2(\Omega^+)} +  \| p_h^-\|^2_{\L^2(\Omega^-)} \right) 
- \alpha C\gamma_0 \left(\| \bu_h^+ \|_{\VV^+}^2 + \| \bu_h^- \|_{\VV^-}^2 \right) \nonumber \\
& & -C\gamma_0 h\left( \| \blambda_h^+ + p_h^+\bn^+\|_{\mathbf{L}^2(\Gamma)}^2 +
\| \blambda_h^- + p_h^-\bn^-\|_{\mathbf{L}^2(\Gamma)}^2 
\right). \label{eststep2}
\end{eqnarray}
\hfill \\
{\bf Step 3.} 
Choose $\mathfrak{v}_h^{(3)} := (0, 0, \bmu_h^+, 0, 0, \bmu_h^-, 0)$ with $\bmu_h^{\pm} = -\frac{1}{h}\mathcal{P}_h(\bu_h^{\pm}-\Phi_h)$. The operator $\mathcal{P}_h$ is introduced in Lemma~\ref{lemma-projector}. Then
\begin{eqnarray*}
\mathcal{M}(\mathfrak{u}_h,\mathfrak{v}^{(3)}_h) & = & \frac{1}{h}\|\mathcal{P}_h(\bu_h^{+}-\Phi_h)\|^2_{\mathbf{L}^2(\Gamma)} 
+ \frac{1}{h}\|\mathcal{P}_h(\bu_h^{-}-\Phi_h)\|^2_{\mathbf{L}^2(\Gamma)}\\
& & - \gamma_0 \int_{\Gamma} (\sigma(\bu^+_h,p^+_h)\bn^+ -\blambda^+_h ) \cdot 
\mathcal{P}_h(\bu^+_h-\Phi_h)\, \d \Gamma
- \gamma_0 \int_{\Gamma} (\sigma(\bu^-_h,p^-_h)\bn^- -\blambda^-_h ) \cdot 
\mathcal{P}_h(\bu^-_h-\Phi_h)\, \d \Gamma \\
& \geq & \frac{1}{h}\|\bu_h^{+}-\Phi_h\|^2_{\mathbf{L}^2(\Gamma)}
+ \frac{1}{h}\|\bu_h^{-}-\Phi_h\|^2_{\mathbf{L}^2(\Gamma)} \\
& & - \frac{1}{h}\|\mathcal{P}_h(\bu_h^{+}-\Phi_h) - (\bu_h^+-\Phi_h)\|^2_{\mathbf{L}^2(\Gamma)}
- \frac{1}{h}\|\mathcal{P}_h(\bu_h^{-}-\Phi_h) - (\bu_h^--\Phi_h)\|^2_{\mathbf{L}^2(\Gamma)}\\
& & - \gamma_0 C \left(\|\varepsilon(\bu^+_h)\bn^+ \|_{\mathbf{L}^2(\Gamma)} 
+ \|\blambda^+_h + p^+_h\bn^+\|_{\mathbf{L}^2(\Gamma)} \right)
\|\mathcal{P}_h(\bu_h-\Phi_h)\|_{\mathbf{L}^2(\Gamma)}\\
& & - \gamma_0 C \left(\|\varepsilon(\bu^-_h)\bn^- \|_{\mathbf{L}^2(\Gamma)}
+ \|\blambda^-_h + p^-_h\bn^-\|_{\mathbf{L}^2(\Gamma)} \right)
\|\mathcal{P}_h(\bu^-_h-\Phi_h)\|_{\mathbf{L}^2(\Gamma)}.
\end{eqnarray*}
From~\eqref{estproj} and assumption~$(\mathbf{A2})$, this yields
\begin{eqnarray}
\mathcal{M}(\mathfrak{u}_h,\mathfrak{v}^{(3)}_h) & \geq &
\frac{1}{h}\|\bu_h^{+}-\Phi_h\|^2_{\mathbf{L}^2(\Gamma)}
+ \frac{1}{h}\|\bu_h^{-}-\Phi_h\|^2_{\mathbf{L}^2(\Gamma)}  -C\left(\|\bu_h^+-\Phi_h\|^2_{\mathbf{H}^{1/2}(\Gamma)} + \|\bu_h^--\Phi_h\|^2_{\mathbf{H}^{1/2}(\Gamma)} \right) \nonumber \\
& & - \gamma_0 C\|\bu^+_h \|_{\mathbf{V}^+} \|\bu^+_h-\Phi_h\|_{\mathbf{H}^{1/2}(\Gamma)}
- \gamma_0 C h^{1/2} \|\blambda^+_h + p^+_h\bn^+\|_{\mathbf{L}^2(\Gamma)} 
\|\bu^+_h-\Phi_h\|_{\mathbf{H}^{1/2}(\Gamma)} \nonumber \\
& & - \gamma_0 C\|\bu^-_h \|_{\mathbf{V}^-} \|\bu^-_h-\Phi_h\|_{\mathbf{H}^{1/2}(\Gamma)}
- \gamma_0 C h^{1/2} \|\blambda^-_h + p^-_h\bn^-\|_{\mathbf{L}^2(\Gamma)} 
\|\bu^-_h-\Phi_h\|_{\mathbf{H}^{1/2}(\Gamma)} \nonumber \\
& \geq &
\frac{1}{h}\|\bu_h^{+}-\Phi\|^2_{\mathbf{L}^2(\Gamma)}
+ \frac{1}{h}\|\bu_h^{-}-\Phi\|^2_{\mathbf{L}^2(\Gamma)}  -C\left(\|\bu_h^+\|_{\mathbf{V}^+} + \|\bu_h^+\|_{\mathbf{V}^-} 
+ \|\Phi_h\|^2_{\mathbf{Z}}  \right) \nonumber \\
& & - \gamma_0 C h\left(\|\blambda^+_h + p^+_h\bn^+\|^2_{\mathbf{L}^2(\Gamma)} +
\|\blambda^-_h + p^-_h\bn^-\|^2_{\mathbf{L}^2(\Gamma)}  \right) \label{eststep3}
\end{eqnarray}
for $\gamma_0$ small enough.

\hfill \\
{\bf Step 4.} Choose $\mathfrak{v}_h = \mathfrak{v}^{(1)}_h + \kappa_2 \mathfrak{v}_h^{(2)} + \kappa_3\mathfrak{v}_h^{(3)} $, for some positive constants $\kappa_2$ and $\kappa_3$. Then, gathering the estimates~\eqref{eststep1},~\eqref{eststep2} and~\eqref{eststep3} yields
\begin{eqnarray}
	\mathcal{M}(\mathfrak{u}_h,\mathfrak{v}_h) & \geq & 
	\left(2\underline{\nu} - C\alpha_0 - C\gamma_0 - \alpha \overline{\nu} \kappa_2 - \beta \overline{\nu} \kappa_3  \right)
	\left( \| \bu^+ \|_{\VV^+}^2 + \| \bu^- \|_{\VV^-}^2 \right) \nonumber \\
	& & +\kappa_2\left(1-\frac{C\gamma_0}{\alpha}\right)\left( \| p_h^+\|^2_{\L^2(\Omega^+)} +  \| p_h^-\|^2_{\L^2(\Omega^-)} \right) + \frac{\kappa_3}{h}\|\bu_h^{+}-\Phi_h\|^2_{\mathbf{L}^2(\Gamma)}
	+ \frac{\kappa_3}{h}\|\bu_h^{-}-\Phi_h\|^2_{\mathbf{L}^2(\Gamma)} \nonumber \\
	& & + (2\alpha_0 - C\kappa_3)\|\Phi_h\|^2_{\mathbf{L}^2(\Gamma)} +
	\gamma_0 h(1-C\kappa_2 - C\kappa_3)\left(\|\blambda^+_h + p^+_h\bn^+\|^2_{\mathbf{L}^2(\Gamma)} +
	\|\blambda^-_h + p^-_h\bn^-\|^2_{\mathbf{L}^2(\Gamma)}  \right)\nonumber \\
	& \geq & \left(2\underline{\nu} - C\alpha -C\gamma_0 - \alpha \overline{\nu} \kappa_2 - \beta \overline{\nu} \kappa_3  \right)\left( \| \bu_h^+ \|_{\VV^+}^2 +   \| \bu_h^- \|_{\VV^-}^2 \right) \nonumber \\
	& & +\frac{\kappa_2}{2}\left(1-\frac{C}{\alpha}\right)\left( \| p_h^+\|^2_{\L^2(\Omega^+)} +  \| p_h^-\|^2_{\L^2(\Omega^-)} \right) + \frac{\kappa_3}{h}\|\bu_h^{+}-\Phi_h\|^2_{\mathbf{L}^2(\Gamma)}
	+ \frac{\kappa_3}{h}\|\bu_h^{-}-\Phi_h\|^2_{\mathbf{L}^2(\Gamma)}\nonumber \\
	& & + (2\alpha_0-C\kappa_3)\|\Phi_h\|^2_{\mathbf{L}^2(\Gamma)} +
	\gamma_0 h(1-C\kappa_2 - C\kappa_3)\left(\|\blambda^+_h + p^+_h\bn^+\|^2_{\mathbf{L}^2(\Gamma)} +
	\|\blambda^-_h + p^-_h\bn^-\|^2_{\mathbf{L}^2(\Gamma)}  \right) \nonumber \\
	& & +\frac{\kappa_2}{2}\left(1-\frac{C}{\alpha}\right)\left( \| p_h^+\|^2_{\L^2(\Omega^+)} +  \| p_h^-\|^2_{\L^2(\Omega^-)} \right), \label{est-step5}
\end{eqnarray}
where we denote $\overline{\nu} = \max(\nu^+,\nu^-)$ and $\underline{\nu} = \min(\nu^+,\nu^-)$. For splitting the terms involving $\blambda_h^{\pm}$ and $p_h^\pm$, we estimate the two last quantities in the right-hand-side of the estimate above. First we estimate
\begin{eqnarray*}
\|\blambda^+_h + p^+_h\bn^+\|^2_{\mathbf{L}^2(\Gamma)} & = & 
\| \blambda_h^\pm\|^2_{\mathbf{L}^2(\Gamma)} + \|p_h^\pm\|^2_{\L^2(\Gamma)} +
2\langle p^\pm_h\bn^\pm ; \blambda_h^\pm \rangle_{\mathbf{L}^2(\Gamma)} \\
& \geq & \left(1-\frac{1}{\beta} \right)\| \blambda_h^\pm\|^2_{\mathbf{L}^2(\Gamma)} - 
(\beta-1)\|p_h^\pm\|^2_{\L^2(\Gamma)} ,
\end{eqnarray*}
where we used the Young's inequality with some constant $\beta >0$. Furthermore, using assumption~$(\mathbf{A3})$, we have
\begin{eqnarray*}
& & \frac{\kappa_2}{2}\left(1-\frac{C}{\alpha}\right)\| p_h^\pm\|^2_{\L^2(\Omega^\pm)} + 
\gamma_0 h(1-C\kappa_2 - C\kappa_3)\|\blambda^\pm_h + p^\pm_h\bn^\pm\|^2_{\mathbf{L}^2(\Gamma)} \\
 &  &  \geq \left(\frac{\kappa_2 }{2C}\left(1-\frac{C}{\alpha}\right) -\gamma_0(\beta-1)(1-C\kappa_2 - C\kappa_3) \right)h\| p_h^\pm\|^2_{\L^2(\Gamma)} \\ 
 & & \qquad + \gamma_0 h(1-C\kappa_2 - C\kappa_3)\left(1-\frac{1}{\beta}\right)\| \blambda_h^\pm\|^2_{\mathbf{L}^2(\Gamma)}.
\end{eqnarray*}
Thus, coming back to~\eqref{est-step5}, we obtain
\begin{eqnarray*}
	\mathcal{M}(\mathfrak{u}_h,\mathfrak{v}_h) 
	& \geq & \left(\underline{\nu} - C\alpha_0 - C\gamma_0 -\alpha \overline{\nu} \kappa_2 - \beta \overline{\nu} \kappa_3  \right)\left( \| \bu_h^+ \|_{\VV^+}^2 + \| \bu_h^- \|_{\VV^-}^2 \right) \\
	& & +\frac{\kappa_2}{2}\left(1-\frac{C}{\alpha}\right)\left( \| p_h^+\|^2_{\L^2(\Omega^+)} +  \| p_h^-\|^2_{\L^2(\Omega^-)} \right) \\
	& & \gamma_0 (1-C\kappa_2 - C\kappa_3)\left(1-\frac{1}{\beta}\right)
	\left(h\| \blambda_h^+ \|_{\mathbf{L}^2(\Gamma)}^2 + h\| \blambda_h^- \|_{\mathbf{L}^2(\Gamma)}^2 \right)\\
	& & +\left(\frac{\kappa_2 }{2C}\left(1-\frac{C}{\alpha}\right) -\gamma_0(\beta-1)(1-C\kappa_2 - C\kappa_3) \right)\left(h\| p_h^+\|^2_{\L^2(\Gamma)} + h\| p_h^-\|^2_{\L^2(\Gamma)} \right)\\
	& & + \frac{\kappa_3}{h}\left( \|\bu_h^{+}-\Phi_h\|^2_{\mathbf{L}^2(\Gamma)}
	+\|\bu_h^{-}-\Phi_h\|^2_{\mathbf{L}^2(\Gamma)}\right)
	+ (2\alpha_0-C\kappa_3)\|\Phi_h\|^2_{\mathbf{L}^2(\Gamma)} \\
	& & +	\underline{\nu} \left( \| \bu_h^+ \|_{\VV^+}^2 + \| \bu_h^- \|_{\VV^-}^2 \right) .
\end{eqnarray*}
The last terms in the right-hand-side above enable us to control the missing boundary terms as follows
\begin{eqnarray*}
	\underline{\nu} \left( \| \bu_h^+ \|_{\VV^+}^2 + \| \bu_h^- \|_{\VV^-}^2 \right) & \geq & 
	C\underline{\nu} h\left(\|\varepsilon(\bu^+_h)\bn^\pm\|_{\LL^2(\Gamma)} + \|\varepsilon(\bu^-_h)\bn^\pm\|_{\LL^2(\Gamma)}\right), 
\end{eqnarray*}
where we use assumption~$(\mathbf{A2})$. Then by choosing $\alpha$ and $\beta$ large enough, next $\alpha_0$ small enough, next $\kappa_2$ and $\kappa_3$ small enough, and next $\gamma_0$ small enough, we get
\begin{eqnarray*}
	\mathcal{M}(\mathfrak{u}_h,\mathfrak{v}_h) & \geq & C|||\, \mathfrak{u}_h\, |||^2,
\end{eqnarray*}
for some constant $C>0$. It remains us to verify that the norm of $\mathfrak{v}$ so chosen is controlled by the norm of $\mathfrak{u}$, namely the estimate
\begin{eqnarray*}
	|||\, \mathfrak{v}_h\, ||| & \leq & C |||\, \mathfrak{u}_h \, |||
\end{eqnarray*}
which holds from the estimate of~\eqref{estvp} and the continuity -- with constant equal to $1$ -- of the projection $\mathcal{P}_h$:
\begin{eqnarray*}
h\|\bmu_h^\pm\|_{\LL^2(\Gamma)}^2 \leq \frac{1}{h}\|\mathcal{P}_h(\bu^\pm_h - \Phi_h ) \|^2_{\LL^2(\Gamma)}  \leq \frac{1}{h}\|\bu^\pm_h - \Phi_h  \|^2_{\LL^2(\Gamma)} . 
\end{eqnarray*}
So we obtain
\begin{eqnarray*}
	\frac{\mathcal{M}(\mathfrak{u}_h,\mathfrak{v}_h)}{|||\, \mathfrak{v}_h \, |||} & \geq & 
	C|||\, \mathfrak{u}_h \, |||,
\end{eqnarray*}
and thus the proof is complete.
\end{proof}

\section*{Acknowledgments}
The author gratefully acknowledges support by the Austrian Science Fund (FWF) special research grant SFB-F32 "Mathematical Optimization and Applications in Biomedical Sciences" and the BioTechMed-Graz.\\
Our codes are written in C++ language, using the free finite element library {\sc GETFEM++} \cite{Getfem} developed by Yves Renard and Julien Pommier. The author thanks the {\sc GETFEM++} users community for collaborative efforts.

\nocite*
\bibliographystyle{alpha}
\bibliography{Court_JCAM_final}

\newcommand{\etalchar}[1]{$^{#1}$}
\begin{thebibliography}{DvdHO{\etalchar{+}}14}

\bibitem[ABMT14]{Anjos2014}
G.~R. Anjos, N.~Borhani, N.~Mangiavacchi, and J.~R. Thome.
\newblock A 3{D} moving mesh finite element method for two-phase flows.
\newblock {\em J. Comput. Phys.}, 270:366--377, 2014.

\bibitem[ADKL01]{MUMPS1}
P.~R. Amestoy, I.~S. Duff, J.~Koster, and J.-Y. L'Excellent.
\newblock A fully asynchronous multifrontal solver using distributed dynamic
  scheduling.
\newblock {\em SIAM Journal on Matrix Analysis and Applications}, 23(1):15--41,
  2001.

\bibitem[AG18]{Abels2016}
H.~Abels and H.~Garcke.
\newblock {\em Weak solutions and diffuse interface models for incompressible
  two-phase flows}.
\newblock Handbook of Mathematical Analysis in Mechanics of Viscous Fluid.
  Springer International Publishing, 2018.

\bibitem[AGLP06]{MUMPS2}
P.~R. Amestoy, A.~Guermouche, J.-Y. L'Excellent, and S.~Pralet.
\newblock Hybrid scheduling for the parallel solution of linear systems.
\newblock {\em Parallel Computing}, 32(2):136--156, 2006.

\bibitem[BCCK16]{Court2016}
O.~Bodart, V.~Cayol, S.~Court, and J.~Koko.
\newblock X{FEM}-based fictitious domain method for linear elasticity model
  with crack.
\newblock {\em SIAM J. Sci. Comput.}, 38(2):B219--B246, 2016.

\bibitem[BDH96]{Qhull}
C.~B. Barber, D.~P. Dobkin, and H.~Huhdanpaa.
\newblock The quickhull algorithm for convex hulls.
\newblock {\em ACM Trans. Math. Software}, 22(4):469--483, 1996.

\bibitem[BH91]{Barbosa1}
H.~J.~C. Barbosa and T.~J.~R. Hughes.
\newblock The finite element method with {L}agrange multipliers on the
  boundary: circumventing the {B}abu\v ska-{B}rezzi condition.
\newblock {\em Comput. Methods Appl. Mech. Engrg.}, 85(1):109--128, 1991.

\bibitem[BH92]{Barbosa2}
H.~J.~C. Barbosa and T.~J.~R. Hughes.
\newblock Boundary {L}agrange multipliers in finite element methods: error
  analysis in natural norms.
\newblock {\em Numer. Math.}, 62(1):1--15, 1992.

\bibitem[BH10]{Hansbo2010}
E.~Burman and P.~Hansbo.
\newblock Interior-penalty-stabilized {L}agrange multiplier methods for the
  finite-element solution of elliptic interface problems.
\newblock {\em IMA J. Numer. Anal.}, 30(3):870--885, 2010.

\bibitem[BHLM17]{Massing2017}
Erik Burman, Peter Hansbo, Mats~G. Larson, and Andr\'e Massing.
\newblock A cut discontinuous {G}alerkin method for the {L}aplace-{B}eltrami
  operator.
\newblock {\em IMA J. Numer. Anal.}, 37(1):138--169, 2017.

\bibitem[CB03a]{Chessa2003}
J.~Chessa and T.~Belytschko.
\newblock An enriched finite element method and level sets for axisymmetric
  two-phase flow with surface tension.
\newblock {\em Internat. J. Numer. Methods Engrg.}, 58(13):2041--2064, 2003.

\bibitem[CB03b]{Chessa2003-2}
J.~Chessa and T.~Belytschko.
\newblock An extended finite element method for two-phase fluids.
\newblock {\em Trans. ASME J. Appl. Mech.}, 70(1):10--17, 2003.

\bibitem[CCCY12]{Choi2012}
M.~H. Cho, H.~G. Choi, S.~H. Choi, and J.~Y. Yoo.
\newblock A {Q}2{Q}1 finite element/level-set method for simulating two-phase
  flows with surface tension.
\newblock {\em Internat. J. Numer. Methods Fluids}, 70(4):468--492, 2012.

\bibitem[CCCY16]{Choi2016-2}
S.~Choi, M.~H. Cho, H.~G. Choi, and J.~Y. Yoo.
\newblock A {Q}2{Q}1 integrated finite element method with the semi-implicit
  consistent {CSF} for solving incompressible two-phase flows with surface
  tension effect.
\newblock {\em Internat. J. Numer. Methods Fluids}, 81(5):284--308, 2016.

\bibitem[CELR11]{Case2011}
Michael~A. Case, Vincent~J. Ervin, Alexander Linke, and Leo~G. Rebholz.
\newblock A connection between {S}cott-{V}ogelius and grad-div stabilized
  {T}aylor-{H}ood {FE} approximations of the {N}avier-{S}tokes equations.
\newblock {\em SIAM J. Numer. Anal.}, 49(4):1461--1481, 2011.

\bibitem[CF12]{Cheng2012}
K.-W. Cheng and T.-P. Fries.
\newblock X{FEM} with hanging nodes for two-phase incompressible flow.
\newblock {\em Comput. Methods Appl. Mech. Engrg.}, 245/246:290--312, 2012.

\bibitem[CF15]{Court2015}
S.~Court and M.~Fourni\'e.
\newblock A fictitious domain finite element method for simulations of
  fluid–structure interactions: The navier–stokes equations coupled with a
  moving solid.
\newblock {\em Journal of Fluids and Structures}, 55:398--408, 2015.

\bibitem[CFL14]{Court2014}
S.~Court, M.~Fourni{\'e}, and A.~Lozinski.
\newblock A fictitious domain approach for the {S}tokes problem based on the
  extended finite element method.
\newblock {\em Internat. J. Numer. Methods Fluids}, 74(2):73--99, 2014.

\bibitem[CH11]{Hansbo2011}
A.~Chernov and P.~Hansbo.
\newblock An {$hp$}-{N}itsche's method for interface problems with
  nonconforming unstructured finite element meshes.
\newblock In {\em Spectral and high order methods for partial differential
  equations}, volume~76 of {\em Lect. Notes Comput. Sci. Eng.}, pages 153--161.
  Springer, Heidelberg, 2011.

\bibitem[DCRD18]{Duret2018}
B.~Duret, R.~Canu, J.~Reveillon, and F.~X. Demoulin.
\newblock A pressure based method for vaporizing compressible two-phase flows
  with interface capturing approach.
\newblock {\em Int. J. Multiph. Flow}, 108:42--50, 2018.

\bibitem[DHB{\etalchar{+}}07]{Devals2007}
C.~Devals, M.~Heniche, F.~Bertrand, P.~A. Tanguy, and R.~E. Hayes.
\newblock A two-phase flow interface capturing finite element method.
\newblock {\em Internat. J. Numer. Methods Fluids}, 53(5):735--751, 2007.

\bibitem[DP98]{Dufour1998}
S.~Dufour and D.~Pelletier.
\newblock An adaptive finite element method for multiphase flows with surface
  tension.
\newblock In {\em Computational mechanics ({B}uenos {A}ires, 1998)}, pages
  CD--ROM file. Centro Internac. M\'etodos Num\'er. Ing., Barcelona, 1998.

\bibitem[DSHT15]{Dhar2015}
Abhinav Dhar, Naoki Shimada, Kosuke Hayashi, and Akio Tomiyama.
\newblock Assessment of numerical treatments in interface capturing simulations
  for surface-tension-driven interface motion.
\newblock {\em J. Comput. Multiph. Flows}, 7(1):15--32, 2015.

\bibitem[DvdHO{\etalchar{+}}14]{Denner2014}
Fabian Denner, Duncan~R. van~der Heul, Guido~T. Oud, Millena~M. Villar, Aristeu
  da~Silveira~Neto, and Berend G.~M. van Wachem.
\newblock Comparative study of mass-conserving interface capturing frameworks
  for two-phase flows with surface tension.
\newblock {\em Int. J. Multiph. Flow}, 61:37--47, 2014.

\bibitem[EG04]{Ern}
A.~Ern and J.-L. Guermond.
\newblock {\em Theory and practice of finite elements}, volume 159 of {\em
  Applied Mathematical Sciences}.
\newblock Springer-Verlag, New York, 2004.

\bibitem[ET02]{Engquist2002}
Bj\"{o}rn Engquist and Anna-Karin Tornberg.
\newblock A finite element based level-set method for multiphase flows.
\newblock In {\em Innovative methods for numerical solutions of partial
  differential equations ({A}rcachon, 1998)}, pages 86--110. World Sci. Publ.,
  River Edge, NJ, 2002.

\bibitem[FB10]{reviewXfem}
T.-P. Fries and T.~Belytschko.
\newblock The extended/generalized finite element method: an overview of the
  method and its applications.
\newblock {\em Internat. J. Numer. Methods Engrg.}, 84(3):253--304, 2010.

\bibitem[FS17]{Fahsi2017}
Adil Fahsi and Azzeddine Soula\"{i}mani.
\newblock Numerical investigations of the {XFEM} for solving two-phase
  incompressible flows.
\newblock {\em Int. J. Comput. Fluid Dyn.}, 31(3):135--155, 2017.

\bibitem[GAT18]{Anjos2018}
E.~Gros, G.~R. Anjos, and J.~R. Thome.
\newblock Interface-fitted moving mesh method for axisymmetric two-phase flow
  in microchannels.
\newblock {\em Internat. J. Numer. Methods Fluids}, 86(3):201--217, 2018.

\bibitem[GG95]{SPP1}
V.~Girault and R.~Glowinski.
\newblock Error analysis of a fictitious domain method applied to a {D}irichlet
  problem.
\newblock {\em Japan J. Indust. Appl. Math.}, 12(3):487--514, 1995.

\bibitem[GGRG15]{Gonzalez2015}
F.~Guill\'{e}n-Gonz\'{a}lez and J.~R. Rodr\'{i}guez-Galv\'{a}n.
\newblock Analysis of the hydrostatic {S}tokes problem and finite-element
  approximation in unstructured meshes.
\newblock {\em Numer. Math.}, 130(2):225--256, 2015.

\bibitem[GR07a]{Gross2007}
S.~Gro\ss and A.~Reusken.
\newblock An extended pressure finite element space for two-phase
  incompressible flows with surface tension.
\newblock {\em J. Comput. Phys.}, 224(1):40--58, 2007.

\bibitem[GR07b]{Reusken2007}
S.~Gro\ss and A.~Reusken.
\newblock An extended pressure finite element space for two-phase
  incompressible flows with surface tension.
\newblock {\em J. Comput. Phys.}, 224(1):40--58, 2007.

\bibitem[GR07c]{Gross2007-2}
S.~Gro\ss and A.~Reusken.
\newblock Finite element discretization error analysis of a surface tension
  force in two-phase incompressible flows.
\newblock {\em SIAM J. Numer. Anal.}, 45(4):1679--1700, 2007.

\bibitem[Gro11]{Gross2011}
Sven Gross.
\newblock Pressure {XFEM} for two-phase incompressible flows with application
  to 3{D} droplet problems.
\newblock In {\em Meshfree methods for partial differential equations {V}},
  volume~79 of {\em Lect. Notes Comput. Sci. Eng.}, pages 81--87. Springer,
  Heidelberg, 2011.

\bibitem[Han05]{Hansbo2005-3}
P.~Hansbo.
\newblock Nitsche's method for interface problems in computational mechanics.
\newblock {\em GAMM-Mitt.}, 28(2):183--206, 2005.

\bibitem[HH02]{Hansbo2002}
A.~Hansbo and P.~Hansbo.
\newblock An unfitted finite element method, based on {N}itsche's method, for
  elliptic interface problems.
\newblock {\em Comput. Methods Appl. Mech. Engrg.}, 191(47-48):5537--5552,
  2002.

\bibitem[HLPS05a]{Hansbo2005-2}
P.~Hansbo, C.~Lovadina, I.~Perugia, and G.~Sangalli.
\newblock A {L}agrange multiplier method for elliptic interface problems using
  non-matching meshes.
\newblock In {\em Applied and industrial mathematics in {I}taly}, volume~69 of
  {\em Ser. Adv. Math. Appl. Sci.}, pages 360--370. World Sci. Publ.,
  Hackensack, NJ, 2005.

\bibitem[HLPS05b]{Hansbo2005}
P.~Hansbo, C.~Lovadina, I.~Perugia, and G.~Sangalli.
\newblock A {L}agrange multiplier method for the finite element solution of
  elliptic interface problems using non-matching meshes.
\newblock {\em Numer. Math.}, 100(1):91--115, 2005.

\bibitem[HLZ14]{Hansbo2014}
P.~Hansbo, M.~G. Larson, and S.~Zahedi.
\newblock A cut finite element method for a {S}tokes interface problem.
\newblock {\em Appl. Numer. Math.}, 85:90--114, 2014.

\bibitem[HR09]{Renard2009}
J.~Haslinger and Y.~Renard.
\newblock A new fictitious domain approach inspired by the extended finite
  element method.
\newblock {\em SIAM J. Numer. Anal.}, 47(2):1474--1499, 2009.

\bibitem[HTMP17]{Turek2017}
Babak~S. Hosseini, Stefan Turek, Matthias M\"{o}ller, and Christian Palmes.
\newblock Isogeometric analysis of the {N}avier-{S}tokes-{C}ahn-{H}illiard
  equations with application to incompressible two-phase flows.
\newblock {\em J. Comput. Phys.}, 348:171--194, 2017.

\bibitem[JZ16]{Zhang2016}
H.~Ji and Q.~Zhang.
\newblock A simple finite element method for {S}tokes flows with surface
  tension using unfitted meshes.
\newblock {\em Internat. J. Numer. Methods Fluids}, 81(2):87--103, 2016.

\bibitem[KMV{\etalchar{+}}16]{Mesri2016}
M.~Khalloufi, Y.~Mesri, R.~Valette, E.~Massoni, and E.~Hachem.
\newblock High fidelity anisotropic adaptive variational multiscale method for
  multiphase flows with surface tension.
\newblock {\em Comput. Methods Appl. Mech. Engrg.}, 307:44--67, 2016.

\bibitem[KS14]{Kou2014}
J.~Kou and S.~Sun.
\newblock An adaptive finite element method for simulating surface tension with
  the gradient theory of fluid interfaces.
\newblock {\em J. Comput. Appl. Math.}, 255:593--604, 2014.

\bibitem[KT04]{Turek2004}
Dmitri Kuzmin and Stefan Turek.
\newblock Finite element discretization and iterative solution techniques for
  multiphase flows in gas-liquid reactors.
\newblock In {\em Conjugate gradient algorithms and finite element methods},
  Sci. Comput., pages 297--324. Springer, Berlin, 2004.

\bibitem[LKG{\etalchar{+}}17]{Liu2017}
W.~Liu, A.~Koniges, K.~Gott, D.~Eder, J.~Barnard, A.~Friedman, N.~Masters, and
  A.~Fisher.
\newblock Surface tension models for a multi-material {ALE} code with {AMR}.
\newblock {\em Comput. \& Fluids}, 151:91--101, 2017.

\bibitem[LS17]{Laadhari2017}
A.~Laadhari and G.~Sz\'ekely.
\newblock Fully implicit finite element method for the modeling of free surface
  flows with surface tension effect.
\newblock {\em Internat. J. Numer. Methods Engrg.}, 111(11):1047--1074, 2017.

\bibitem[LZ12]{Liao2012}
Jian-Hui Liao and Zhuo Zhuang.
\newblock A consistent projection-based {SUPG}/{PSPG} {XFEM} for incompressible
  two-phase flows.
\newblock {\em Acta Mech. Sin.}, 28(5):1309--1322, 2012.

\bibitem[MDB99]{Moes1999}
N.~Mo{\"e}s, J.~Dolbow, and T.~Belytschko.
\newblock A finite element method for crack growth without remeshing.
\newblock {\em Internat. J. Numer. Methods Engrg.}, 46(1):131--150, 1999.

\bibitem[MF16]{Moortgat2016}
Joachim Moortgat and Abbas Firoozabadi.
\newblock Mixed-hybrid and vertex-discontinuous-{G}alerkin finite element
  modeling of multiphase compositional flow on 3{D} unstructured grids.
\newblock {\em J. Comput. Phys.}, 315:476--500, 2016.

\bibitem[MWnH{\etalchar{+}}18]{Heinrich2018}
Mehdi Mostafaiyan, Sven Wie\ss~ner, Gert Heinrich, Mahdi Salami~Hosseini, Jan
  Domurath, and Hossein~Ali Khonakdar.
\newblock Application of local least squares finite element method ({LLSFEM})
  in the interface capturing of two-phase flow systems.
\newblock {\em Comput. \& Fluids}, 174:110--121, 2018.

\bibitem[NRTL97]{Navti1997}
S.E. Navti, K.~Ravindran, C.~Taylor, and R.W. Lewis.
\newblock Finite element modelling of surface tension effects using a
  lagrangian-eulerian kinematic description.
\newblock {\em Computer Methods in Applied Mechanics and Engineering},
  147(1):41--60, 1997.

\bibitem[NTP18]{Park2018}
Van-Tu Nguyen, Van-Dat Thang, and Warn-Gyu Park.
\newblock A novel sharp interface capturing method for two- and three-phase
  incompressible flows.
\newblock {\em Comput. \& Fluids}, 172:147--161, 2018.

\bibitem[OD13]{Owkes2013}
Mark Owkes and Olivier Desjardins.
\newblock A discontinuous {G}alerkin conservative level set scheme for
  interface capturing in multiphase flows.
\newblock {\em J. Comput. Phys.}, 249:275--302, 2013.

\bibitem[OK98]{Ohmori1997}
Katsushi Ohmori and Hideo Kawarada.
\newblock A sharp interface capturing technique in the finite element
  approximation for two-fluid flows.
\newblock In {\em Navier-{S}tokes equations: theory and numerical methods
  ({V}arenna, 1997)}, volume 388 of {\em Pitman Res. Notes Math. Ser.}, pages
  310--321. Longman, Harlow, 1998.

\bibitem[OTD18]{Turek2018}
A.~Ouazzi, S.~Turek, and H.~Damanik.
\newblock A curvature-free multiphase flow solver via surface stress-based
  formulation.
\newblock {\em Internat. J. Numer. Methods Fluids}, 88(1):18--31, 2018.

\bibitem[PS01]{Peric2001}
D.~Peri\'c and P.~H. Saksono.
\newblock On finite element modelling of surface tension: variational
  formulation and applications.
\newblock In {\em Trends in computational structural mechanics}, pages
  731--740. Internat. Center Numer. Methods Eng. (CIMNE), Barcelona, 2001.

\bibitem[Reu08]{Reusken2008}
A.~Reusken.
\newblock Analysis of an extended pressure finite element space for two-phase
  incompressible flows.
\newblock {\em Comput. Vis. Sci.}, 11(4-6):293--305, 2008.

\bibitem[RP]{Getfem}
Y.~Renard and J.~Pommier.
\newblock {\em An open source generic C++ library for finite element methods}.

\bibitem[SF11]{Sauerland2011}
Henning Sauerland and Thomas-Peter Fries.
\newblock The extended finite element method for two-phase and free-surface
  flows: a systematic study.
\newblock {\em J. Comput. Phys.}, 230(9):3369--3390, 2011.

\bibitem[SF13]{Sauerland2013}
Henning Sauerland and Thomas-Peter Fries.
\newblock The stable {XFEM} for two-phase flows.
\newblock {\em Comput. \& Fluids}, 87:41--49, 2013.

\bibitem[SKS98]{Sato1998}
M.~Sato, N.~Kudo, and N.~Saito.
\newblock Surface tension reduction of liquid by applied electric field using
  vibrating jet method.
\newblock {\em IEEE Transactions on Industry Applications}, 34(2):294--300, Mar
  1998.

\bibitem[Ste95]{Stenberg}
R.~Stenberg.
\newblock On some techniques for approximating boundary conditions in the
  finite element method.
\newblock {\em J. Comput. Appl. Math.}, 63(1-3):139--148, 1995.
\newblock International Symposium on Mathematical Modelling and Computational
  Methods Modelling 94 (Prague, 1994).

\bibitem[Tab07]{Tabata2007}
M.~Tabata.
\newblock Finite element schemes based on energy-stable approximation for
  two-fluid flow problems with surface tension.
\newblock {\em Hokkaido Math. J.}, 36(4):875--890, 2007.

\bibitem[Tem83]{Temam}
R.~Temam.
\newblock {\em Probl\`emes math\'ematiques en plasticit\'e}, volume~12 of {\em
  M\'ethodes Math\'ematiques de l'Informatique [Mathematical Methods for
  Information Science]}.
\newblock Gauthier-Villars, Montrouge, 1983.

\bibitem[Tez03]{Tezduhar2003}
Tayfun~E. Tezduyar.
\newblock Stabilized finite element formulations and interface-tracking and
  interface-capturing techniques for incompressible flows.
\newblock In {\em Numerical simulations of incompressible flows ({H}alf {M}oon
  {B}ay, {CA}, 2001)}, pages 221--239. World Sci. Publ., River Edge, NJ, 2003.

\bibitem[TT00]{Tabata2000}
M.~Tabata and D.~Tagami.
\newblock A finite element analysis of a linearized problem of the
  {N}avier-{S}tokes equations with surface tension.
\newblock {\em SIAM J. Numer. Anal.}, 38(1):40--57, 2000.

\bibitem[Whi15]{Whiteley2015}
J.~P. Whiteley.
\newblock A discontinuous {G}alerkin finite element method for multiphase
  viscous flow.
\newblock {\em SIAM J. Sci. Comput.}, 37(4):B591--B612, 2015.

\bibitem[XPS{\etalchar{+}}16]{Xie2016}
Z.~Xie, D.~Pavlidis, P.~Salinas, J.~R. Percival, C.~C. Pain, and O.~K. Matar.
\newblock A balanced-force control volume finite element method for interfacial
  flows with surface tension using adaptive anisotropic unstructured meshes.
\newblock {\em Comput. \& Fluids}, 138:38--50, 2016.

\end{thebibliography}

\end{document}